\documentclass[12pt,letter]{article}
\usepackage[nottoc]{tocbibind}
\usepackage[latin1]{inputenc}
\usepackage{fancyhdr}
\usepackage{verbatim}
\usepackage{indentfirst}
\usepackage{graphicx}
\usepackage{epstopdf}
\usepackage{color}
\usepackage{newlfont}
\usepackage{amssymb}
\usepackage[intlimits]{amsmath}
\usepackage{latexsym}
\usepackage{amsthm}
\usepackage{amscd}
\usepackage[dvips, outer=1in, top=1in, bottom=1in]{geometry}
\usepackage{multirow}
\usepackage{hyperref}
\usepackage{rotating}

\theoremstyle{plain}                    
\newtheorem{theorem}{Theorem}[section]
\newtheorem*{theorem*}{Theorem}
\newtheorem{lem}[theorem]{Lemma}    
\newtheorem{prop}[theorem]{Proposition}
\newtheorem{cor}[theorem]{Corollary}
\theoremstyle{definition}

\newtheorem{remark}{Remark}[section]   

\newcommand{\N}{\mathbb{N}}
\newcommand{\Z}{\mathbb{Z}}
\newcommand{\Q}{\mathbb{Q}}
\newcommand{\R}{\mathbb{R}}
\newcommand{\C}{\mathbb{C}}

\newcommand{\bianco}{\textcolor{white}{.}}

\newcommand{\ha}{\frac{1}{2}}
\newcommand{\tha}{\tfrac{1}{2}}
\newcommand{\qua}{\frac{1}{4}}

\newcommand{\be}{\begin{equation}}
\newcommand{\bey}{\begin{eqnarray}}
\newcommand{\ee}{\end{equation}}
\newcommand{\eey}{\end{eqnarray}}
\newcommand{\ba}{\begin{array}}
\newcommand{\ea}{\end{array}}
\newcommand{\de}{\mathrm{d}}

\newcommand{\h}{\frak{H}}
\newcommand{\sltr}{\mathrm{SL}(2,\R)}
\newcommand{\tsltr}{\widetilde{\mathrm{SL}}(2,\R)}

\newcommand{\psltr}{\mathrm{PSL}(2,\R)}
\newcommand{\sltz}{\mathrm{SL}(2,\Z)}
\newcommand{\Hei}{\mathbb{H}(\mathbb{R})}
\newcommand{\ve}[2]{\left(\ba{c}\!#1\!\\ \!#2\!\ea\right)}
\newcommand{\sve}[2]{\left(\begin{smallmatrix}\!#1\!\\ \!#2\!\end{smallmatrix}\right)}
\newcommand{\veH}[3]{\left(\!\begin{pmatrix}#1\\ #2\end{pmatrix},#3\right)}

\newcommand{\sgn}{\mathrm{sgn}}

\newcommand{\bm}[1]{\mbox{\boldmath{$#1$}}}

\newcommand{\LtR}{\mathrm L^2(\R)}
\newcommand{\SR}{\mathcal{S}(\R)}
\newcommand{\ma}[4]{
\begin{pmatrix}#1&#2\\#3&#4\end{pmatrix}
}
\newcommand{\sma}[4]{\left(\begin{smallmatrix} #1&#2\\#3&#4\end{smallmatrix}\right)}
\newcommand{\e}[1]{e\!\left(#1\right)}
\newcommand{\ti}{\to\infty}

\def\CC{{\mathbb C}}

\def\HH{{\mathbb H}}
\def\NN{{\mathbb N}}
\def\QQ{{\mathbb Q}}
\def\RR{{\mathbb R}}

\def\ZZ{{\mathbb Z}}

\def\vecxi{{\text{\boldmath$\xi$}}}

\def\vecnull{{\text{\boldmath$0$}}}

\def\scrE{{\mathcal E}}
\def\scrF{{\mathcal F}}

\def\scrK{{\mathcal K}}

\def\scrQ{{\mathcal Q}}

\def\scrS{{\mathcal S}}

\def\Re{\operatorname{Re}}
\def\Im{\operatorname{Im}}

\def\e{\mathrm{e}}
\def\i{\mathrm{i}}

\def\dist{\operatorname{dist}}

\def\SL{\operatorname{SL}}
\def\ASL{\operatorname{ASL}}

\def\PSL{\operatorname{PSL}}

\def\sgn{\operatorname{sgn}}

\def\GamG{\Gamma\backslash G}

\def\gene{\gamma}

\def\Onder#1#2#3#4#5{#1 \setbox0=\hbox{$#1$}\setbox1=\hbox{$#2$}
       \dimen0=.5\wd0 \dimen1=\dimen0 \dimen2=\dp0 \dimen3=\dimen2
       \advance\dimen0 by .5\wd1 \advance\dimen0 by -#4
       \advance\dimen1 by -.5\wd1 \advance\dimen1 by -#4
       \advance\dimen2 by -#3 \advance\dimen2 by \ht1
       \advance\dimen2 by 0.3ex \advance\dimen3 by #5
        \kern-\dimen0\raisebox{-\dimen2}[0ex][\dimen3]{\box1}
       \kern\dimen1}

 \numberwithin{equation}{section}

\input xy
\xyoption{all}
\begin{document}
\clearpage{\pagestyle{empty}\cleardoublepage}

\title{Quadratic Weyl Sums, Automorphic Functions,\\ and Invariance Principles}
\author{Francesco Cellarosi\footnote{University of Illinois Urbana-Champaign, 1409 W Green Street, Urbana, IL 61801, U.S.A., \texttt{fcellaro@illinois.edu}}, Jens Marklof\footnote{School of Mathematics, University of Bristol, Bristol BS8 1TW, U.K., \texttt{j.marklof@bristol.ac.uk}}}


\date{}

\maketitle

\begin{center}
\today
\end{center}

\begin{abstract}
Hardy and Littlewood's approximate functional equation for quadratic Weyl sums (theta sums) provides, by iterative application, a powerful tool for the asymptotic analysis of such sums. The classical Jacobi theta function, on the other hand, satisfies an exact functional equation, and extends to an automorphic function on the Jacobi group. In the present study we construct a related, almost everywhere non-differentiable automorphic function, which approximates quadratic Weyl sums up to an error of order one, uniformly in the summation range. This not only implies the approximate functional equation, but allows us to replace Hardy and Littlewood's renormalization approach by the dynamics of a certain homogeneous flow. The great advantage of this construction is that the approximation is global, i.e., there is no need to keep track of the error terms accumulating in an iterative procedure. Our main application is a new functional limit theorem, or  invariance principle, for theta sums. The interesting observation here is that the paths of the limiting process share a number of key features with Brownian motion (scale invariance, invariance under time inversion, non-differentiability), although time increments are not independent and the value distribution at each fixed time is distinctly different from a normal distribution.
\end{abstract}

\newpage
\tableofcontents

\newpage

\section{Introduction}\label{section:introduction}

In their classic 1914 paper \cite{Hardy-Littlewood1914partII}, Hardy and Littlewood investigate exponential sums of the form
\be
S_N(x,\alpha)=\sum_{n=1}^N e \left( \tha n^2x+n\alpha\right),\label{theta-sum-intro0} 
\ee
where $N$ is a positive integer, $x$ and $\alpha$ are real, and $e(x):=\e^{2\pi\i x}$.
In today's literature these sums are commonly refered to as {\em quadratic Weyl sums}, {\em finite theta series} or {\em theta sums}.
Hardy and Littlewood estimate the size of $|S_N(x,\alpha)|$ in terms of the continued fraction expansion of $x$. At the heart of their argument is the approximate functional equation, valid for $0<x<2$, $0\leq \alpha\leq 1$,
\be\label{ApproxFeq}
S_N(x,\alpha) = \sqrt{\frac{\i}{x}}\; e\left(-\frac{\alpha^2}{2x}\right) S_{\lfloor x N \rfloor}\bigg(-\frac1x,\frac{\alpha}{x}\bigg) + O\left(\frac{1}{\sqrt x}\right),
\ee
stated here in the slightly more general form due to Mordell \cite{Mordell1926}.
This reduces the length of the sum from $N$ to the smaller $N'=\lfloor x N \rfloor$, the integer part of $xN$ (note that we may always assume that $0<x\leq 1$, replacing $S_N(x,\alpha)$ with its complex conjugate if necessary).
Asymptotic expansions of $S_N(x,\alpha)$ are thus obtained by iterating \eqref{ApproxFeq}, where after each application the new $x'$ is $-1/x\bmod 2$. The challenge in this renormalization approach is to keep track of the error terms that accummulate after each step, cf.~Berry and Goldberg \cite{Berry-Goldberg-1988}, Coutsias and Kazarinoff \cite{Coutsias-Kazarinoff-1998} and Fedotov and Klopp \cite{Fedotov-Klopp-2012}. The best asymptotic expansion of $S_N(x,\alpha)$ we are aware of is due to Fiedler, Jurkat and K\"orner  \cite{Fiedler-Jurkat-Korner77}, who avoid \eqref{ApproxFeq} and the above inductive argument by directly estimating $S_N(x,\alpha)$ for $x$ near a rational point.

Hardy and Littlewood motivate \eqref{ApproxFeq} by the {\em exact} functional equation for Jacobi's elliptic theta functions
\begin{equation}\label{Jacobi-theta}
\vartheta(z,\alpha)= \sum_{n\in\ZZ} e \left( \tha n^2 z +n\alpha\right), 
\end{equation}
where $z$ is in the complex upper half-plane $\frak H=\{z\in\CC:\Im z>0\}$, and $\alpha\in\CC$. In this case \be\label{ExactFeq}
\vartheta(z,\alpha) = \sqrt{\frac{\i}{z}}\; e\left(-\frac{\alpha^2}{2z}\right) \vartheta\bigg(-\frac1z,\frac{\alpha}{z}\bigg) .
\ee
The theta function $\vartheta(z,\alpha)$ is a Jacobi form of half-integral weight, and can thus be identified with an automorphic function on the Jacobi group $G$ which is invariant under a certain discrete subgroup $\Gamma$, the {\em theta group}. (Formula \eqref{ExactFeq} corresponds to one of the generators of $\Gamma$.) In the present study, we develop a unified geometric approach to both functional equations, exact and approximate. The plan is to construct an automorphic function $\Theta:\GamG\to\CC$ that yields $S_N(x,\alpha)$ for all $x$ and $\alpha$, up to a uniformly bounded error. This in turn enables us not only to re-derive \eqref{ApproxFeq}, but to furthermore obtain an asymptotic expansion without the need for an inductive argument. The value of $S_N(x,\alpha)$ for large $N$ is simply obtained by evaluating $\Theta$ along an orbit of a certain homogeneous flow at large times. (This flow is an extension of the geodesic flow on the modular surface.) As an application of our geometric approach we present a new functional limit theorem, or invariance principle, for $S_N(x,\alpha)$ for random $x$.

To explain the principal ideas and results of our investigation, define the generalized theta sum
\be\label{theta-sum-intro01} 
S_N(x,\alpha;f)=\sum_{n\in\ZZ} f\!\left(\frac{n}{N}\right) e\!\left( \tha n^2x+n\alpha\right),
\ee
where $f:\RR\to\RR$ is bounded and of sufficient decay at $\pm\infty$ so that \eqref{theta-sum-intro01} is absolutely convergent.
Thus $S_N(x,\alpha)=S_N(x,\alpha;f)$ if $f$
is the indicator function of $(0,1]$, and $\vartheta(z,\alpha)=S_N(x,\alpha;f)$ if $f(t)=\e^{-\pi t^2}$ and $y=N^{-2}$. (We assume here, for the sake of argument, that $\alpha$ is real. Complex $\alpha$ can also be used, but lead to a shift in the argument of $f$ by the imaginary part of $\alpha$, cf.~Section \ref{section:jacobi-theta-sums}.)

A key role in our analysis is played by the one- resp.\ two-parameter subgroups $\{\Phi^s:s\in\RR\}<G$ and $H_+=\{n_+(x,\alpha): (x,\alpha)\in\RR^2\} < G$. The dynamical interpretation of $H_+$ under the action of $\Phi^s$ ($s>0$) is that of an unstable horospherical subgroup, since (as we will show)
\begin{equation}
H_+=\{ g\in G :  \Phi^s g \Phi^{-s} \to  e \text{ for } s\to\infty \} .
\end{equation}
(Here $e\in G$ denotes the identity element.) The corresponding stable horospherical subgroup is defined by
\begin{equation}
H_-=\{ g\in G :  \Phi^{-s} g \Phi^s \to  e \text{ for }  s\to\infty \}.
\end{equation}
There is a completely explicit description of these groups, which we will defer to later sections.

The following two theorems describe the connection between theta sums and automorphic functions on $\GamG$. The proof of Theorem \ref{thm:1} (for smooth cut-off functions $f$) follows the strategy of \cite{Marklof-1999}. Theorem \ref{thm:2} below extends this to non-smooth cut-offs by a geometric regularization, and is the first main result of this paper. 

\begin{theorem}\label{thm:1}
Let $f:\RR\to\RR$ be of Schwartz class. Then there is a square-integrable, infinitely differentiable  function $\Theta_f:\GamG\to\CC$ and a continuous function $E_f:H_-\to [0,\infty)$ with $E_f(e)=0$, such that for all $s\in[0,\infty)$, $x,\alpha\in\RR$ and $h\in H_-$,
\begin{equation}
\left| S_N(x,\alpha;f) - \e^{s/4} \, \Theta_f(\Gamma n_+(x,\alpha) h \Phi^s) \right| \leq E_f(h),\label{statement-thm1-intro}
\end{equation}
where $N=\e^{s/2}$.
\end{theorem}

Of special interest is the choice $h=e$, since then $S_N(x,\alpha;f) = \e^{s/4} \, \Theta_f(\Gamma n(x,\alpha)\Phi^s)$. As we will see, Theorem \ref{thm:1} holds for a more general class of functions, e.g., for $C^1$ functions with compact support (in which case $\Theta_f$ is continuous but no longer smooth).
For more singular functions, such as piecewise constant, the situation is more complicated and we can only approximate $S_N(x,\alpha)$ for almost every $h$. We will make the assumptions on $h$ explicit in terms of natural Diophantine conditions, which exclude in particular $h=e$. 

\begin{theorem}\label{thm:2}
Let $\chi$ be the indicator function of the open interval $(0,1)$. Then there is a square-integrable function $\Theta_\chi:\GamG\to\CC$ and, for every $x\in\RR$, a measurable function $E_\chi^x:H_-\to [0,\infty)$ and a set $P^x\subset H_-$ of full measure, such that for all $s\in[0,\infty)$, $x,\alpha\in\RR$ and $h\in P^x$,
\begin{equation}
\left| S_N(x,\alpha) - \e^{s/4} \, \Theta_\chi(\Gamma n_+(x,\alpha) h \Phi^s) \right| \leq E_\chi^x(h),\label{statement-thm2-intro}
\end{equation}
where $N=\lfloor \e^{s/2} \rfloor$. 
\end{theorem}

This theorem in particular implies the approximate functional equation \eqref{ApproxFeq}, see Section \ref{sec:HL}.

The central part of our analysis is to understand the continuity properties of $\Theta_\chi$ and its growth in the cusps of $\GamG$, which, together with well known results on the dynamics of the flow $\GamG\to\GamG$, $\Gamma g \mapsto \Gamma g \Phi^s$, can be used to obtain both classical and new results on the value distribution of $S_N(x,\alpha)$ for large $N$. 
The main new application that we will focus on is an invariance principle for $S_N(x,\alpha)$ at random argument. A natural setting would be to take $x\in[0,2]$, $\alpha\in[0,1]$ uniformly distributed according to Lebesgue measure. We will in fact study a more general setting where $\alpha$ is fixed, and $\frac12 n^2$ is replaced by an arbitrary quadratic polynomial $P(n)=\frac12 n^2 + c_1 n +c_0$, with real coeffcients $c_0$, $c_1$. The resulting theta sum
\be
S_N(x)=S_N(x;P,\alpha)=\sum_{n=1}^N e \left( P(n) x+\alpha n\right),\label{theta-sum-intro}
\ee
is not necessarily periodic in $x$. We thus assume in the following that $x$ is distributed according to a given Borel probability measure $\lambda$ on $\RR$ which is absolutely continuous with respect to Lebesgue measure.

Let us consider the complex-valued curve $[0,1]\to\C$ defined by
\be
t\mapsto X_N(t)=\frac{1}{\sqrt N}\, S_{\lfloor t N\rfloor}(x)
+\frac{\{tN\}}{\sqrt N}
(S_{\lfloor tN\rfloor+1}(x)
-S_{\lfloor tN\rfloor}(x)).
\label{definition-X_N(x;t)}
\ee 
Figure \ref{fig:fivecurlicues} 
shows examples of $X_N(t)$ for five randomly generated values of $x$.

\begin{sidewaysfigure}
    \hspace{-3cm}\includegraphics[width=26cm]{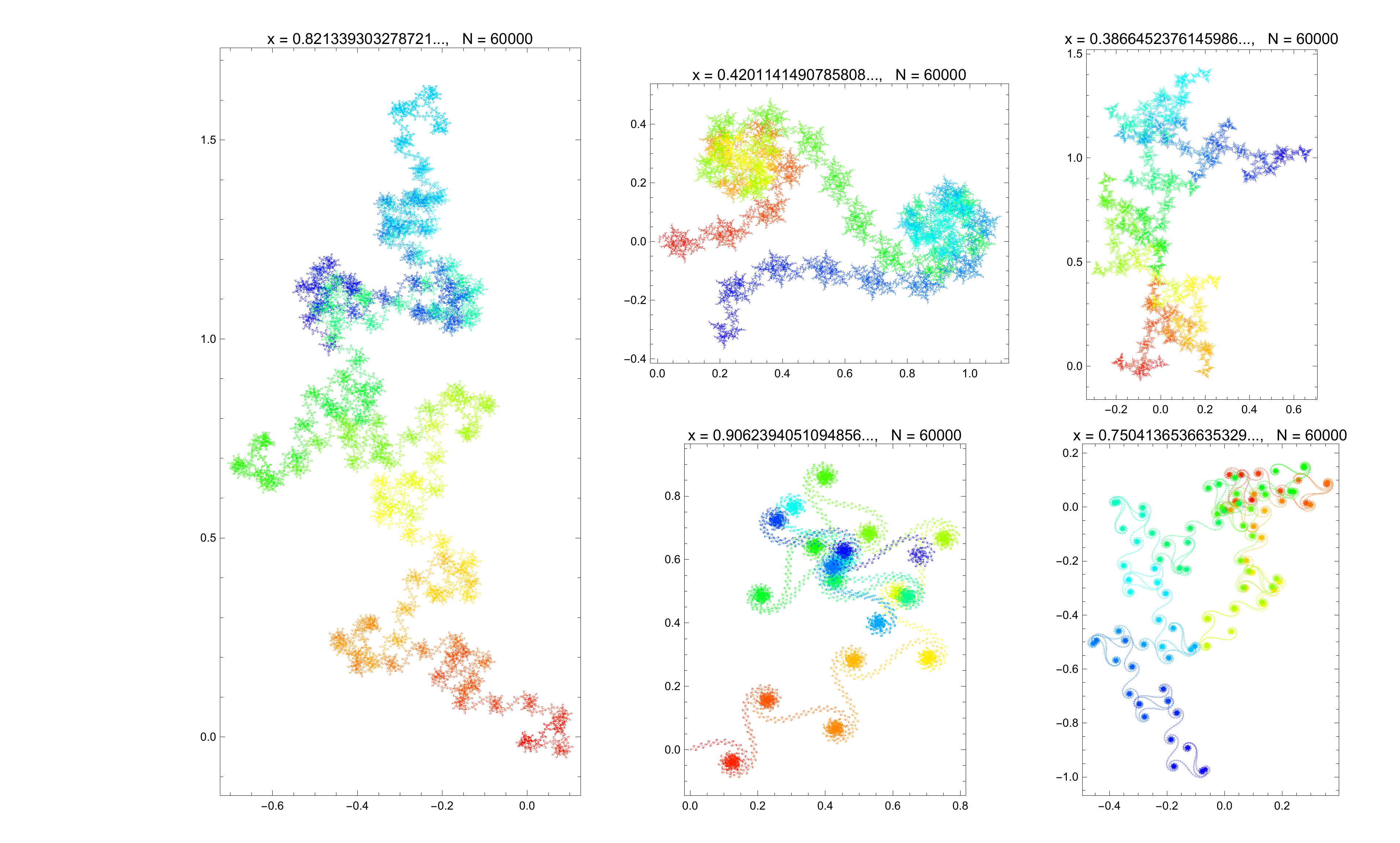}
        \caption{Curlicues $\{X_N(t)\}_{0<t\leq 1}$ for five randomly chosen $x$, and $\alpha=c_0=0$, $c_1=\sqrt2$. The color ranges from red at $t=0$ to blue at $t=1$. }
    \label{fig:fivecurlicues}
\end{sidewaysfigure}

Consider the space $\mathcal{C}_0=\mathcal{C}_0([0,1],\C)$ of complex-valued, continuous functions on $[0,1]$, taking value $0$ at $0$. Let us equip $\mathcal{C}_0$ with the uniform topology, i.e. the topology induced by the metric $d(f,g):=\|f-g\|$, where $\|f\|=\sup_{t\in[0,1]}|f(t)|$. The space $(\mathcal{C}_0,d)$ is separable and complete (hence Polish) and is called the \emph{classical Wiener space}.
The probability measure $\lambda$ on $\R$ induces, for every $N$, a probability measure on the space $\mathcal{C}_0$ of random curves $t\mapsto X_N(t)
$. For fixed $t\in[0,1]$, $X_N(t)$ is a random variable on $\C$.
The second principal result of this paper is the following.

\begin{theorem}[Invariance principle for quadratic Weyl sums]\label{thm-1} Let $\lambda$ be a Borel probability measure on $\RR$ which is absolutely continuous with respect to Lebesgue measure. Let $c_1,c_0,\alpha\in\R$ be fixed with $\sve{c_1}{\alpha}\notin\Q^2$. Then
\bianco
\begin{itemize}
\item[(i)] for every 
$t\in[0,1]$, we have 
\be\lim_{N\ti}\mathrm{Var}(X_N(t))=t;\ee
\item[(ii)]
there exist a random process $t\mapsto X(t)$ on $\C$ such that 
\be
X_N(t)\Longrightarrow X(t)\hspace{.5cm}\mbox{as $N\to\infty$},
\ee
where ``$\Rightarrow$'' denotes weak convergence of the induced probability measures on $\mathcal{C}_0$. The process $t\mapsto X(t)$ does not depend on the choice of $\lambda$, $P$ or $\alpha$.
\end{itemize}
\end{theorem}

The process $X(t)$ can be extended to arbitrary values of $t\geq 0$. We will refer to it as the \emph{theta process}. The distribution of $X(t)$ is a probability measure on $\mathcal{C}_0([0,\infty),\CC)$, and by ``almost surely'' we mean ``outside a null set with respect to this measure''. Moreover, by $X\sim Y$ we mean that the two random variables $X$ and $Y$ have the same distribution. 

Throughout the paper, we will use Landau's ``$O$'' notation and Vinogradov's ``$\ll$'' notation. By ``$f(x)=O(g(x))$'' and ``$f(x)\ll g(x)$'' we mean that there exists a constant $c>0$ such that 
$|f(x)|\leq c |g(x)|$. If $a$ is a  parameter, then by ``$O_{a}$'' and ``$\ll_{a}$'' we mean that $c$ may depend on $a$.

The properties of the theta process are summarized as follows.

\begin{theorem}[Properties of the theta process]\label{thm-2}\bianco
\begin{itemize}
\item[(i)] \textbf{Tail asymptotics}. 
For $R\geq 1$,
\be
\mathbb P\{|X(1)|\geq R\}=
\frac{6}{\pi^2}
R^{-6}\!\left(1+O(R^{-\frac{12}{31}})\right) . \label{tail-asymptotics-intro}
\ee
\item[(ii)] \textbf{Increments}. For every $k\geq2$ and every $t_0<t_1<t_2<\cdots< t_k$ the increments \be X(t_1)-X(t_0),\: X(t_2)-X(t_1),\: \ldots,\: X(t_k)-X(t_{k-1}) \ee are not independent.
\item[(iii)] \textbf{Scaling}. For $a>0$ let $Y(t)=\frac{1}{a}X(a^2t)$. Then $Y\sim X$.
\item[(iv)] \textbf{Time inversion}. Let 
\be Y(t):=\begin{cases}0&\mbox{if $t=0$};
\\ tX(1/t)& \mbox{if $t>0$}.
\end{cases} 
\ee Then $Y\sim X$.
\item[(v)] \textbf{Law of large numbers}. Almost surely, $\lim_{t\ti}\frac{X(t)}{t}=0$.
\item[(vi)] \textbf{Stationarity}. For $t_0\geq 0$ let $Y(t)=X(t_0+t)-X(t_0)$. Then $Y\sim X$.
\item[(vii)] \textbf{Rotational invariance}. For $\theta\in\R$ let $Y(t)=\e^{2\pi i\theta}X(t)$. Then $Y\sim X$.
\item[(viii)] \textbf{Modulus of continuity}. For every $\varepsilon>0$ there exists a constant $C_\varepsilon>0$ such that 
\be\limsup_{h\downarrow0}\sup_{0\leq t\leq 1-h}\frac{|X(t+h)-X(t)|}{\sqrt{h}(\log(1/h))^{1/4+\varepsilon}} \leq C_\varepsilon
\ee
almost surely.
\item[(ix)] \textbf{H\"older continuity}. Fix $\theta<\ha$. Then, almost surely, the curve $t\mapsto X(t)$ is everywhere locally $\theta$-H\"older continuous.
\item[(x)] \textbf{Nondifferentiability}. Fix $t_0\geq 0$. Then, almost surely, the curve $t\mapsto X(t)$ is not differentiable at $t_0$.
\end{itemize} 
\end{theorem}

\begin{remark}
Properties (i, ii, vii) allow us to predict the distribution of $|X_N(t)|$, $\Re(X_N(t))$, and $\Im(X_N(t))$ for large $N$. See Figure \ref{fig:histogram}.
Our approach in principle permits a generalization of Theorems \ref{thm-1} and \ref{thm-2} to the case of rational $\sve{c_1}{\alpha}\in\Q^2$, however, with some crucial differences. In particular, the tail asymptotics would now be of order $R^{-4}$, and stationarity and rotation-invariance of the process fail. In the special case $c_1=\alpha=0$, a limiting theorem for the absolute value $|X_N(1)|=N^{-1/2}|S_N(x)|$  was previously obtained by Jurkat and van Horne \cite{Jurkat-vanHorne1981-proofCLT}, \cite{Jurkat-vanHorne1982-CLT}, \cite{Jurkat-VanHorne1983} with tail asymptotics $\frac{4\log2}{\pi^2} R^{-4}$ (see also \cite[Example 75]{Cellarosi-thesis}), while the distribution for the complex random variable $X_N(1)=N^{-1/2}S_N(x)$ for was found by Marklof \cite{Marklof-1999}; the existence of finite-dimensional distribution of of the process $t\mapsto S_{\lfloor tN\rfloor}(x)$ was proven by Cellarosi \cite{Cellarosi-curlicue}, \cite{Cellarosi-thesis}. Demirci-Akarsu and Marklof \cite{DAM2013}, \cite{DA2014} have established analogous limit laws for incomplete Gauss sums, and Kowalski and Sawin \cite{KS2014} limit laws and  invariance principles for incomplete Kloosterman sums and Birch sums.
\end{remark}

\begin{center}
\begin{figure}[h!]
\includegraphics[width=8cm]{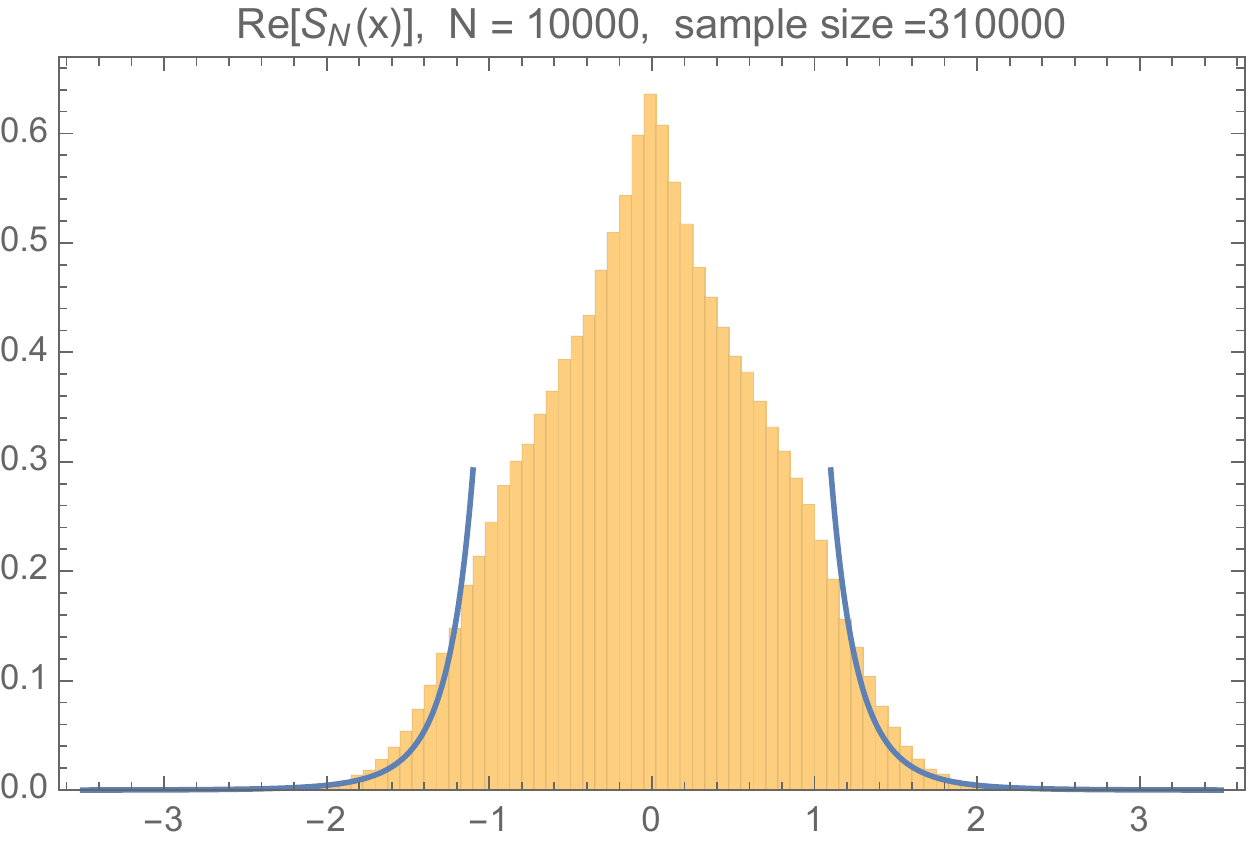}\hspace{.2cm}\includegraphics[width=8cm]{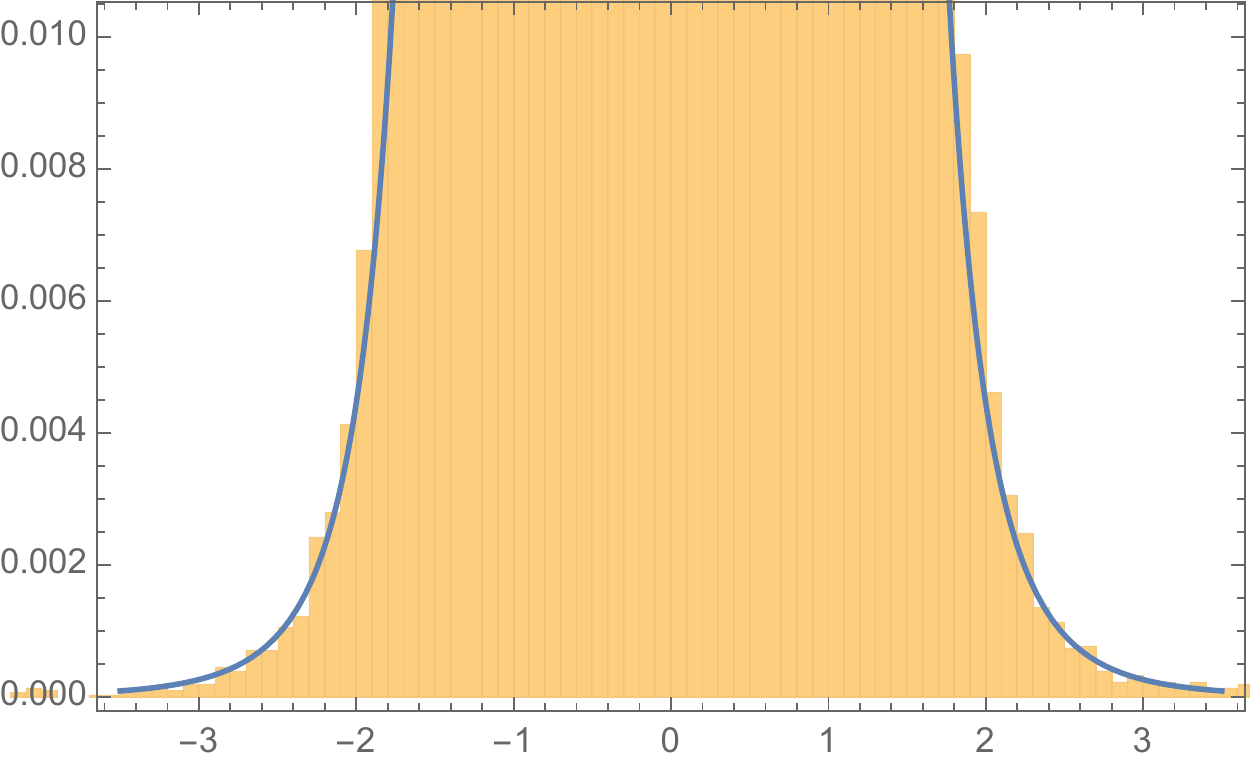}
\caption{The value distribution for the real part of $X_N(1)$, $N=10000$. The continuous curve is the tail estimate for the limit density $-\frac{\de}{\de x}\mathbb P\{\Re X(1)\geq x\}\sim \frac{45}{8\pi^2} |x|^{-7}$ as $|x|\to\infty$. This formula follows from the tail estimate for $|X(1)|$ in \eqref{tail-asymptotics-intro} by the rotation invariance of the limit distribution.}\label{fig:histogram}
\end{figure}
\end{center}

\begin{remark}
If we replace the quadratic polynomial $P(n)$ by a higher-degree polynomial, no analogues of the above theorems are known. If, however, $P(n)$ is replaced by a lacunary sequence $P(n)$ (e.g., $P(n)=2^n$), then $X_N(t)$ is well known to converge to a Wiener process (Brownian motion). In this case we even have an almost sure invariance principle; see Berkes \cite{Berkes-1975} as well as Philipp and Stout \cite{PS1975}. Similar invariance principles (both weak and almost sure) are known for sequences generated by rapidly mixing dynamical systems; see 
Melbourbe and Nicol \cite{Melbourne-Nicol-2009}, Gou\"{e}zel \cite{Gouezel2010} and references therein. The results of the present paper may be interpreted as  invariance principles for random skew translations. Griffin and Marklof \cite{Griffin-Marklof} have shown that a fixed, non-random skew translation does not satisfy a standard limit theorem (and hence no  invariance principle); convergence occurs only along subsequences. A similar phenomenon holds for other entropy-zero dynamical systems, such as translations (Dolgopyat and Fayad \cite{Dolgopyat-Fayad2014}), translation flows (Bufetov \cite{Bufetov2014}), tiling flows (Bufetov and Solomyak \cite{Bufetov-Solomyak2013}) and horocycle flows (Bufetov and Forni \cite{Bufetov-Forni}).
\end{remark}

\begin{remark}
Properties (i) and (ii) are the most striking differences between the theta process and the Wiener process. 
Furthermore, compare property (viii) with the following result by L\'evy for the Wiener process $W(t)$ \cite{levy1937theorie}: almost surely
\begin{align}
\limsup_{h\downarrow0}\sup_{0\leq t\leq 1-h}\frac{|W(t+h)-W(t)|}{\sqrt{2 h\log(1/h)}}=1.\end{align}
All the other properties are the same for sample paths of the Wiener process.
This means that typical realizations of the theta process are slightly more regular than those of the Wiener process, but this difference in regularity cannot be seen in H\"older norm (property (ix)).
Figure \ref{fig:five-real-curlicues} compares the real parts of the five curlicues in Figure \ref{fig:fivecurlicues} with five realization of a standard Wiener process.
\end{remark}

\begin{figure}
\begin{center}
\includegraphics[width=15cm]{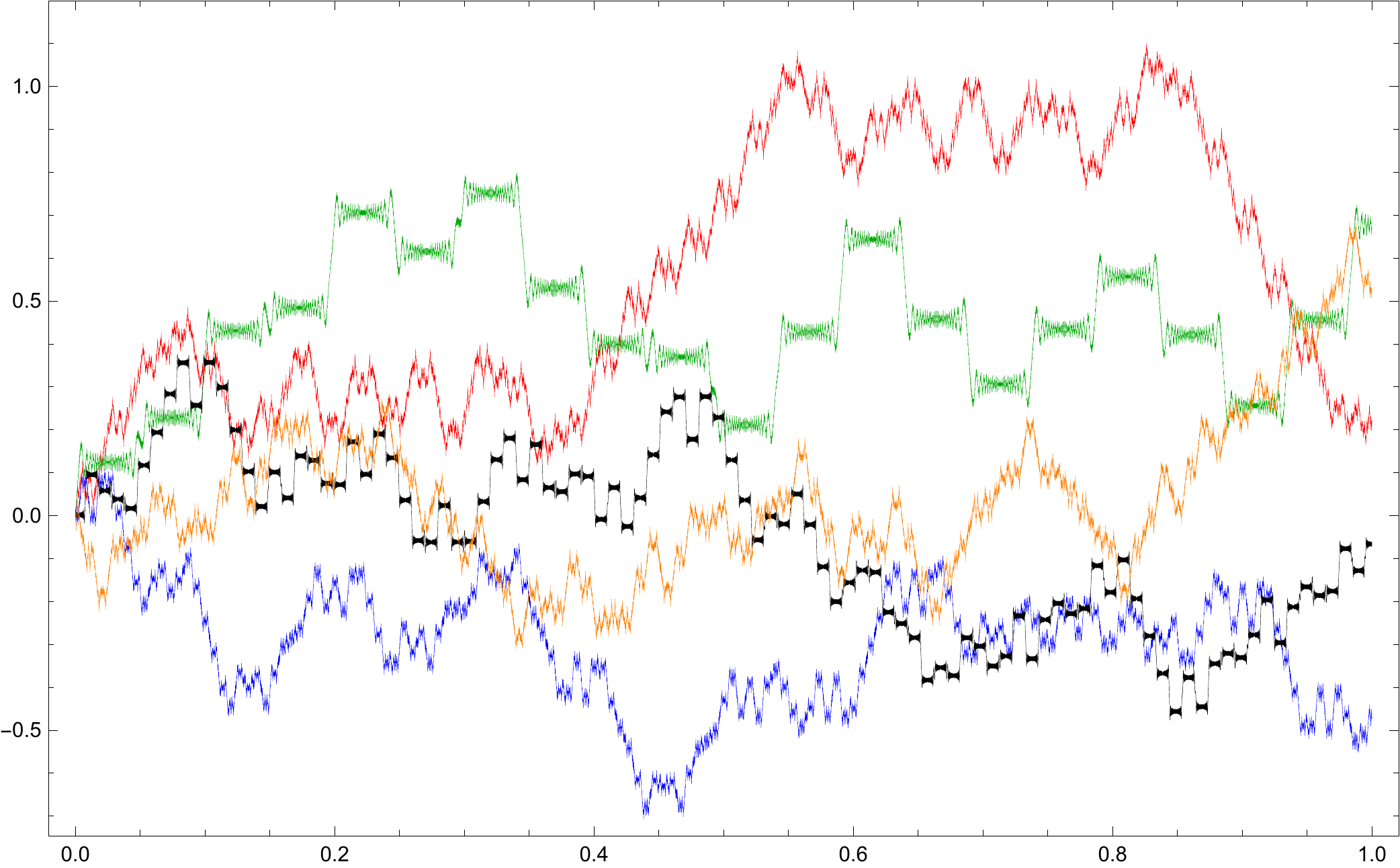}
\includegraphics[width=15cm]{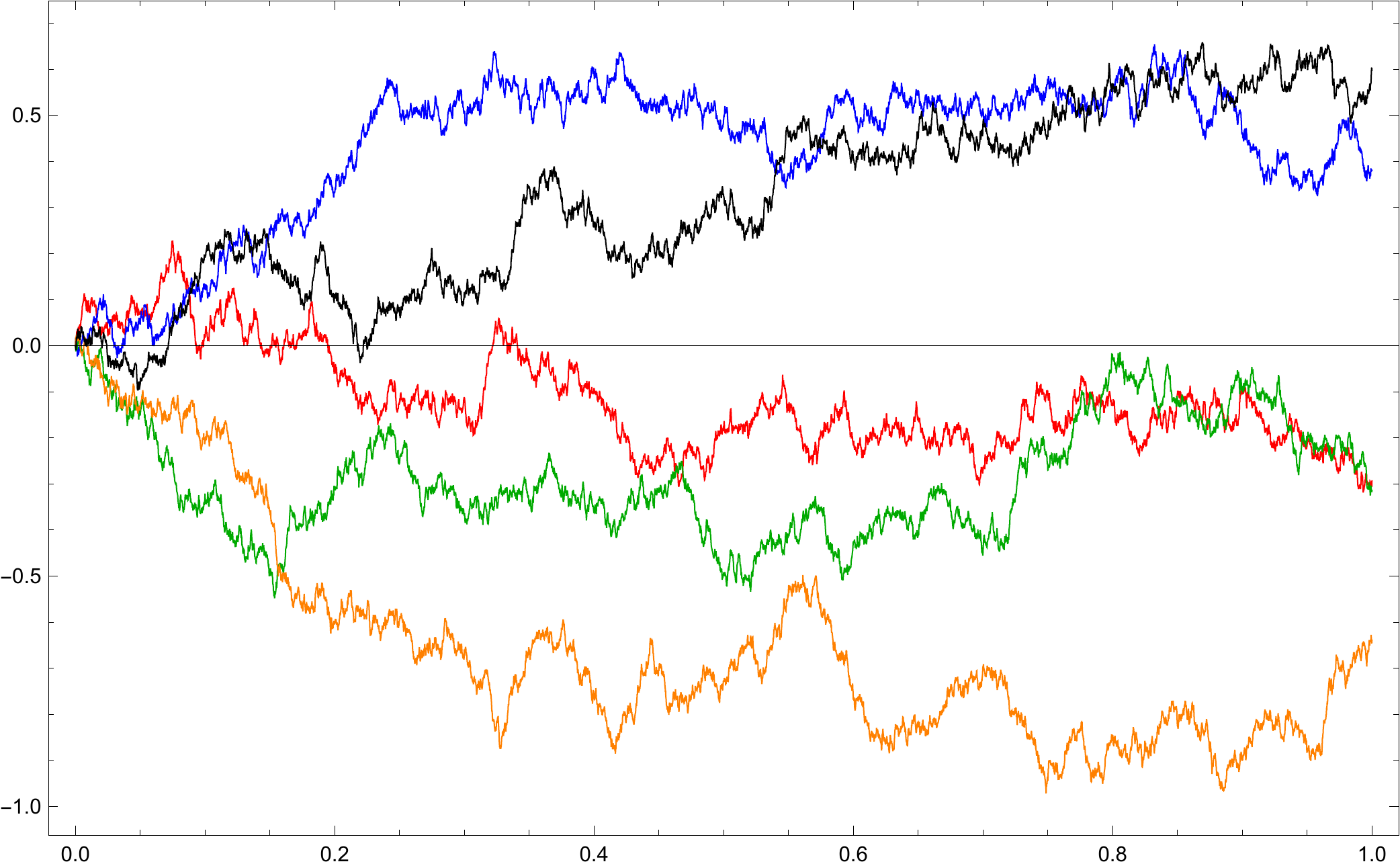}
\end{center}
\caption{Top: $t\mapsto \mathrm{Re}(X_N(t))$ for the five curlicues $\{X_N(t)\}_{0<t\leq 1}$ shown in Figure \ref{fig:fivecurlicues}. Bottom: five sample paths for the Wiener process.}\label{fig:five-real-curlicues}
\end{figure}

\begin{remark}
The tail asymptotics \eqref{tail-asymptotics-intro}  shows that the sixth moment of the limiting distribution of $X_N(1)=N^{-1/2}S_N(x)$ does not exist. In the special case $P(n)=\frac12 n^2$, with $x\in[0,2]$ and $\alpha\in[0,1]$ uniformly distributed, the sixth moment $\int_0^1\int_0^1 |S_N(\alpha;x)|^6 \de x\, \de\alpha$ yields the number $\scrQ(N)$ of solutions of the Diophantine system
\begin{equation}
\begin{split}
& x_1^2+x_2^2+x_3^2=y_1^2+y_2^2+y_3^2\\
& x_1+x_2+x_3=y_1+y_2+y_3
\end{split}
\end{equation}

with $1\leq x_i,y_i\leq N$ ($i=1,2,3$). Bykovskii \cite{Bykovskii} showed that $\scrQ(N)=\frac{12}{\pi^2} \rho_0 N^3\log N +O(N^3)$ with 
\begin{align}\label{theint}
\rho_0=\int_{-\infty}^\infty\int_{-\infty}^\infty \bigg|\int_0^1 e(u w^2-zw) dw\bigg|^6 dzdu .
\end{align}
Using a different method, N.N. Rogovskaya \cite{Rogovskaya} proved $\scrQ(N)=\frac{18}{\pi^2}N^3\log N +O(N^3)$, which yields (without having to compute the integral \eqref{theint} directly) $\rho_0=\frac32$. As we will see, the integral in \eqref{theint} also appears in the calculation of the tail asymptotics \eqref{tail-asymptotics-intro}. The currently best asymptotics for $\scrQ(N)$ is, to the best of our knowledge, due to V. Blomer and J. Br\"udern \cite{Blomer-Brudern}. 
\end{remark}

\begin{remark}
A different dynamical approach to quadratic Weyl sums has been developed by Flaminio and Forni \cite{Flaminio-Forni2006}. It employs nilflows and yields effective error estimates in the question of uniform distribution modulo one. Their current work generalizes this to higher-order polynomials \cite{Flaminio-Forni-2014}, and complements Wooley's recent breakthrough \cite{Wooley-2012}, \cite{Wooley-2014}. It would be interesting to see whether Flaminio and Forni's techniques could provide an alternative to those presented here, with the prospect of establishing invariance principles for cubic and other higher-order Weyl sums.
\end{remark}

 This paper is organized as follows. In Section \ref{section:jacobi-theta-sums} we define complex-valued Jacobi theta functions, and we construct a probability space  on which they are defined. The probability space  is realized as a 6-dimensional homogeneous space $\GamG$.
 Theorem \ref{thm:1} is proven in Section \ref{sec:pf-thm-1} 
 and is used in the proof of Theorem \ref{thm:2}, which is carried out in Section \ref{section:dyadic-representation-Theta-sharp-cutoff}. This section also includes a new proof of Hardy and Littlewood's approximate functional equation (Section \ref{sec:HL}) and discusses several properties of the automorphic function $\Theta_\chi$. In Section  \ref{sec:limit-theorems} we first prove the existence of finite-dimensional limiting distribution for quadratic Weyl sums (Section \ref{section:convergence-fdd}) using  equidistribution  of closed horoycles in $\GamG$ under the action of the geodesic flow, then we prove that the finite dimensional distributions are tight (Section \ref{sec:tightness}). 
As a consequence, the finite dimensional distributions define a random process (a probability measure on $\mathcal C_0$), whose explicit formula is given in Section \ref{sec:the-limiting-process}. This formula is the key to derive all the properties of the process listed in Theorem \ref{thm-2}. Invariance properties are proved in Section \ref{section: invariance properties}, continuity properties in Section \ref{section:Holder-continuity}.

\section*{Acknowledgments}
 We would like to thank Trevor Wooley for telling us about references 
 \cite{Bykovskii}, \cite{Rogovskaya}, \cite{Blomer-Brudern}, which helped in the explicit evaluation of the leading-order constant of the tail asymptotics \eqref{tail-asymptotics-intro}. We are furthermore grateful to the Max Planck Institute for Mathematics in Bonn for providing a very congenial working environment during the program \emph{Dynamics and Numbers} in June 2014, where part of this research was carried out. 

The first author gratefully acknowledges the financial support of the AMS-Simons travel grant (2013-2014) and the NSF grant DMS-1363227. The second author thanks the Isaac Newton Institute, Cambridge for its support and hospitality during the semester \emph{Periodic and Ergodic Spectral Problems}. 
 The research leading to the results presented here has received funding from the European Research Council under the European Union's Seventh Framework Programme (FP/2007-2013) / ERC Grant Agreement n. 291147.

\section{Jacobi theta functions}\label{section:jacobi-theta-sums}

This section explains how to identify the theta sums $S_N(x,\alpha;f)$ in \eqref{theta-sum-intro01} with automorphic functions $\Theta_f$ on the Jacobi group $G$, provided $f$ is sufficiently regular and of rapid decay. These automorphic functions arise naturally in the representation theory of $\SL(2,\R)$ and the Heisenberg group, which we recall in Sections \ref{sec:2.1}--\ref{subsection-Jacobi-group-Schrodinger-Weil-representation}. The variable $\log N$ has a natural dynamical interpretation as the time parameter of the ``geodesic'' flow on $G$, whereas $(x,\alpha)$ parametrises the expanding directions of the geodesic flow (Section \ref{section:geodesic+horocycle+flows}). Section \ref{sec:JTF} states the transformation formulas for $\Theta_f$, which allows us to represent them as functions on $\GamG$, where $\Gamma$ is a suitable discrete subgroup. Sections \ref{sec:growth}--\ref{subsection:Hermite-f} provide more detailed analytic properties of $\Theta_f$, such as growth in the cusp and square-integrability. The proof of Theorem \ref{thm:1} is based on the exponential convergence of nearby points in the stable direction of the geodesic flow (Section \ref{sec:pf-thm-1}).

\subsection{The Heisenberg group and its Schr\"odinger representation}\label{sec:2.1}
Let $\omega$ be the standard symplectic form on $\R^2$, $\omega(\vecxi,\bm{\xi'})=xy'-yx'$, where $\vecxi=\sve{x}{y},\bm{\xi'}=\sve{x'}{y'}$. The Heisenberg group $\Hei$ is defined as $\R^2\times\R$ with the multiplication law
\be(\vecxi,t)(\bm{\xi'},t')=\left(\vecxi+\bm{\xi'},t+t'+\tfrac{1}{2}\omega(\vecxi,\bm{\xi'})\right).\ee
The group $\Hei$ defined above is isomorphic to the group of upper-triangular matrices 
\be\left\{\begin{pmatrix}1&x&z\\0&1&y\\0&0&1\end{pmatrix},\:\:x,y,z\in\R\right\}, \ee
with the usual matrix multiplication law. The isomorphism is given by \be \veH{x}{y}{t}\mapsto\begin{pmatrix}1&x&t+\tha xy\\0&1&y\\0&0&1\end{pmatrix}.\label{isomorphism-Heisenberg}\ee 

The following decomposition holds:
\be\veH{x}{y}{t}=\veH{x}{0}{0}\veH{0}{y}{0}\veH{0}{0}{t-\frac{xy}{2}}.\ee

The Schr\"{o}dinger representation $W$ of $\Hei$ on $\mathrm L^2(\R)$ is defined by 
\bey
\left[W\!\veH{x}{0}{0}f\right]\!(w)&=&e(xw)f(w),\\
\left[W\!\veH{0}{y}{0}f\right]\!(w)&=&f(w-y),\\
\left[W\!\veH{0}{0}{t}f\right]\!(w)&=&e(t)\mathrm{id},
\eey
with $x,y,t,w\in \R$.

For every $M\in\sltr$ we can define a new representation of $\Hei$ by setting
$W_M(\vecxi,t)=W(M\vecxi,t)$. All such representations are irreducible and unitarily equivalent. Thus for each $M\in\sltr$ there is a unitary operator $R(M)$ s.t.
\be R(M)W(\vecxi,t)R(M)^{-1}=W(M\vecxi,t).\label{RWR^-1=W}\ee

$R(M)$ is determined up to a unitary phase cocycle, i.e.
\be R(MM')=c(M,M')R(M)R(M'),\ee
with $c(M.M')\in\C$, $|c(M,M')|=1$.
If \be M_1=\ma{a_1}{b_1}{c_1}{d_1},\:\:M_2=\ma{a_2}{b_2}{c_2}{d_2},\:\:M_3=\ma{a_3}{b_3}{c_3}{d_3},\ee
with $M_1M_2=M_3$, then
\be c(M_1,M_2)=\e^{-i \pi\,\sgn(c_1c_2c_3)/4}.\label{cocycle-c}\ee
$R$ is the so-called \emph{projective Shale-Weil representation of $\sltr$}, and lifts to a true representation of its universal cover $\tsltr$. 

 \subsection{Definition of $\tsltr$}\label{section: tsltr}
Let $\h:=\{w\in\C:\:\Im(w)>0\}$ denote the upper half plane.
The group $\sltr$ acts on $\h$ by M\"{o}bius transformations $z\mapsto gz:=\frac{az+b}{cz+d}$, where $g=\sma{a}{b}{c}{d}\in\sltr$. Every $g\in\sltr$ can be written uniquely as the Iwasawa decomposition
\be g=n_xa_yk_\phi ,\label{Iwasawa-dec}\ee
where 
\be n_x=\ma{1}{x}{0}{1},\:\:a_y=\ma{y^{1/2}}{0}{0}{y^{-1/2}},\:\:k_\phi=\ma{\cos\phi}{-\sin\phi}{\sin\phi}{\cos\phi},\ee
and $z=x+i y\in\h$, $\phi\in[0,2\pi)$.
This allows us to parametrize $\mathrm{SL}(2,\R)$ with $\h\times[0,2\pi)$; we will use the shorthand $(z,\phi):=n_xa_yk_\phi$.
Set $\epsilon_g(z)=(cz+ d)/|cz+ d|$. The universal cover of $\sltr$ is defined as 
\be\tsltr:=\{[g,\beta_g]:\:g\in\sltr,\:\beta_g \mbox{ a continuous function on $\h$ s.t. }\e^{i\beta_g(z)}=\epsilon_g(z)\},\ee 
and has the group structure given by
\be[g,\beta_g^1][h,\beta_h^2]=[gh,\beta_{gh}^3],\qquad\beta_{gh}^3(z)=\beta_g^1(hz)+\beta_h^2(z),\label{mult-tsltr}
\ee
\be
[g,\beta_g]^{-1}=[g^{-1},\beta_{g^{-1}}'],\qquad \beta_{g^{-1}}'(z)=-\beta_{g}(g^{-1}z).\ee
$\tsltr$ is identified with $\h\times\R$ via $[g,\beta_g]\mapsto(z,\phi)=(g i,\beta_g(i))$ and it acts on $\h\times\R$ via
\be [g,\beta_g](z,\phi)=(gz,\phi+\beta_g(z)).\label{act-tsltr-on-hxR}\ee 

We can extend the Iwasawa decomposition \eqref{Iwasawa-dec} of $\sltr$ to a decomposition of $\tsltr$ (identified with $\h\times\R$): for every $\tilde g=[g,\beta_g]\in\tsltr$ we have
\be\tilde g=[g,\beta_g]=\tilde n_x\tilde a_y\tilde k_\phi=[n_x,0][a_y,0][k_\phi,\beta_{k_\phi}].\label{Iwasawa-tlstr}\ee

For $m\in\N$ consider the cyclic subgroup $Z_m=\langle(-1,\beta_{-1})^m\rangle$, where $\beta_{-1}(z)=\pi$. In particular, we can recover the classical groups $\psltr=\tsltr/Z_1$ and $\sltr=\tsltr/Z_2$.

\subsection{Shale-Weil representation of $\tsltr$}\label{section-Schrodinger-Shale-Weil-reps}

The Shale-Weil representation $R$ of defined above as a projective representation of $\sltr$ lifts to a true representation of $\tsltr$ as follows.
Using the decomposition \eqref{Iwasawa-tlstr}, it is enough to define the representation on each of the three factors as follows (see \cite{Lion-Vergne}):
for $f\in\LtR$ let
\be
\left[R( \tilde n_x)f\right](w)=\left[R(n_x)f\right](w):=e\!\left(\tfrac{1}{2}w^2 x\right)f(w), 
\ee
\be
\left[R ( \tilde a_y)f\right](w)=\left[R (a_y)f\right](w):=y^{1/4}f(y^{1/2}w), 
\ee
\be
[R(k_\phi )f ](w)=\begin{cases} 
f(w),&\mbox{if $\phi\equiv0$ mod $2\pi$,}\\
f(-w),&\mbox{if $\phi\equiv\pi$ mod $2\pi$,}\\ \\\displaystyle
|\sin\phi|^{-1/2}\!\int_\R e\!\left(\frac{\ha(w^2+w'^2)\cos\phi-w w'}{\sin\phi}\right)f(w')\de w',&\mbox{if $\phi\equiv\hspace{-.38cm}/\hspace{.2cm}0$ mod $\pi$,}
\end{cases}\:\:\label{def-R-tilde-k-f}
\ee
and $R(\tilde k_\phi)=e(-\sigma_\phi/8)R(k_\phi)$. The function $\phi\mapsto\sigma_\phi$ is given by
\be
\sigma_\phi:=\begin{cases}2\nu,&\mbox{if $\phi=\nu\pi$, $\nu\in\Z$;}\\
2\nu+1,&\mbox{if $\nu\pi<\phi<(\nu+1)\pi$, $\nu\in\Z$.}\end{cases}\label{def-sigma-phi}
\ee
and the reason for the factor $e(-\sigma_\phi/8)$ in the definition of $R(\tilde k_\phi)$ is that for $f\in\SR$ \be\lim_{\phi\to0\pm}\left[R(k_\phi)f\right](w)=e(\pm\tfrac{1}{8})f(w).\ee
Throughout the paper, we will use the notation $f_\phi(w)=[R(\tilde k_\phi)f](w)$. 
More explicitly, the Shale-Weil representation of $\tsltr$ on $\LtR$ reads as
\be  [R(z,\phi)f](w)=[R(\tilde n_x)R(\tilde a_y)R(\tilde k_\phi)f](w)=y^{1/4}e(\tha w^2 x)f_\phi(y^{1/2}w),
\ee
where $z=x+i y\in\h$ and $\phi\in\R$.

\subsection{The Jacobi group and its Schr\"{o}dinger-Weil representation}\label{subsection-Jacobi-group-Schrodinger-Weil-representation}
The \emph{Jacobi group} is defined as the semidirect product \be\sltr\ltimes\Hei\ee
with multiplication law
\be(g;\vecxi,\zeta)(g';\bm{\xi'},\zeta')=\left(gg';\vecxi+g\bm{\xi'},\zeta+\zeta'+\tha\omega(\vecxi,g\bm{\xi'})\right).\label{mult-Jacobi}\ee
The special affine group $\ASL(2,\R)=\sltr\ltimes\R^2$ is isomorphic to the subgroup $\sltr\ltimes (\R^2\times\{0\})$ of the Jacobi group and has the multiplication law
\be(g;\vecxi)(g';\bm{\xi'})=\left(gg';\vecxi+g\bm{\xi'})\right)\label{mult-ASL}.\ee
For $\sltr\ni g=n_x a_y k_\phi=(x+i y,\phi)\in\h\times[0,2\pi)$ let $R(g)f:=R(n_x)R(a_y)R(k_\phi)f$.
If we rewrite \eqref{RWR^-1=W} as
\be R(g)W(\vecxi,t)=W(g\vecxi,t)R(g),\label{RWWR}\ee
then \be R(g;\vecxi,t)=W(\vecxi,t)R(g)\label{defR-projective}\ee defines a projective representation of the Jacobi group with cocycle $c$ as in \eqref{cocycle-c}. It is called the \emph{Schr\"{o}dinger-Weil representation}.
For $\tsltr\ni[g,\beta_g]=\tilde n_x\tilde a_y\tilde k_\phi=(x+i y,\phi)\in\h\times\R$, we define
\be R(z,\phi;\vecxi,t)=W(\vecxi,t) R(z,\phi),\label{def-hat-R}\ee
 and we get a genuine representation of the universal Jacobi group 
 \be G=\tsltr\ltimes\Hei=(\h\times\R)\ltimes\Hei,\ee
having the multiplication law
\be([g,\beta_g];\vecxi,\zeta)([g',\beta'_{g'}];\bm{\xi'},\zeta')=\left([gg',\beta''_{gg'}];\vecxi+g\bm{\xi'},\zeta+\zeta'+\tha\omega(\vecxi,g\bm{\xi'})\right)\label{mult--univ-Jacobi},\ee
where $\beta''_{gg'}(z)=\beta_g(g'z)+\beta'_{g'}(z)$.
The Haar measure on $G$ is given in coordinates $(x+i y,\phi;\scriptsize{\sve{\xi_1}{\xi_2}},\zeta)$ by \be\de \mu(g)=\frac{\de x\,\de y\,\de\phi\,\de\xi_1\,\de\xi_2\,\de\zeta}{y^2}.\label{Haar-measure-on-G}\ee

\subsection{Geodesic and horocycle flows on $G$}\label{section:geodesic+horocycle+flows}
The group $G$ naturally acts on itself by multiplication. 
Let us consider the action by right-multiplication  by elements of the 1-parameter group
$\{\Phi^t:\:t\in\RR\}$ (the \emph{geodesic flow}), where 
\be \Phi^t=\left(\left[\left(\ba{cc}\e^{-t/2}&0\\0&\e^{t/2}\ea\right),0\right];\bm0,0\right).\ee
Let $H_+$ and $H_-$ be the unstable and stable manifold for $\{\Phi^s\}_{s\in\R}$, respectively. That is
\begin{align}
H_+&=\{g\in G:\: \Phi^s g \Phi^{-s}\to e\hspace{.3cm}\mbox{as $s\to\infty$}\},\\
H_-&=\{g\in G:\: \Phi^{-s} g \Phi^s\to e\hspace{.3cm}\mbox{as $s\to\infty$}\}.
\end{align}
A simple computation using \eqref{mult--univ-Jacobi} yields
\begin{align}
H_+&=\left\{\left(\left[\left(\ba{cc}1&x\\0&1\ea\right),0\right];\ve{\alpha}{0},0\right):\: x,\alpha\in\RR\right\},\\
H_-&=\left\{\left(\left[\left(\ba{cc}1&0\\u&1\ea\right),\arg(u\cdot+1)\right];\ve{0}{\beta},0\right):\: u,\beta\in\RR\right\}.
\end{align}
We will denote the elements of $H_+$ by $n_+(x,\alpha)$ (see the Introduction) and those of $H_-$ by $n_-(u,\beta)$.
We will also denote by 
\be\Psi^{x}=n_+(x,0)=\left(\left[\left(\ba{cc}1&x\\0&1\ea\right),0\right];\bm 0,0\right)\ee 
the \emph{horocycle flow} corresponding to the unstable $x$-direction only.

\subsection{Jacobi theta functions as functions on $G$}\label{subsection-Jacobi-theta-functions-universal-Jacobi-group}
Let us consider the space of functions $f:\R\to\R$ for which $f_\phi$ has some decay at infinity, uniformly in $\phi$: let us denote
\begin{align}
\kappa_\eta(f)=\sup_{w,\phi}|f_\phi(w)|(1+|w|)^\eta.\label{def-kappa_eta(f)}
\end{align} and define 
\be\mathcal S_{\eta}(\mathbb{R}):=\left\{f:\R\to\R\:\big|\:\kappa_\eta(f)<\infty\right\},\ee
see \cite{Marklof2007b}. It generalizes the Schwartz space, 
since $\SR\subset\mathcal S_\eta(\mathbb{R})$ for every $\eta$.
For $g\in G$ and $f\in\mathcal S_\eta(\R)$, $\eta>1$, define the \emph{Jacobi theta function} as
\be
\Theta_f(g):=\sum_{n\in\Z}[R(g)f](n).\label{def-Jacobi-Theta-function-on-G}
\ee
More explicitely, for $g=(z,\phi;\vecxi,\zeta)$,
\be
\Theta_f(z,\phi;\vecxi,\zeta)=y^{1/4}e(\zeta-\tha \xi_1\xi_2)\sum_{n\in\Z}f_{\phi}\left((n-\xi_2)y^{1/2}\right)e\!\left(\tha(n-\xi_2)^2x+n\xi_1\right)\label{Jacobi-theta-sum-2},\ee
where $z=x+ i y$, $\vecxi=\sve{\xi_1}{\xi_2}$ and $f_\phi= R(i,\phi)f$.
In the next section we will show that there is a discrete subgroup $\Gamma< G$, so that $\Theta_f(\gamma g)=\Theta_f(g)$ for all $\gamma\in\Gamma$, $g\in G$. The theta function $\Theta_f$ is thus well defined as a function on $\GamG$.

For the original theta sum \eqref{theta-sum-intro} we have
\be
S_N(x)=y^{-1/4}\Theta_f(x+i y,0;\sve{\alpha+c_1 x}{0},c_0x),\label{rewriting-SNalpha}
\ee
where $y=N^{-2}$ and $f=\bm{1}_{(0,1]}$ is the indicator function of $(0,1]$. Here
$\Theta_f(z,0;\vecxi,\zeta)$ is well defined 
because the series in \eqref{Jacobi-theta-sum-2} is a finite sum. The same is true for $\Theta_f(z,\phi;\vecxi,\zeta)$ when $\phi\equiv 0 \bmod \pi$ by \eqref{def-R-tilde-k-f}. However, for other values of $\phi$, the function $f_\phi(w)$ decays too slow as $|w|\to\infty$ and we have $f\notin \mathcal S_\eta(\R)$ for any $\eta>1$. For example, for $\phi=\pi/2$,
\be 
f_{\pi/2}(w)=\e^{-\frac{\pi i}{4}}\int_0^1\e^{-2\pi i w w'}\de w'=\e^{\frac{\pi i}{4}}\;\frac{\e^{-2\pi i w}-1}{2\pi w} ,
\ee
and the series \eqref{Jacobi-theta-sum-2} defining $\Theta_\chi(z,\pi/2;\vecxi,\zeta)$ does not converge absolutely. This illustrates that 
$\Theta_f(\gamma(z,0;\vecxi,\zeta))$ may not be well-defined for general $(z,0;\vecxi,\zeta)$ and $\gamma\in\Gamma$. We shall show in Section \ref{section:dyadic-representation-Theta-sharp-cutoff} how to overcome this problem---the key step in the proof of Theorem \ref{thm:2}.

\subsection{Transformation formul\ae}\label{sec:JTF}
The purpose of this section is to determine a subgroup $\Gamma$ of $G$ under which the Jacobi theta function $\Theta_f(z,\phi;\vecxi,\zeta)$ is invariant. Fix $f\in\mathcal{S}_\eta$, $\eta>1$. We have the 
 following transformation formul\ae\: (cf. \cite{Marklof2003ann}):
\begin{align}
&\Theta_f\!\left(-\frac{1}{z},\phi+\arg z;\ve{-\xi_2}{\xi_1},\zeta\right)=\e^{-i\frac{\pi}{4}}\Theta_f\left(z,\phi,\vecxi,\zeta\right)\label{Jacobi1}\\
&\Theta_f\!\left(z+1,\phi,\ve{\frac{1}{2}}{0}+\ma{1}{1}{0}{1}\ve{\xi_1}{\xi_2},\zeta+\frac{\xi_2}{4}\right)=\Theta_f\left(z,\phi,\vecxi,\zeta\right)\label{Jacobi2}\\
&\Theta_f\!\left(z,\phi,\bm m+\vecxi,r+\zeta+\tha\omega\left(\bm m,\bm \xi\right)\right)=(-1)^{m_1 m_2}\Theta_f\left(z,\phi,\vecxi,\zeta\right),\hspace{.3cm}\bm m\in\Z^2, r\in\Z\label{Jacobi3}
\end{align}
Notice that
\begin{equation}
\begin{split}
\left(-\frac{1}{z},\phi+\arg z;\ve{-\xi_2}{\xi_1},\zeta\right)&=\left(\left[\ma{0}{-1}{1}{0},\arg\right];\bm 0,0\right)(z,\phi;\vecxi,\zeta) \\
&=\left(i,\frac{\pi}{2};\bm 0,0\right)(z,\phi;\vecxi,\zeta). 
\end{split}
\end{equation}

In other words, \eqref{Jacobi1} describes how the Jacobi theta function $\Theta_f$ transforms under left multiplication by $\left(i,\frac{\pi}{2};\bm 0,0\right)$. Define
\begin{align}
\gene_1&=\left(\left[\ma{0}{-1}{1}{0},\arg\right];\bm 0,\frac{1}{8}\right)=\left(i,\frac{\pi}{2};\bm0,\frac{1}{8}\right),\label{invariance-by-h1}\\
\gene_2&=\left(\left[\ma{1}{1}{0}{1},0\right];\ve{1/2}{0},0\right)=\left(1+i,0;\ve{1/2}{0},0\right),\\
\gene_3&=\left(\left[\ma{1}{0}{0}{1},0\right];\ve{1}{0},0\right)=\left(i,0;\ve{1}{0},0\right),\\
\gene_4&=\left(\left[\ma{1}{0}{0}{1},0\right];\ve{0}{1},0\right)=\left(i,0;\ve{0}{1},0\right),\\
\gene_5&=\left(\left[\ma{1}{0}{0}{1},0\right];\ve{0}{0},1\right)=\left(i,0;\bm 0,1\right).
\end{align}
Then (\ref{Jacobi1}, \ref{Jacobi2},
\ref{Jacobi3}) imply that for $i=1,\ldots,5$ we have
$\Theta_f(\gene_i g)=\Theta_f(g)$ for every $g\in G$.
The Jacobi theta function $\Theta_f$ is therefore invariant under the left action by the group 
\be \Gamma=\langle \gene_1,\gene_2,\gene_3,\gene_4,\gene_5\rangle
<G,\label{def-Gamma}.\ee 
This means that $\Theta_f$ is well defined on the quotient $\Gamma\backslash G$. 
Let $\Gamma_0$ be the image of $\Gamma$ under the natural homomorphism $\varphi:G \to G/Z \simeq \ASL(2,\RR)$, with
\begin{equation}
Z=\{ (1,2\pi m;\vecnull,\zeta) : m\in\ZZ,\; \zeta\in\RR\} .\label{def-Z}
\end{equation}
Notice that $\Gamma_0$ is commensurable to $\ASL(2,\Z)=\sltz\ltimes\Z^2$. Moreover, for fixed $(g,\vecxi)\in\Gamma_0$ we have that $\{([g,\beta_g];\bm \xi,\zeta)\in\Gamma:\: (g,\vecxi)\in\Gamma_0\}$ projects via $([g,\beta_g];\vecxi,\zeta)\mapsto (\beta_g(i),\zeta)\in\R\times\R$ onto $\{(\beta_g(i)+k \pi,\frac{k}{4}+l):\: k,l\in\Z\}$ since $\gene_1^2$ fixes the point $(g,\bm \xi)$. This means that $\Gamma$ is discrete and that $\Gamma\backslash G$ is a 4-torus bundle over the modular surface $\sltz\backslash\frak H$. This implies that $\GamG$ is non-compact. A fundamental domain for the action of $\Gamma$ on $G$ is 
\be \mathcal F_{\Gamma}=\left\{(z,\phi;\vecxi,\zeta)\in\mathcal F_{\SL(2,\Z)} \times [0,\pi) \times [-\tfrac12,\tfrac12)^2\times [-\tfrac12,\tfrac12) 
\right\},\label{fund-dom-H}\ee
where $\mathcal F_{\SL(2,\Z)}$ is a fundamental domain of the modular group in $\h$. The hypebolic area of $\mathcal F_{\SL(2,\Z)}$ is $\frac{\pi}{3}$, and hence,
by \eqref{Haar-measure-on-G}, we find $\mu(\GamG)=\mu(\mathcal F_\Gamma)=\frac{\pi^2}{3}$.

\subsection{Growth in the cusp}\label{sec:growth}

We saw that if $f\in \mathcal S_\eta(\R)$ with $\eta>1$, then $\Theta_f$ is a function on $\GamG$. We now observe that it is unbounded in the cusp $y>1$ and provide the precise asymptotic. Recall \eqref{def-kappa_eta(f)}.

\begin{lem}\label{lem-expansion-at-infinity}
Given $\xi_2\in\RR$, write $\xi_2=m+\theta$, with $m\in\Z$ and $-\tha\leq\theta<\tha$. Let $\eta>1$. Then there exists a constant $C_\eta$ such that for $f\in\mathcal{S}_\eta(\R)$, $y\geq\ha$ and all $x,\phi,\vecxi,\zeta$,
\be\left|\Theta_f(x+iy,\phi;\vecxi,\zeta)-y^{1/4}e\!\left(\zeta+\frac{(m-\theta)\xi_1+\theta^2x}{2}\right)f_\phi(-\theta y^{\ha})\right|\leq C_\eta \kappa_\eta(f)\,  y^{-(2\eta-1)/4}.\label{exp-at-infty}\ee
\end{lem}
\begin{proof}
Since the term $y^{1/4}e\!\left(\zeta+\frac{(m-\theta)\xi_1+\theta^2x}{2}\right)f_\phi(-\theta y^{\ha})$ in the left hand side of \eqref{exp-at-infty} comes from the index $n=m$, it is enough to show that
\be\left|\sum_{n\neq m}f_\phi\!\left((n-\xi_2)y^{\ha}\right)e\!\left(\tha(n-\xi_2)^2x+n\xi_1\right)\right|\leq C_\eta\kappa_\eta(f)\,y^{-\eta/2}. \ee
Indeed,
\bey
&&\left|\sum_{n\neq m}f_\phi\!\left((n-\xi_2)y^{\ha}\right)e\!\left(\tha(n-\xi_2)^2x+n\xi_1\right)\right|\leq\sum_{n\neq m}\left|f_\phi\!\left((n-\xi_2)y^{\ha}\right)\right| \\
&&\leq\sum_{n\neq m}\frac{\kappa_\eta(f)\,}{\left(1+|n-\xi_2|y^{\ha}\right)^\eta}=\kappa_\eta(f)\,y^{-\eta/2}\sum_{n\neq m}\frac{1}{\left(y^{-1/2}+|n-m-\theta|\right)^\eta} \\
&&=\kappa_\eta(f)\,y^{-\eta/2}\sum_{n\neq 0}\frac{1}{\left(y^{-1/2}+|n-\theta|\right)^\eta} \leq C_\eta \kappa_\eta(f)\,y^{-\eta/2} .
\eey
\end{proof}

Lemma \ref{lem-expansion-at-infinity} allows us derive an asymptotic for the measure of the region of $\GamG$ where the theta function $\Theta_f$ is large. 
Let us define
\begin{align}
D(f):=\int_{-\infty}^\infty\int_0^{\pi}|f_{\phi}(w)|^6\de\phi\,\de w.\label{def-D(f)}
\end{align}

\begin{lem}\label{mu(tail>R)}
Given $\eta>1$ there exists a constant $K_\eta\geq 1$ such that, for all $f\in\mathcal S_\eta(\R)$, $R\geq K_\eta \kappa_\eta(f)$,
\be\label{mu(tail>R)-formula}
\mu(\{g\in\GamG:\:|\Theta_f(g)|>R\})= \frac{2}{3}D(f) R^{-6} \big(1+O_\eta(\kappa_\eta(f)^{2\eta} R^{-{2\eta}}) \big)
\ee
where the implied constant depends only on $\eta$. 
\end{lem}

\begin{proof}
Recall the fundamental domain $\scrF_\Gamma$ in \eqref{fund-dom-H}, and define the subset
\be \scrF_{T}=\left\{(x+\i y,\phi;\vecxi,\zeta): x\in  [-\tfrac12,\tfrac12),\; y>T,\; \phi\in \times [0,\pi),\vecxi\in [-\tfrac12,\tfrac12)^2,\;\zeta\in [-\tfrac12,\tfrac12) 
\right\}.\label{siegelset}\ee
We note that $\scrF_1\subset\scrF_\Gamma\subset\scrF_{1/2}$. 
To simplify notation, set $\tilde\kappa=C_\eta \kappa_\eta(f)$.
We obtain an upper bound for \eqref{mu(tail>R)-formula} via Lemma \ref{lem-expansion-at-infinity},
\begin{equation}\label{mu(tail>R)-formula22}
\begin{split}
\mu(\{g\in\GamG:\:|\Theta_f(g)|>R\}) 
&\leq \mu(\{g\in\scrF_{1/2} :\:|\Theta_f(g)|>R\}) \\
&\leq \mu(\{g\in\scrF_{1/2} :\: y^{1/4} |f_{\phi}(-\theta y^{1/2})| + \tilde\kappa\,  y^{-(2\eta-1)/4} >R\}) .
\end{split}
\end{equation}
In particular, we have $y^{1/4} +  y^{-(2\eta-1)/4} >R/\tilde\kappa$ and $y>\ha$, and hence $y\geq c_\eta (R/\tilde\kappa)^4$ for a sufficiently small $c_\eta>0$.
Thus
\begin{multline}\label{mu(tail>R)-formula23}
\mu(\{g\in\GamG:\:|\Theta_f(g)|>R\}) \\
\leq \mu(\{g\in\scrF_{1/2} :\: y^{1/4} |f_{\phi}(-\theta y^{1/2})| + c_\eta^{-(2\eta-1)/4} \tilde\kappa (\tilde\kappa/R)^{2\eta-1} >R\}) .
\end{multline}
The same argument yields the lower bound
\begin{multline}\label{mu(tail>R)-formula24}
\mu(\{g\in\GamG:\:|\Theta_f(g)|>R\}) 
\geq \mu(\{g\in\scrF_{1} :\:|\Theta_f(g)|>R\}) \\
\geq \mu(\{g\in\scrF_{1} :\: y^{1/4} |f_{\phi}(-\theta y^{1/2})| -  c_\eta^{-(2\eta-1)/4} \tilde\kappa (\tilde\kappa/R)^{2\eta-1} >R\}) .
\end{multline}
The terms in \eqref{mu(tail>R)-formula22} and \eqref{mu(tail>R)-formula24} are of the form
\begin{equation}
I_{T}(\Lambda) = \mu(\{g\in\scrF_{T} :\: y^{1/4} |f_{\phi}(-\theta y^{1/2})| >\Lambda \}) ,
\end{equation}
where $T=\frac12$ or $1$ and $\Lambda=R- c_\eta^{-(2\eta-1)/4}\tilde\kappa (\tilde\kappa/R)^{2\eta-1}$ or $R+   c_\eta^{-(2\eta-1)/4}\tilde\kappa(\tilde\kappa/R)^{2\eta-1}$, respectively.
We have
\begin{equation}
I_{T}(\Lambda) = \int_{-\ha}^{\ha}\de\zeta\int_{-\ha}^{\ha}\de\xi_1\int_{-\ha}^{\ha}\de x\int_{0}^{\pi}\de \phi \int_{-\ha}^{\ha}\de\theta  \int_{y\geq\max(T, |f_\phi(-\theta y^{1/2})|^{-4}\Lambda^4)} \frac{\de y}{y^2} .
\end{equation}
By choosing the constant $K_\eta$ sufficiently large, we can ensure that $\Lambda \geq \kappa_\eta(f) \geq \kappa_{1/2}(f)=\sup_{w,\phi}|f_\phi(w)|(1+|w|)^{1/2}$.
Then $T \leq |f_\phi(-\theta y^{1/2})|^{-4}\Lambda^4$ and $\frac12\geq |w||f_\phi(w)|^2\Lambda^{-2}$, and
using the change of variables $y\mapsto w=-\theta y^{1/2}$, we obtain
\begin{align}
I_{T}(\Lambda) &=2\int_0^{\pi}\left(\int_{-\infty}^0\frac{\de w}{|w|^3}\int_{0}^{|w||f_\phi(w)|^2\Lambda^{-2}}\!\!\theta^2\de\theta+\int_{0}^\infty\frac{\de w}{|w|^3}\int_{-|w||f_\phi(w)|^2\Lambda^{-2}}^0\!\!\theta^2\de\theta\right)\!\de\phi\:\:\:\:\:\:\\
&=\frac{2}{3\Lambda^{6}}\int_0^{\pi}\int_{-\infty}^\infty|f_\phi(w)|^6 \de w\,\de\phi ,
\end{align}
where $\Lambda^{-6} = R^{-6} \big(1+ O_\eta((\tilde\kappa/R)^{2\eta})\big).$
\end{proof}

\subsection{Square integrability of $\Theta_f$ for $f\in L^2(\R)$.}
Although we defined  the Jacobi theta function in (\ref{def-Jacobi-Theta-function-on-G}, \ref{Jacobi-theta-sum-2}) assuming that $f$ is regular enough so that $f_\phi(w)$ decays sufficiently fast as $|w|\to\infty$ uniformly in $\phi$, we recall here that $\Theta_f$ is a well defined element of $L^2(\GamG)$ provided $f\in L^2(\R)$.

\begin{lem}\label{moments2,4}
Let $f_1,f_2,f_3,f_4:\R\to\mathbb{C}$ be Schwartz functions.
Then 
\begin{align}
\frac{1}{\mu(\GamG)}\int_{\GamG}\Theta_{f_1}(g)\overline{\Theta_{f_2}(g)}\de\mu(g)=&\int_{\R}f_1(u)\overline{f_2(u)}\de u\label{statement-Theta_f1Theta_f2}
\end{align}
and
\begin{multline}
\frac{1}{\mu(\GamG)}\int_{\GamG}\Theta_{f_1}(g)\overline{\Theta_{f_2}(g)}\Theta_{f_3}(g)\overline{\Theta_{f_4}(g)}\de\mu(g)\\
=\left(\int_{\R}f_1(u)\overline{f_2(u)}\de u\right)\left(\int_{\R}f_3(u)\overline{f_4(u)}\de u\right) + \left(\int_{\R}f_1(u)\overline{f_4(u)}\de u\right)\left(\int_{\R}\overline{f_2(u)}f_3(u)\de u\right).\label{statement-Theta_f1Theta_f2Theta_f3Theta_f4}
\end{multline}
\end{lem}
\begin{proof}
The statement \eqref{statement-Theta_f1Theta_f2} is a particular case of Lemma 7.2 in \cite{Marklof2002Duke}, while \eqref{statement-Theta_f1Theta_f2Theta_f3Theta_f4} follows from 
Lemma A.7 in \cite{Marklof2003ann}.
\end{proof}

\begin{cor}\label{cor-L2L4}
For every $f\in L^2(\R)$, the function $\Theta_f$ is a well defined element of $L^4(\GamG)$. Moreover
\be
\left\|\Theta_f\right\|_{L^2(\GamG)}^2=\mu(\GamG)\|f\|_{L^2(\R)}^2,
\ee
\be
\left\|\Theta_f\right\|_{L^4(\GamG)}^4=2\mu(\GamG)\|f\|_{L^2(\R)}^4 .
\ee
\end{cor}
\begin{proof}
Use Lemma \ref{moments2,4}, linearity in each of the $f_i$'s and density to get the desired statements for $f_1=f_2=f_3=f_4=f$.
\end{proof}

\subsection{Hermite expansion for $f_\phi$}\label{subsection:Hermite-f}

In this section we use the strategy of \cite{Marklof-1999} and find another representation for $\Theta_f$ in terms of Hermite polynomials.
We will use this equivalent representation in the proof of Theorem \ref{thm:1} in Section \ref{sec:pf-thm-1}.

Let  $H_k$ be the $k$-th Hermite polynomial
\begin{align} 
H_k(t)&=(-1)^k \e^{t^2}\frac{\de^k}{\de t^k}\e^{-t^2}
= k! \sum_{m=0}^{\lfloor{\frac{k}{2}}\rfloor} \frac{(-1)^m(2t)^{k-2m}}{m!(k-2m)!}.
\end{align}
Consider the classical Hermite functions 
\begin{align}
h_k(t)=(2^{k}k! \sqrt{\pi})^{-1/2}\e^{-\frac{1}{2}t^2}H_k(t). 
\end{align}
For our purposes, we will use a slightly different normalization for our Hermite functions, namely
\begin{align}
\psi_k(t)=\left(2\pi\right)^{\qua}h_k(\sqrt{2\pi}t)=(2^{k-\ha}k!)^{-1/2} H_k(\sqrt{2\pi}\,t)\e^{-\pi t^2}
\end{align}
The families $\{h_k\}_{k\geq 0}$ and $\{\psi_k\}_{k\geq 0}$ are both orthonormal bases for $L^2(\R,\de x)$.
Following \cite{Marklof-1999},  we can write 
\be f_\phi(t)=\sum_{k=0}^\infty\hat f(k)\e^{- i (2k+1)\phi/2}\psi_k(t)
\label{f_phi-Hermite}\ee
where
\begin{align}
\hat f(k)&=\langle f, \psi_k \rangle_{L^2(\R)},
\end{align}
are the Hermite coefficients of $f$ with respect to the basis $\{\psi_k\}_{k\geq0}$. The uniform bound
\be|\psi_k(t)|\ll 1\hspace{.4cm}\mbox{for all $k$ and all real $t$}\label{classical-bound-psi}\ee
is classical, see \cite{Szego}. It is shown in \cite{Thangavelu} that
\be
|\psi_k(t)|\ll\begin{cases}
\left( (2k+1)^{1/3}+|2\pi t^2-(2k+1)|\right)^{-1/4},& \pi t^2\leq 2k+1\\
\e^{-\gamma t^2},&\pi t^2>2k+1
\end{cases}\label{est-hermite-h_k(t)}
\ee
for some $\gamma>0$, where the implied constant does not depend on $t$ or $k$.
For small values of $t$ (relative to $k$) one has the more precise asymptotic
\be
\begin{split}
\psi_k(t)= & \frac{2^{3/4}}{\pi^{1/4}}
((2k+1)-2\pi t^2)^{-\qua}\cos\!\left(\tfrac{(2k+1)(2\theta-\sin\theta)-\pi}{4}\right)\\
& +O\!\left((2k+1)^{\ha}(2k+1-2\pi t^2)^{-\frac{7}{4}}\right)\label{asymptotic-psi_k(t)-small-t}
\end{split}
\ee
where $0\leq \sqrt{2\pi}t\leq (2k+1)^{\ha}-(2k+1)^{-\frac{1}{6}}$ and $\theta=\arccos\!\left(\sqrt{2\pi} t (2k+1)^{-1/2}\right)$.
It will be convenient to consider the normalized Hermite polynomials
\begin{align}\bar{H}_k(t)=(2^{k-\ha}k!)^{-1/2}H_k(\sqrt{2\pi}t)\end{align}
since they satisfy the antiderivative relation
\begin{align}
\int\bar
H_k(t)\de t&= 
(2\pi)^{-\ha}\frac{\bar H_{k+1}(t)}{(2k+2)^{1/2}}.\label{antiderivative-H}
\end{align}

\begin{lem}\label{lem-hat-f(k)}
Let $f:\RR\to\RR$ be of Schwartz class.
For every $k\geq 0$
\begin{align}
&|\hat f(k)|\ll_m \frac{1}{1+k^{m}}\hspace{.5cm}\mbox{for every $m>1$};\label{statement-lem-hat-f(k)}\end{align}
\end{lem}
\begin{proof}
We use integration by parts,
\begin{align}
(2\pi)^{1/2}\int f(t)\e^{-\pi t^2}\bar H_k(t)\de t=\frac{f(t)\e^{-\pi t^2} \bar H_{k+1}(t)}{(2k+2)^{1/2}}-\frac{1}{(2k+2)^{1/2}}\int (Lf)(t)\e^{-\pi t^2}\bar H_{k+1}(t)\de t,\label{int-by-parts}
\end{align}
where $L$ is the operator
\be (Lf)(t)=f'(t)-2\pi t f(t).\ee
Since the function $f$ is rapidly decreasing, the boundary terms vanish. Since $L f$ is also Schwarz class, we can iterate \eqref{int-by-parts} as many times as we want. Each time we gain a power $k^{-1/2}$. This fact yields \eqref{statement-lem-hat-f(k)}. 
\end{proof}

The following lemma allows us to approximate $f_\phi$ by $f$ 
 when $\phi$ is near zero. We will use this approximation in the proof of Theorem \ref{thm:1}.
We will use the shorthand $\scrE_f(\phi,t):=|f_\phi(t)-f(t)|$.

\begin{lem}\label{lem-E_f(phi,t)}
Let $f:\RR\to\RR$ be of Schwartz class, and $\sigma>0$. 
Then, for all $|\phi|<1$, $t\in\R$,
\begin{align}
&\scrE_f(\phi,t)\ll_\sigma\frac{|\phi|}{1+|t|^\sigma}.
\label{statement-lem-Ef(phi,t)}
\end{align}
\end{lem}
\begin{proof}
Assume $\phi\geq0$, the case $\phi\leq 0$ being similar. 
Write \be \e^{- i (2k+1)\phi/2}=1+O(k\phi\wedge 1)\ee
and by \eqref{f_phi-Hermite} we get
\be
\begin{split}
\scrE_f(\phi,t)
&=O\!\left(\sum_{k=0}^\infty |\hat f(k)(k\phi \wedge 1)\psi_k(t)|\right)\\
&=O\!\left(\phi \sum_{0\leq k\leq 1/\phi}k |\hat f(k)\psi_k(t)|\right)+O\!\left(\sum_{k> 1/\phi}|\hat f(k)\psi_k(t)|\right) .\label{lemma-f_phi-f-1a}
\end{split}
\ee
If $1/\phi<\frac{\pi t^2-1}{2}$ then by \eqref{est-hermite-h_k(t)}
\begin{align}
\phi \sum_{0\leq k\leq 1/\phi}k |\hat f(k)\psi_k(t)|\ll \phi \sum_{0\leq k\leq 1/\phi}k|\hat f(k)|\e^{-\gamma t^2}\ll\phi\, \e^{-\gamma t^2}\label{lemma-f_phi-f-2}
\end{align}
and, by (\ref{est-hermite-h_k(t)}, \ref{asymptotic-psi_k(t)-small-t}), 
\be
\begin{split}
\sum_{k> 1/\phi}|\hat f(k)\psi_k(t)| &\ll \sum_{1/\phi < k <\frac{\pi t^2-1}{2}} |\hat f(k)| \e^{-\gamma t^2}\\
&+\sum_{k\geq \frac{\pi t^2-1}{2}}|\hat f(k)|\left((2k+1)^{1/3}+|2k+1-\pi t^2|\right)^{-1/4} \\ &\ll_\sigma \phi\, \e^{-\gamma t^2}+\left(\frac{\pi t^2-1}{2}\right)^{-(\sigma+1)}\\
&\ll_\sigma \phi\, \e ^{-\gamma t^2}+\phi (1+|t|^{\sigma})^{-1}\ll_\sigma \phi (1+|t|^{\sigma})^{-1}.\label{lemma-f_phi-f-3}
\end{split}
\ee
If, on the other hand, $1/\phi\geq \frac{\pi t^2-1}{2}$, then 
\be
\begin{split}
\phi \sum_{0\leq k\leq 1/\phi}k |\hat f(k)\psi_k(t)|&\ll \phi \sum_{0\leq k<\frac{\pi t^2-1}{2}}k|\hat f(k)|\e^{-\gamma t^2}
\\&+\phi \sum_{\frac{\pi t^2-1}{2}\leq k\leq 1/\phi}k |\hat f(k)|\left((2k+1)^{1/3}+|2k+1-\pi t^2|\right)^{-1/4}\\
&\ll_\sigma \phi\, \e^{-\gamma t^2}+\phi(1+ |t|^{\sigma})^{-1}\ll_\sigma \phi(1+|t|^{\sigma})^{-1}\label{lemma-f_phi-f-3bis}
\end{split}
\ee
and
\begin{align}
\sum_{k> 1/\phi}|\hat f(k)\psi_k(t)| &\ll \sum_{k>1/\phi}|\hat f(k)|\e^{-\gamma t^2}\ll\phi\, \e^{-\gamma t^2}.\label{lemma-f_phi-f-5a}
\end{align}
Combining (\ref{lemma-f_phi-f-1a}, \ref{lemma-f_phi-f-2}, \ref{lemma-f_phi-f-3}, \ref{lemma-f_phi-f-3bis}, \ref{lemma-f_phi-f-5a}) we get the desired statement \eqref{statement-lem-Ef(phi,t)}.
\end{proof}

\subsection{The proof of Theorem \ref{thm:1}}\label{sec:pf-thm-1}
Recall the notation introduced in Section \ref{section:geodesic+horocycle+flows}. The automorphic function featured in the statement of Theorem \ref{thm:1} is the Jacobi theta function $\Theta_f$ defined in \eqref{def-Jacobi-Theta-function-on-G}.
\begin{proof}[Proof of Theorem \ref{thm:1}]
The fact that $\Theta_f\in C^\infty(\GamG)$ for smooth $f$ follows from \eqref{f_phi-Hermite} and the estimates for the Hermite functions as shown in \cite{Marklof-1999}, Section 3.1. Let us then prove the remaining part of the theorem. Recall that 
\be
S_N(x,\alpha;f)=\e^{-s/4}\Theta_f\!\left(x+i \e^{-s},0;\sve{\alpha}{0},0\right)=\e^{s/4}\Theta_f\!\left(n_+(x,\alpha)\Phi^{s}\right)
\ee
where $N=\e^{s/2}$.
Notice that 
\be
n_+(x,\alpha)n_-(u,\beta)\Phi^s=\left(x+\frac{u}{e^{2s}+u^2}+i\frac{\e^s}{\e^{2s}+u^2},\arctan(u \e^{-s}),\ve{\alpha+x \beta}{\beta},\ha\alpha\beta\right).\label{pf-thm1-1}
\ee
We need to estimate the difference between $\Theta_f\!\left(n_+(x,\alpha)n_-(u,\beta)\Phi^s\right)$ and $\Theta_f\!\left(n_+(x,\alpha)\Phi^s\right)$ and show it depends continuously on $n_-(u,\beta)\in H_-$. To this extent, it is enough to show continuity on compacta of $H_-$ and therefore we can assume that $u$ and $\beta$ are both bounded. To simplify notation, we assume without loss of generality $u> 0$. 
In the following, we will use the  bounds 
\begin{align}
\left(\frac{\e^s}{\e^{2s}+u^2}\right)^{1/4}&=\frac{1}{\sqrt N}+O\!\left(\frac{u^2}{N^{9/2}}\right)\label{pf-thm1-2}\\
\left(\frac{\e^s}{\e^{2s}+u^2}\right)^{1/2}&=\frac{1}{ N}+O\!\left(\frac{u^2}{N^5}\right)\label{pf-thm1-3}
\end{align}
and
\be
\begin{split}
& e\!\left(\ha (n-\beta)^2\left(x+\frac{u}{\e^{2s}+u^2}\right)+n(\alpha+x\beta)-\ha x\beta^2\right)\\
& =e\!\left(\ha n^2x+n\alpha\right)\left(1+O\!\left(\frac{u\, n^2}{N^4}\wedge 1\right)\right),\label{pf-thm1-4}
\end{split}
\ee
where all the implied constants are uniform for $N\geq1$ and for $n_-(u,\beta)$ in compacta of $H_-$.
From \eqref{pf-thm1-1} we get
\be
\begin{split}
&\Theta_f\!\left(n_+(x,\alpha)n_-(u,\beta)\Phi^s\right)=e\!\left(\ha\alpha\beta-\ha(\alpha+x\beta)\beta \right)\left(\frac{\e^s}{\e^{2s}+u^2}\right)^{\qua}\\
&\times\sum_{n\in\Z}f_{\arctan(u \e^{-s})}\!\left((n-\beta)\left(\frac{\e^s}{\e^{2s}+u^2}\right)^{\ha}\right)e\!\left(\ha(n-\beta)^2\left(x+\frac{u}{\e^{2s}+u^2}\right)+n(\alpha+x\beta)\right).
\end{split}
\ee
By using \eqref{pf-thm1-2} and \eqref{pf-thm1-4}, we obtain
 \be
\begin{split}
&\Theta_f\!\left(n_+(x,\alpha)n_-(u,\beta)\Phi^s\right)=
\left(\frac{1}{\sqrt N}+O\!\left(\frac{u^2}{N^{9/2}}\right)\right)
\\
&\times \sum_{n\in\Z}\left(\left|f\!\left((n-\beta)\left(\frac{\e^s}{\e^{2s}+u^2}\right)^{\ha}\right)\right|+
\mathcal E_f\!\left(\arctan\!\left(\frac{u}{N^2}\right),(n-\beta)\left(\frac{\e^s}{\e^{2s}+u^2}\right)^{1/2}\right)
\right)\\
&\times e\!\left(\ha n^2 x+n\alpha\right)\left(1+O\left(\frac{u\, n^2}{N^4}\wedge 1\right)\right),
\end{split}
\ee
where $\mathcal E_f$ is as in Lemma \ref{lem-E_f(phi,t)}.
We claim that
\be
\begin{split}
&\sum_{n\in\Z} f\!\left( (n-\beta)\left(\frac{\e^s}{\e^{2s}+u^2}\right)^\ha\right)e\!\left(\ha n^2 x+n\alpha \right)\\
&=\sum_{n\in\Z}f\!\left(\frac{n}{N}\right) e\!\left(\ha n^2 x+n\alpha \right)+O(\beta)+O\!\left(\frac{u^2}{N^{3}}\right). \label{pf-thm1-5}
\end{split}
\ee
Indeed, by the Mean Value Theorem, the fact that $f\in\SR$,  and \eqref{pf-thm1-3},  we have
\be
\begin{split}
&\left|\sum_{n\in\Z}\left( f\!\left( (n-\beta)\left(\frac{N^2}{N^4+u^2}\right)^{1/2}\right)-
f\!\left(\frac{n}{N}\right) \right) e\!\left(\ha n^2 x+n\alpha \right)\right|\\
&\ll\sum_{n\in\Z}\left|f'\left(\tau\right)\right|\left(O\!\left(\frac{\beta}{N}\right)+O\!\left(\frac{u^2n}{N^5}\right)\right) = O(\beta)+O\!\left(\frac{u^2}{N^3}\right)
\end{split}
\ee
where $\tau=\tau(u,\beta;n,N)$ belongs to the interval with endpoints $\frac{n}{N}$ and $(n-\beta)\left(\frac{N^2}{N^4+u^2}\right)^{1/2}$, and the implied constants are uniform in $N$ and in $u,\beta$ on compacta. This proves \eqref{pf-thm1-5}.

We require two more  estimates. The first uses \eqref{pf-thm1-3}:
\be
\begin{split}
&\sum_{n\in\Z} \left| f\!\left( (n-\beta)\left(\frac{\e^s}{\e^{2s}+u^2}\right)^\ha\right)\right| O\!\left(\frac{u n^2}{N^4}\wedge 1 \right)\\
&\ll \frac{u}{N}\sum_{|n|\leq N^2/\sqrt{u}}\left| f\!\left(\frac{n}{N}\right)\right|\left(\frac{n}{N}\right)^2\frac{1}{N}+\sum_{|n|>N^2/\sqrt{u}}\left| f\!\left(\frac{n}{N}\right)\right|\\
&=O\!\left(\frac{u}{N}\right)+O\!\left(N\int_{N/\sqrt{u}}^\infty |f(x)| \de x\right)=O\!\left(\frac{u}{N}\right).\label{pf-thm1-6} 
\end{split}
\ee
The second one uses \eqref{statement-lem-Ef(phi,t)} and (\ref{pf-thm1-2}, \ref{pf-thm1-3}):
\be
\begin{split}
&\left(\frac{\e^s}{\e^{2s}+u^2}\right)^{1/4}\sum_{n\in\Z} \mathcal E_f\!\left(\arctan\!\left(\frac{u}{N^2}\right),(n-\beta)\left(\frac{\e^s}{\e^{2s}+u^2}\right)^{1/2}\right)\\
&\ll\frac{1}{\sqrt{N}}\frac{u}{N^2}\sum_{n\in\Z}\frac{1}{1+\left(\frac{|n|}{N}\right)^2}=O\!\left(\frac{u}{N^{3/2}}\right). \label{pf-thm1-8}
\end{split}
\ee
Now, combining (\ref{pf-thm1-5}, \ref{pf-thm1-6},  \ref{pf-thm1-8}), we obtain
\begin{align}
\Theta_f\!\left(n_+(x,\alpha)n_-(u,\beta)\Phi^s\right)=\Theta_f\!\left(n_+(x,\alpha)\Phi^s\right)+O\!\left(\frac{u}{ N^{3/2}}\right)+O\!\left(\frac{\beta}{ N^{1/2}}\right).
\end{align}
This implies \eqref{statement-thm1-intro} with $E_f(n_-(u,\beta))=C\left(|\frac{u}{N}|+|\beta|\right) $ for some positive constant $C$. 
\end{proof}

\section{The automorphic function $\Theta_\chi$}\label{section:dyadic-representation-Theta-sharp-cutoff}

We saw in Corollary \ref{cor-L2L4} that $\Theta_f$ is a well defined element of $L^2(\GamG)$ if $f\in L^2(\R)$.  In this section we will consider the sharp cut-off function $f=\chi=\bm{1}_{(0,1)}$ and find a new series representation for $\Theta_\chi$ by using a dyadic decomposition of the cut-off function (Sections \ref{sec:dyadic}--\ref{sec:hermite2}).
We will find an explicit $\Gamma$-invariant subset $D\subset G$, defined in terms natural Diophantine conditions (Section \ref{sec:diophantine}), where this series is absolutely convergent. Moreover,
the coset space $\Gamma \backslash D$ is of full measure in $\GamG$. 
After proving Theorem \ref{thm:2} in Section \ref{subs:pf-thm-2}, we will show how it implies Hardy and Littlewood's classical approximate functional equation \eqref{ApproxFeq} (Section \ref{sec:HL}). Furthermore, we will use the explicit series representation for $\Theta_\chi$ to prove an analogue of Lemma \ref{mu(tail>R)} (Theorem \ref{tail-asymptotics-for-Theta_chi} in section \ref{subs:tail-asymptotic-Theta_chi}). A uniform variation of this result is shown in Section \ref{subs:unif-tail-bound}.

\subsection{Dyadic decomposition for $\Theta_\chi$}\label{sec:dyadic}
Let $\chi=\bm{1}_{(0,1)}$. 
Define the ``triangle'' function. 
\be
\Delta(w):=\begin{cases}
0&w\notin [\tfrac{1}{6},\tfrac{2}{3}]\\
72\left(x-\frac{1}{6}\right)^2&w\in [\tfrac{1}{6},\tfrac{1}{4}]\\
1-72\left(x-\frac{1}{3}\right)^2&w\in [\tfrac{1}{4},\tfrac{1}{3}]\\
1-18\left(x-\frac{1}{3}\right)^2&w\in [\tfrac{1}{3},\tfrac{1}{2}]\\
18\left(x-\frac{2}{3}\right)^2&w\in [\tfrac{1}{2},\tfrac{2}{3}]\\
\end{cases}\label{def-triangle-Delta}
\ee
Notice that
\be
\sum_{j=0}^{\infty}\Delta(2^j w)=\begin{cases}0&w\notin(0,\tfrac{2}{3})\\1&w\in(0,\tfrac{1}{3}]\\\Delta(w)&w\in[\tfrac{1}{3},\tfrac{2}{3}]\end{cases}
\ee
and hence 
\be \chi(w)=\sum_{j=0}^\infty\Delta(2^j w)+\sum_{j=0}^\infty\Delta(2^j(1-w)).\label{Chi=sums-of-Deltas-0}\ee
In other words, $\{\Delta(2^{j}\cdot),\Delta(2^{j'}(1-\cdot))\}_{j,j'\geq 0}$ is a partition of unity, see Figure \ref{fig:dyadic}.
\begin{center}
\begin{figure}[h!]
\hspace{.1cm}\includegraphics[width=16.4cm]{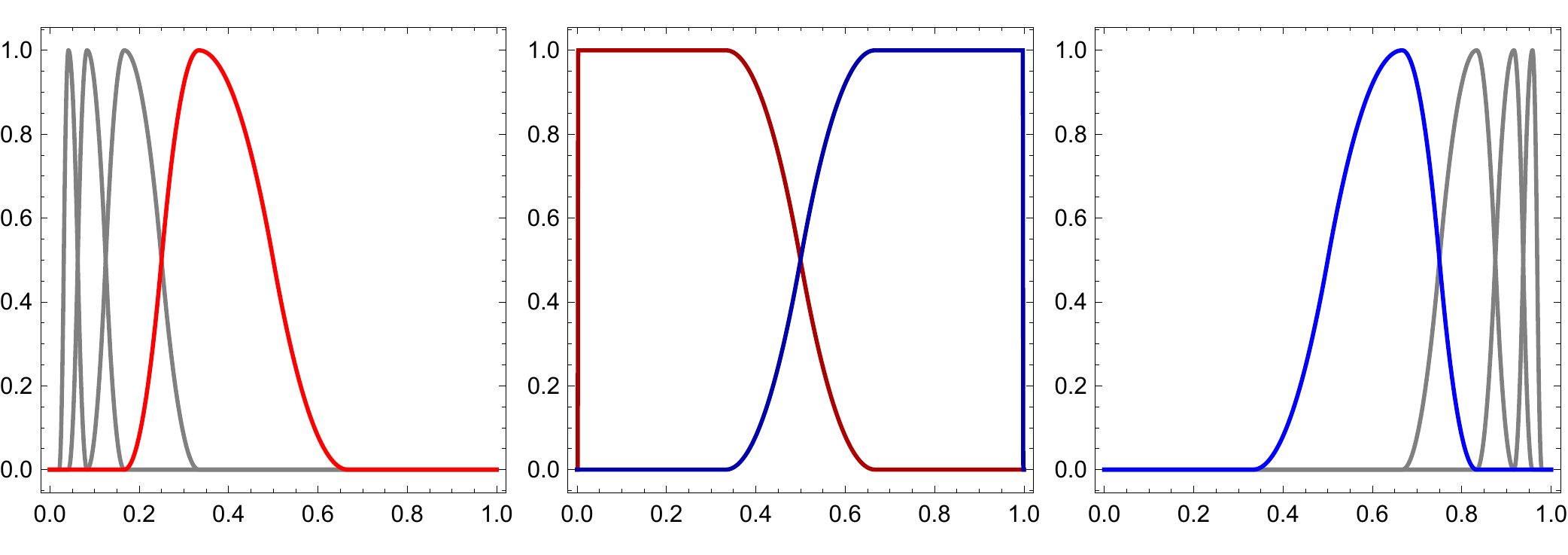}
\caption{Left: the functions $w\mapsto \Delta(w)$ (red) and  $w\mapsto \Delta(2^jw)$, $j=1,\ldots,3$ (gray). Center: the functions $w\mapsto \sum_{j=0}^\infty \Delta(2^jw)$  (red) and $w\mapsto \sum_{j=0}^\infty \Delta(2^j(1-w))$ (blue). Right: the functions $w\mapsto \Delta(1-w)$ (blue) and $w\mapsto \Delta(2^j(1-w))$, $j=1,\ldots,3$ (gray). }\label{fig:dyadic}
\end{figure}
\end{center}
Recall (\ref{Iwasawa-tlstr}, \ref{defR-projective}). We have
\be
2^{j/2}\Delta(2^j w)=[R(\tilde a_{2^{2j}};\bm 0,0)\Delta](w),\label{Chi=sums-of-Deltas-1}
\ee
\be
2^{j/2}\Delta(2^j(1-w))=[R(\tilde a_{2^{2j}};\sve{0}{1},0)\Delta_-](w),\label{Chi=sums-of-Deltas-2}
\ee
where  $\Delta_-(t):=\Delta(-t)$. 
Thus 
\be\chi(t)=\sum_{j=0}^\infty2^{-j/2}[R(\tilde a_{2^{2j}};\bm 0,0)\Delta](t)+\sum_{j=0}^\infty2^{-j/2}[R(\tilde a_{2^{2j}};\sve{0}{1},0)\Delta_-](t).\ee
Let us also write the partial sums
\begin{align}
\chi^{(J)}_L&=\sum_{j=0}^{J-1}2^{-j/2}[R(\tilde a_{2^{2j}};\bm 0,0)\Delta](t),\label{def-chi_L^J}\\
\chi^{(J)}_R&=\sum_{j=0}^{J-1}2^{-j/2}[R(\tilde a_{2^{2j}};\sve{0}{1},0)\Delta_-](t),\label{def-chi_R^J}
\end{align}
and $\chi^{(J)}=\chi^{(J)}_L+\chi^{(J)}_T$.  Consider 
 the following ``trapezoidal'' function:
\be
T^{\varepsilon,\delta}_{a,b}(w)=\begin{cases}
0&w\leq a-\varepsilon\\
\frac{2}{\varepsilon^2}(w-(a-\varepsilon))^2&a-\varepsilon<w\leq a-\frac{\varepsilon}{2}\\
1-\frac{2}{\varepsilon^2}(w-a)^2&a-\frac{\varepsilon}{2}<w\leq a\\
1&a<w< b\\
1-\frac{2}{\delta^2}(w-b)^2&b\leq w< b+\frac{\delta}{2}\\
\frac{2}{\delta^2}(w-(b+\delta))^2&b+\frac{\delta}{2}\leq w< b+\delta\\
0&w\geq b+\delta,\end{cases}\label{def-trapezoidal-function}
\ee
see Figure \ref{fig:trapezoid}. \begin{figure}[h!]
\begin{center}
\includegraphics[width=10cm]{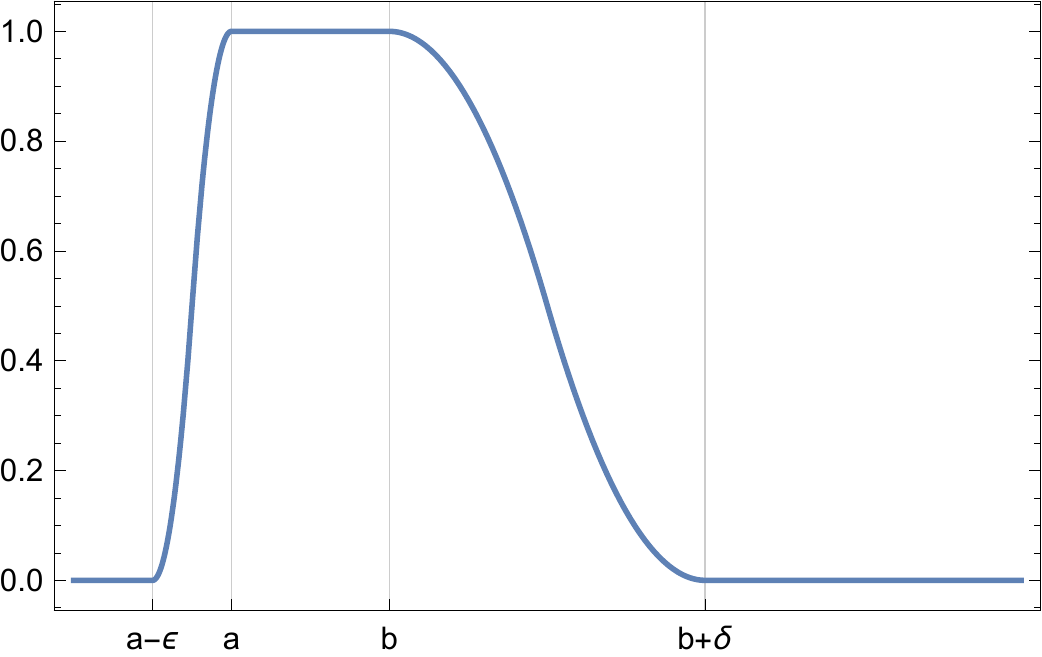}
\end{center}
\caption{The function $w\mapsto T^{\varepsilon,\delta}_{a,b}(w)$.}
\label{fig:trapezoid}
\end{figure}

Later we will use the notation $I_1=[a-\varepsilon,a-\varepsilon/2]$, $I_2=[a-\varepsilon/2,a]$, $I_3=[a,b]$, $I_4=[b,b+\delta/2]$, $I_5=[b+\delta/2,b+\delta]$ and $f_i= T^{\varepsilon,\delta}_{a,b}|_{I_i}$ for $i=1,\ldots, 5$. The functions $\chi$, $\chi^{(J)}_L$, $\chi^{(J)}_R$, $\chi^{(J)}$, $\Delta$, $\Delta_-$ are all special cases of \eqref{def-trapezoidal-function}, with parameters as in the table below.
\bgroup
\def\arraystretch{1.3}
\begin{center}
\begin{tabular}{|c|c|c|c|c|}\hline 	 & $a$ & $b$ & $\varepsilon$ & $\delta$ \\\hline $\chi$ & 0 & 1 & 0 & 0 \\\hline $\chi_L^{(J)}$ & $\frac{1}{3\cdot 2^{J-1}}$  & $\frac{1}{3}$  &  $\frac{1}{6\cdot 2^{J-1}}$& $\frac{1}{3}$  \\\hline $\chi_R^{(J)}$ & $\frac{2}{3}$& $1-\frac{1}{3\cdot 2^{J-1}}$ &$\frac{1}{3}$  & $\frac{1}{6\cdot 2^{J-1}}$    \\\hline $\chi^{(J)}$ &  $\frac{1}{3\cdot 2^{J-1}}$  &  $1-\frac{1}{3\cdot 2^{J-1}}$ &  $\frac{1}{6\cdot 2^{J-1}}$  &  $\frac{1}{6\cdot 2^{J-1}}$ \\\hline $\Delta$ & $\frac{1}{3}$ & $\frac{1}{3}$ & $\frac{1}{6}$ & $\frac{1}{3}$ \\\hline $\Delta_-$ & $\frac{2}{3}$ & $\frac{2}{3}$ & $\frac{1}{3}$ & $\frac{1}{6}$ \\\hline\end{tabular}
\end{center}
\egroup

\begin{lem}\label{lem-general-trapezoid} There exists a constant $C$ such that 
\be\kappa_2(T^{\varepsilon,\delta}_{a,b})=\sup_{\phi,w}\left|(T^{\varepsilon,\delta}_{a,b})_\phi(w)\right|(1+|w|)^2\leq C(\varepsilon^{-1}+\delta^{-1})\label{statement-lemma-trapezoid}\ee
for all $\varepsilon,\delta\in(0,1]$ and $0\leq a\leq b\leq 1$. 
\end{lem}

By adjusting $C$, the restriction of $a,b$ to $[0,1]$ can be replaced by any other bounded interval; we may also replace the upper bound on $\varepsilon,\delta$ by an arbitrary positive constant.
The lemma shows in particular that $T^{\varepsilon,\delta}_{a,b}\in\mathcal S_2(\R)$ for $\varepsilon,\delta>0$ and $a\leq b$. 
Its proof requires the following two estimates.

\begin{lem}[Second derivative test for exponential integrals, see Lemma 5.1.3 in \cite{Huxley-book}]\label{lemma-Huxley}
Let $\varphi(x)$ be real and twice differentiable on the open interval $(\alpha,\beta)$ with $\varphi''(x)\geq \lambda>0$ on $(\alpha,\beta)$. Let $f(x)$ be real and let $V=V_\alpha^\beta(f)+\max_{\alpha\leq x\leq \beta}|f(x)|$, where $V_\alpha^\beta(g)$ denotes the total variation of $f(x)$ on the closed interval $[\alpha,\beta]$. Then
\be
\left|\int_\alpha^\beta e(\varphi(x)) f(x)\de x\right|\leq\frac{4V}{\sqrt{\pi\lambda}}.
\ee
\end{lem}

\begin{lem}\label{lemma-|f_phi(w)|<const}
Let $f$ be  real and compactly supported on $[\alpha,\beta]$, with $V_\alpha^\beta(f)<\infty$. Then, for every $w,\phi\in\R$,
\be
|f_\phi(w)|\leq \max\{ 3V,2 I\},\label{statement-lemma-|f_phi(w)|<const}
\ee
where $V=V_\alpha^\beta(f)+\max_{\alpha\leq x\leq \beta}|f(x)|$ and $I=\int_{\alpha}^\beta |f(x)|\de x$.
\end{lem}
\begin{proof}
If $\phi\equiv0\bmod\pi$, then $|f_\phi(w)|=|f(w)|\leq V$. If $\phi\equiv\frac{\pi}{2} \bmod\pi$, then $|f_\phi(w)|\leq I$.
If $0<(\phi\bmod\pi)<\frac{\pi}{4}$, let $\varphi(x)=\frac{\ha(x^2+w^2)\cos\phi-wx}{\sin\phi}$ satisfies the hypothesis of Lemma \ref{lemma-Huxley} with $\lambda=\cot\phi$ and we get 
\be
|f_\phi(w)|\leq |\sin\phi|^{-\ha} \frac{4 V}{\sqrt{\pi\cot\phi}}=\frac{4V}{\sqrt{\pi |\cos\phi|}}\leq 3V.
\ee
The case $\frac{3\pi}{4}\leq (\phi \bmod\pi)<\pi$ yields the same bound by considering the complex conjugate of the integral before applying Lemma \ref{lemma-Huxley}. If $\frac{\pi}{4}\leq (\phi\bmod\pi)\leq \frac{3\pi}{4}$ then we have the trivial bound
\be
|f_\phi(w)|\leq |\sin\phi|^{-\ha}\int_{\alpha}^\beta |f(x)|\de x\leq 2\int_{\alpha}^\beta |f(x)|\de x.
\ee
Combining all the estimates we get \eqref{statement-lemma-|f_phi(w)|<const}.
\end{proof}

\begin{proof}[Proof of Lemma \ref{lem-general-trapezoid}]
 If $\phi\equiv0 \bmod \pi$, then $|(T^{\varepsilon,\delta}_{a,b})_\phi(w)|=|T^{\varepsilon,\delta}_{a,b}(w)|$ and the estimate
\be\sup_{w}\left|(T^{\varepsilon,\delta}_{a,b})_\phi(w)\right|(1+|w|)^2=O(1) \label{estimate-sup-w-trapezoid}\ee
holds trivially. 

If $\phi\equiv\tfrac{\pi}{2}\bmod 2\pi$, then by \eqref{def-R-tilde-k-f}, the function $f_\phi(w)=e(-\sigma_\phi/8)\int_{\R}e(- w w')f(w')\de w'$ is (up to a phase factor) the Fourier transform of $f$,
which reads for $w\neq 0$:
\begin{align}
\begin{split}
(T^{\varepsilon,\delta}_{a,b})_{\phi}(w)=\frac{ie(-\sigma_\phi/8)}{2 \pi^3 w^3 \varepsilon^2 \delta^2}\left(\varepsilon^2 e(-w(b+\delta))(1-e(w\delta/2))^2-\delta^2 e(-a w)(1-e(w\varepsilon/2))^2\right),\label{pf-lemma-trapezoid-sin^3-new1}
\end{split}
\end{align}
and for $w=0$: $(T^{\varepsilon,\delta}_{a,b})_{\phi}(0)=e(-\sigma_\phi/8)\frac{2b-2a+\varepsilon+\delta}{2}$.

Similarly, if $\phi\equiv-\tfrac{\pi}{2}\bmod 2\pi$, then $f_\phi(w)=e(-\sigma_\phi/8)\int_{\R}e( w w')f(w')\de w'$, and formula \eqref{pf-lemma-trapezoid-sin^3-new1} holds with $w$ replaced by $-w$.

We use the bound
\begin{equation} \label{u1}
|1-e(x)|^2\leq 2 |1-e(x)| \leq  4\pi |x| 
\end{equation}
applied to $x=w\delta/2$ and $x=w\varepsilon/2$
to conclude that, for $\phi\equiv\tfrac{\pi}{2}\bmod \pi$, 
\be
\begin{split}
|(T^{\varepsilon,\delta}_{a,b})_\phi(w)| & \ll |w|^{-2} \left( \delta^{-1}+\varepsilon^{-1}\right).
\end{split}
\ee
This gives the desired bound for $|w|\geq 1$. For $|w|<1$, we employ instead of \eqref{u1}
\begin{equation}
|1-e(x)|^2\leq 4\pi^2 |x|^2, 
\end{equation}
which shows that $|(T^{\varepsilon,\delta}_{a,b})_\phi(w)|=O(1)$ in this range.

For all other $\phi$ (i.e. such that $\sin\phi,\cos\phi\neq0$) we apply twice the identity
\be\int_a^b \e^{g(v)}f(v)\de v=\left[\e^{g(v)}\frac{f(v)}{g'(v)}\right]_{a}^{b}-\int_a^b \e^{g(v)}\left(\frac{f(v)}{g'(v)}\right)'\de v,\label{identity-by-parts}\ee
where $g(v)=2\pi i \frac{\ha(w^2+v^2)\cos\phi-wv}{\sin\phi}$. We have
\begin{multline}
|(T^{\varepsilon,\delta}_{a,b})_\phi(w)|=|\sin\phi|^{-\ha}\!\left| \sum_{j=1}^5 \int_{I_j}\e^{g(v)}f_j(v)\de v\right| =\frac{|\sin\phi|^{3/2}}{4\pi^2}\\
\times\!\left| \sum_{j=1}^5\int_{I_j}\e^{g(v)}\frac{3\cos^2\phi f_j(v)+3\cos\phi(w-v\cos\phi)f_j'(v)+(w-v\cos\phi)^2f_j''(v)}{ (w- v \cos\phi)^4}
\de v\right|. 
\label{pf-lemma-trapezoid-sin^3/2}
\end{multline}
Let us estimate the integrals in \eqref{pf-lemma-trapezoid-sin^3/2}. 
Consider the range $|w|\geq 3$ first. The bounds
\be
|f_j(v)|\leq 	\begin{cases}1& v\in[a-\varepsilon,b+\delta];\\0&\mbox{otherwise},\end{cases}
\ee
\be\label{first-deri}
|f_j'(v)| \leq \begin{cases} \frac{2}{\varepsilon}&v\in[a-\varepsilon,a],\\
\frac{2}{\delta}&v\in[b,b+\delta],\\
0&\mbox{otherwise},\end{cases}
\ee
and
\be
|f_j''(v)|\leq \begin{cases} \frac{4}{\varepsilon^2}&v\in[a-\varepsilon,a];\\
\frac{4}{\delta^2}&v\in[b,b+\delta];\\
0&\mbox{otherwise},\end{cases}
\ee
imply that
\be
\sum_{j=1}^5\int_{I_j}\frac{|3\cos^2\phi f_j(v)|}{ (w- v \cos\phi)^4}\de v \ll \int_{a-\varepsilon}^{b+\delta}\frac{\de v}{ (w- v \cos\phi)^4}\ll \frac{1}{w^4},
\ee
\be
\sum_{j=1}^5\int_{I_j}\frac{\left|3\cos\phi(w-v\cos\phi)f_j'(v)\right|}{ (w- v \cos\phi)^4}\de v\ll \int_{I_1\sqcup I_2}\frac{\varepsilon^{-1}\de v}{ |w- v \cos\phi|^3} +\int_{I_4\sqcup I_5}\frac{\delta^{-1}\de v}{ |w- v \cos\phi|^3} \ll \frac{1}{|w|^3},
\ee
\be
\sum_{j=1}^5\int_{I_j}\frac{|f_j''(v)|}{ (w- v \cos\phi)^2}\de v\ll \int_{I_1\sqcup I_2}\frac{\varepsilon^{-2}\de v}{ (w- v \cos\phi)^2} +\int_{I_4\sqcup I_5}\frac{\delta^{-2}\de v}{ (w- v \cos\phi)^2}\\
\ll\frac{1}{w^2}(\varepsilon^{-1}+\delta^{-1}).
\ee
Therefore, for $|w|\geq 3$, we have
\be
|(T^{\varepsilon,\delta}_{a,b})_\phi(w)|\ll\frac{1}{w^2}(\varepsilon^{-1}+\delta^{-1})
\ee
uniformly in all variables.
For $|w|< 3$ we apply Lemma \ref{lemma-|f_phi(w)|<const}, which yields 
\be
|(T_{a,b}^{\varepsilon,\delta})_\phi(w)|=O(1)
\ee
since in view of \eqref{first-deri} the total variation of $T_{a,b}^{\varepsilon,\delta}$ is uniformly bounded.
\end{proof}

\begin{cor}\label{cor-Theta_Delta}
The series defining $\Theta_{\Delta}(g)$ and $\Theta_{\Delta_-}(g)$ converge absolutely and uniformly on compacta in $G$.
\end{cor}
\begin{proof}
 We saw that $\Delta$ and $\Delta_-$ are of the form $T_{a,b}^{\varepsilon,\delta}$ with $\varepsilon,\delta>0$. The statement follows from  Lemma \ref{lem-general-trapezoid}. 
\end{proof}
Formul\ae\, (\ref{Chi=sums-of-Deltas-0}, \ref{Chi=sums-of-Deltas-1}, \ref{Chi=sums-of-Deltas-2}) motivate the following definition of $\Theta_\chi$: 
\bey 
\Theta_{\chi}(g)=\sum_{j=0}^\infty 2^{-j/2}\Theta_\Delta(\Gamma g(\tilde a_{2^{2j}};\bm0,0))+
\sum_{j=0}^\infty2^{-j/2}\Theta_{\Delta_-}(\Gamma g(1;\sve{0}{1},0)(\tilde a_{2^{2j}};\bm0,0)).\label{def-Theta-chi}
\eey
Each term in the above is a Jacobi theta function and, by Corollary \ref{cor-Theta_Delta}, is $\Gamma$-invariant (cf. Section \ref{sec:JTF}). 
We will show that the series \eqref{def-Theta-chi} defining $\Theta_\chi( g)$ is absolutely convergent for an explicit, $\Gamma$-invariant subset of $G$. This set projects onto a full measure set of $\GamG$. This means that we are only allow to write $\Theta_\chi(\Gamma g)$ only for almost every $g$. Therefore
$\Theta_\chi$ is an almost everywhere defined automorphic function on the homogeneous space $\GamG$.

\subsection{Hermite expansion for $\Delta_\phi$}\label{sec:hermite2}
We will use here the  notations from Section \ref{subsection:Hermite-f}.

\begin{lem}\label{lem-hat-Delta(k)}
Let $\Delta:\RR\to\RR$ be the ``triangle'' function \eqref{def-triangle-Delta}.
For every $k\geq 0$
\begin{align}
&|\hat \Delta(k)|\ll \frac{1}{1+k^{3/2}}.\label{statement-lem-hat-Delta(k)}
\end{align}
\end{lem}
\begin{proof}

Repeat the argument in the proof of Lemma \ref{lem-hat-f(k)} 
with $\Delta$ in place of $f$. In this case we can apply \eqref{int-by-parts} three times, and we get \eqref{statement-lem-hat-Delta(k)}.
\end{proof}
\begin{remark}\label{rk:non-optimal-1}
The estimate \ref{statement-lem-hat-Delta(k)} 
 is not optimal. One can get an additional  $O(k^{-1/4})$ saving by applying \eqref{asymptotic-psi_k(t)-small-t} to the boundary terms after the three integration by parts. Since the additional power saving does not improve our later results, we will simply use \eqref{statement-lem-hat-Delta(k)}. 
 \end{remark}

The following lemma allows us to approximate 
$\Delta_\phi$ by $\Delta$ when $\phi$ is near zero. We will use this approximation in the proof of Theorem \ref{thm:2} in Section \ref{subs:pf-thm-2}.

\begin{lem}\label{lem-E_Delta(phi,t)}
Let $\Delta:\RR\to\RR$ be the ``triangle'' function \eqref{def-triangle-Delta} 
and let  $\mathcal E_\Delta(\phi,t)=|\Delta_\phi(t)-\Delta(t)|$.
For every $|\phi|<\frac{1}{6}$ and every $t\in\R$ we have
\begin{align}
&\mathcal E_\Delta(\phi,t)\ll\begin{cases}
|\phi|^{3/4},& |t|\leq 2;\\
\\
\displaystyle{\frac{|\phi|^{3/2}}{1+|t|^2}},& |t|>2.
\end{cases}\label{statement-lem-E(phi,t)}
\end{align}
\end{lem}
\begin{proof}
Assume $\phi\geq0$, the case $\phi\leq 0$ being similar. 
The estimate \eqref{statement-lem-E(phi,t)} for $|t|>2$ follows from the proof of Lemma  \ref{lem-general-trapezoid} (see \eqref{pf-lemma-trapezoid-sin^3/2}) and the fact that $\Delta$ is compactly supported. 
Let us then consider the case $|t|<2$. 
We get
\begin{align}
\mathcal E_\Delta(\phi,t)
&=O\!\left(\phi \sum_{0\leq k\leq 1/\phi}k |\hat \Delta(k)\psi_k(t)|\right)+O\!\left(\sum_{k> 1/\phi}|\hat \Delta(k)\psi_k(t)|\right).\label{lemma-f_phi-f-1}
\end{align}
Since $|\phi|<\frac{1}{6}$, the inequality $1/\phi\geq\frac{\pi t^2-1}{2}$ is satisfied.
Therefore, by \eqref{est-hermite-h_k(t)} and Lemma \ref{lem-hat-Delta(k)},  
\be
\begin{split}
\phi \sum_{0\leq k\leq 1/\phi}k |\hat \Delta(k)\psi_k(t)|&\ll \phi \sum_{0\leq k<\frac{\pi t^2-1}{2}}k|\hat \Delta(k)|e^{-\gamma t^2}
\\&+\phi \sum_{\frac{\pi t^2-1}{2}\leq k\leq 1/\phi}k |\hat \Delta(k)|\left((2k+1)^{1/3}+|2k+1-\pi t^2|\right)^{-1/4}\\
&\ll \phi\e^{-\gamma t^2}+\phi\sum_{1\leq k\leq 1/\phi} k^{-3/4}\\
&\ll \phi^{3/4}
\end{split}
\ee
and
\begin{align}
\sum_{k> 1/\phi}|\hat \Delta(k)\psi_k(t)| &\ll \sum_{k>1/\phi}
k^{-7/4}\ll\phi^{3/4}
.\label{lemma-f_phi-f-5}
\end{align}
Combining \eqref{lemma-f_phi-f-1}-\eqref{lemma-f_phi-f-5} we get the desired statement \eqref{statement-lem-E(phi,t)}.
\end{proof}
\begin{remark}
The statement of Lemma \ref{lem-E_Delta(phi,t)} is not optimal. The estimate \eqref{statement-lem-E(phi,t)} could be improved for $|t|<2$ to $O(|\phi|\log(1/|\phi|)$ by using a stronger version of Lemma \ref{lem-hat-Delta(k)}, see Remark \ref{rk:non-optimal-1}. Since this improvement is not going to affect our results,  we are content with \eqref{statement-lem-E(phi,t)}.
\end{remark}

\subsection{Divergent orbits and Diophantine conditions}\label{sec:diophantine}
In this section we recall a well-known fact relating the excursion of divergent geodesics into the cusp of $\GamG$ and the Diophantine properties of the limit point.

A real number $\omega$ is said to be \emph{Diophantine of type $(A,\kappa)$} for  $A>0$  and $\kappa\geq 1$ if 
\be\left|\omega-\frac{p}{q}\right|>\frac{A}{q^{1+\kappa}}\ee
for every $p,q\in\Z$, $q\geq 1$.
We will denote by $\mathcal D(A,\kappa)$ the set of such $\omega$'s, and by $\mathcal D(\kappa)$ the union of $\mathcal D(A,\kappa)$ for all  $A>0$.
It is well known that for every $\kappa>1$, the set $\mathcal D(\kappa)$
has full Lebesgue measure.
The elements of $\mathcal D(1)$ are 
called \emph{badly approximable}. 
The set $\mathcal D(1)$  
has zero Lebesgue measure but  Hausdorff measure 1.

If we consider the action of $\sltr$ on $\RR$, seen as the boundary of $\frak H$, then for every $\kappa\geq 1$ the set $\mathcal D(\kappa)$ is $\sltz$-invariant:
\begin{lem}\label{lem-D(kappa)-is-Gamma-invariant}
Let $\kappa\geq 1$ and  $\omega\in\mathcal D(\kappa)$. 
Then for every $M=\sma{a}{b}{c}{d}\in\SL(2,\Z)$, $M\omega=\frac{a\omega+b}{c\omega+d}\in\mathcal D(\kappa)$.
\end{lem}
\begin{proof}
(This is standard.) It is enough to check that the claim holds for the generators $\sma{1}{1}{0}{1}$ and $\sma{0}{-1}{1}{0}$. For the first one the statement is trivial. For the second, it suffices to show $\omega\in\mathcal D(\kappa)$ is equivalent to $\omega^{-1}\in\mathcal D(\kappa)$. Assume without loss of generality $0<\omega<1$. Suppose first $\omega^{-1}\in\mathcal D(A,\kappa)$ for some $A>0$. Then, for $0<p\leq q$, 
\begin{equation}
\left|\omega - \frac{p}{q} \right| = \frac{\omega p}{q} \left|\omega^{-1} - \frac{q}{p} \right| 
> \frac{A \omega}{q p^{\kappa}} \geq  \frac{A \omega}{q^{1+\kappa}} .
\end{equation}
For $p\not\in(0,q]$, $|\omega - \frac{p}{q}| \geq \min(\omega,1-\omega)$.
We have thus proved $\omega\in\mathcal D(\kappa)$. To establish the reverse implication, suppose $\omega\in\mathcal D(A,\kappa)$ for some $A>0$. Then, for $0< q\leq p \lceil \omega^{-1} \rceil$, 
\begin{equation}
\left|\omega^{-1} - \frac{q}{p} \right| = \frac{q}{\omega p} \left|\omega - \frac{p}{q} \right| 
> \frac{A}{\omega p q^{\kappa}} \geq  \frac{A}{\omega  \lceil \omega^{-1} \rceil^{\kappa} p^{1+\kappa}} .
\end{equation}
Again we have a trivial bound for the remaining range $q\notin (0,p \lceil \omega^{-1} \rceil]$. This shows that $\omega^{-1}\in\mathcal D(\kappa)$.
\end{proof}

\begin{lem}\label{lem:diophantine-visit-cusp}
Let $x\in \mathcal D(A,\kappa)$ 
 for some $A\in(0,1]$ and $\kappa\geq1$.  
Define
\begin{align}\label{3.39}
z_s(x,u)&=\ma{1}{x}{0}{1}\ma{1}{0}{u}{1}\ma{\e^{-s/2}}{0}{0}{\e^{s/2}}i=x+\frac{u}{\e^{2s}+u^2}+i\frac{\e^{s}}{\e^{2s}+u^2}.
\end{align}
Then, for $s\geq 0$ and $u\in\R$,
\be
\begin{split}\label{statement-lem:diophantine-visit-cusp}
\sup_{M\in \sltz}\Im(M z_s(x,u)) & \leq A^{-\frac{2}{\kappa}} \e^{-(1-\frac{1}{\kappa})s}\, W(u\e^{-s}) \\
& \leq A^{-\frac{2}{\kappa}} \e^{-(1-\frac{1}{\kappa})s}\, W(u) ,
\end{split}
\ee
with
\begin{equation}
W(t):=  1+\ha \left(t^2+|t|\sqrt{4+t^2}\right).
\end{equation}
\end{lem}
\begin{proof}
Let us set $y:=\e^{-s}\leq 1$. The supremum in \eqref{statement-lem:diophantine-visit-cusp} is achieved when $M z_s(x,u)$ belongs to the fundamental domain $\mathcal F_{\sltz}$. Then
\be\mbox{either}\hspace{.3cm}\frac{\sqrt 3}{2}\leq\Im(M z_s(x,0))< 1,\hspace{.3cm}\mbox{or}\hspace{.3cm} \Im(M z_s(x,0))\geq 1.\ee
In the first case we have the obvious bound $\Im(Mz_s(x,0))\leq 1$. In the second case, write $M=\sma{a}{b}{c}{d}$. If $c=0$, then $\Im(Mz_s(x,0))=y\leq 1$. If $c\neq 0$,
\be
\Im(Mz_s(x,0))=\frac{y}{(cx+d)^2+c^2 y^2}\geq1.\label{pf-lem-dioph-cusp-1}
\ee
This implies that $(cx+d)^2/y\leq 1$ and $c^2 y\leq 1$. 
The first inequality yields
\be
y\geq A^2 |c|^{-2\kappa},
\ee and therefore we have
\be\left(\frac{
A^2}{y}\right)^{1/2\kappa}\leq |c|\leq  \left(\frac{1}{y}\right)^{1/2}.\ee
This means that 
\begin{align}
1
\leq\Im\left(\sma{a}{b}{c}{d}z\right)=\frac{y}{(cx+d)^2+c^2 y^2}\leq \frac{1}{c^2y}\leq 
A^{-\frac{2}{\kappa}}
y^{-1+\frac{1}{\kappa}}.
\end{align}
This proves the lemma for $u=0$. 

Let us now consider the case of general $u$. 
Let us estimate the hyperbolic distance between $\Gamma z_s(x,0)$ and $\Gamma z_s(x,u)$ on $\Gamma\backslash {\frak H}$, which is
\begin{equation}
\dist_{\Gamma\backslash {\frak H}} (\Gamma z_s(x,u),\Gamma z_s(x,0)) 
:=\inf_{M\in\Gamma} \dist_{\frak H}(M z_s(x,u),z_s(x,0)).
\end{equation}
We compute
\be
\begin{split}
\dist_{\frak H}(z_s(x,u),z_s(x,0)) 
&=\mathrm{arcosh}\!\left(1+\frac{(\Re(z(u)-z(0))^2+(\Im(z(u)-z(0))^2}{2\,\Im(z(u))\Im(z(0))}\right)\\
&=\mathrm{arcosh}\!\left(1+\ha u^2 y^2\right),
\end{split}
\ee
and hence
\be
\begin{split}
\dist_{\Gamma\backslash {\frak H}} (\Gamma z_s(x,u),\Gamma z_s(x,0)) 
& \leq \mathrm{arcosh}\!\left(1+\ha u^2 y^2\right) \\
& = \log \left(1+\ha u^2 y^2+\ha |u| y\sqrt{4+u^2 y^2}\right).
\end{split}
\ee
Now,
\be
\begin{split}
\sup_{M\in \sltz}\Im(M z_s(x,u)) & \leq \sup_{M\in \sltz}\Im(M z_s(x,0)) \e^{\dist_{\Gamma\backslash {\frak H}} (\Gamma z_s(x,u),\Gamma z_s(x,0)) } \\
& \leq \sup_{M\in \sltz}\Im(M z_s(x,0)) \left(1+\ha u^2 y^2+\ha |u| y\sqrt{4+u^2 y^2}\right),
\end{split}
\ee
and the claim follows from the case $u=0$.
\end{proof}

We will in fact use the following backward variant of Lemma \ref{lem:diophantine-visit-cusp}.

\begin{lem}\label{lem:diophantine-visit-cusp22}
Let $x\in\R$, $u\in\R-\{0\}$ such that $x+\frac1u\in \mathcal D(A,\kappa)$ for some $A\in(0,1]$ and $\kappa\geq1$. Then, for $s\geq 2\log(1/|u|)$,
\begin{equation}\label{statement-lem:diophantine-visit-cusp22}
\sup_{M\in \sltz}\Im(M z_{-s}(x,u)) \leq
\begin{cases}
\max(\frac{1}{2|u|} , 2) & \text{if $0\leq s\leq 2\log^+(1/|u|)$,} \\
 A^{-\frac{2}{\kappa}} u^{1-\frac{1}{\kappa}}\, \e^{-(1-\frac{1}{\kappa})s}\, W(u) & \text{if $s\geq 2\log^+(1/|u|)$},
\end{cases}
\end{equation}
with $\log^+(x):=\max(\log x, 0)$.
\end{lem}
\begin{proof}
We have
\begin{equation}
\ma{1}{x}{0}{1}\ma{1}{0}{u}{1}\Phi^{-s}i=\ma{1}{x+\frac{1}{u}}{0}{1}\ma{1}{0}{-u}{1}\Phi^{\tau}i,\label{backward-forward-t-tau}
\end{equation}
with $\tau=s+\log^+(u^2)$. In the range $s\geq 2\log(1/|u|)$, we may therefore apply Lemma \ref{lem:diophantine-visit-cusp} with $\tau$ in place of $s$, and $x+\frac{1}{u}$ in place of $x$.

In the range $0\leq s\leq 2\log^+(1/|u|)$ we have $\Im(z_{-s}(x,u))=\frac{1}{\e^s u^2+\e^{-s}}\geq \frac12$. If  $\Im(z_{-s}(x,u))\geq 1$, then the maximal possible height is $\frac{1}{2|u|}$. If on the other hand $\frac12\leq \Im(z_{-s}(x,u))< 1$, then $\Im(M z_{-s}(x,u))\leq 2$ for all $M\in\Gamma$.
\end{proof}

\subsection{Proof of Theorem \ref{thm:2}}\label{subs:pf-thm-2}
Let us now give a more precise formulation of Theorem \ref{thm:2} from the Introduction.
\begin{theorem}\label{thm:2-rephrased}
Fix $\kappa>1$. For  $x\in\R$, define
\begin{align}
P^x:=\bigcup_{A>0} P_A^x , \qquad P_A^x:=\left\{n_-(u,\beta)\in H_-:\: 
\:x+\frac{1}{u}\in \mathcal D(A,\kappa)\right\}. 
\end{align}
Then, for every $(x,\alpha)\in\RR^2$,  $h\in P^x$ and $s\geq 0$, 
the series \eqref{def-Theta-chi} defining \be\Theta_\chi(\Gamma n_+(x,\alpha) h \Phi^s)\ee is absolutely convergent. 
Moreover, there exists a measurable function $E_\chi^x:P^x\to\R_{\geq 0}$  so that 
\begin{enumerate}
\item[(i)] for every $(x,\alpha)\in\RR^2$, $h\in P^x$ and $s\geq 0$
\begin{align}
\left|\frac{1}{\sqrt N}S_N(x,\alpha)
-\Theta_\chi\!\left(\Gamma n_+(x,\alpha)h\Phi^s\right)\right|
\leq \frac{1}{\sqrt N}
E_\chi^x(h),\label{statement-thm2-rephrased}
\end{align}
where $N=\lfloor \e^{s/2}\rfloor$;
\item[(ii)] for every $u_0>1$, $\beta_0>0$, $A>0$,
\begin{equation}\label{upperEchi}
\sup_{x\in\RR} \sup_{h \in P_A^x \cap \scrK(u_0,\beta_0)} E_\chi^x(h) <\infty,
\end{equation}
with the compact subset $\scrK(u_0,\beta_0)=\{ n_-(u,\beta): u_0^{-1}\leq |u| \leq u_0, \; |\beta|\leq\beta_0\}\subset H_-$.
\end{enumerate}
\end{theorem}

\begin{remark}
To see that Theorem \ref{thm:2-rephrased} implies Theorem \ref{thm:2}  notice that for every $x$ the set $P^x$ is of full measure in the stable horospherical subgroup $H_-$. 
Let
\begin{equation}
\begin{split}
Z&= \{ g\in G : \Phi^{-s} g \Phi^s = g \text{ for all } s\in\RR \} \\
& = \{ \Phi^s (1;\vecnull,t) : (s,t)\in\RR^2 \}.
\end{split}
\end{equation}
Then $G=H_+ H_- Z$ up to a set of Haar measure zero.
By Lemma \ref{lem-D(kappa)-is-Gamma-invariant} and a short calculation (see Lemma \ref{lem-prods}, Remark \ref{remark-D-kappa} below), the set
\be
D=\left\{n_+(x,\alpha) h \in G:\: (x,\alpha)\in\RR^2,\: h \in P^x\right\} Z \subset G
\ee
is $\Gamma$-invariant and has full measure in $G$. 
\end{remark}

\begin{proof}[Proof of Theorem \ref{thm:2-rephrased}]
We assume in following that $|u|\leq u_0$ for an arbitrary $u_0>0$. All implied constants will depend on $u_0$.
Since $\chi$ is the characteristic function of $(0,1)$ (rather than $(0,1]$), we have
\begin{equation}
\left| S_N(x,\alpha)-\sqrt{N}\;\Theta_\chi(n_+(x,\alpha)\Phi^s)\right|\leq 1 \label{pf-thm2-new-00}
\end{equation}
with $s=2\log N$. Now, for every integer $J\geq 0$,
\begin{align}
\sum_{j=0}^J 2^{-j/2}\left(\Theta_\Delta(\Gamma g(\tilde a_{2^{2j}};\bm0,0))+
\Theta_{\Delta_-}(\Gamma g(1;\sve{0}{1},0)(\tilde a_{2^{2j}};\bm0,0))\right)=\Theta_{T}(\Gamma g),
\end{align}
where $T$ is the trapezoidal function \be T=T_{\frac{1}{3\cdot 2^J},1-\frac{1}{3\cdot 2^J}}^{\frac{1}{6\cdot 2^J},\frac{1}{6\cdot 2^J}}.\label{trapezoid-T}\ee
Since, by Lemma \ref{lem-general-trapezoid}, $T\in\mathcal S_2(\R)$, then the series \eqref{def-Jacobi-Theta-function-on-G} defining $\Theta_T(\Gamma g)$ is absolutely convergent for every $\Gamma g\in\GamG$. In view of the support of $T$, we have 
\begin{equation}
\Theta_{T}(\Gamma n_+(x,\alpha)\Phi^s) = \Theta_\chi(n_+(x,\alpha)\Phi^s)
\end{equation}
provided $2^J>\frac{1}{3}\e^{s/2}$. Furthermore,
\be
\begin{split}
\Theta_\chi(\Gamma n_+(x,\alpha) n_-(u,\beta)\Phi^s)=&\,\Theta_T( n_+(x,\alpha) n_-(u,\beta)\Phi^s)\\
&+\sum_{j=J+1}^\infty 2^{-j/2}\Theta_\Delta(n_+(x,\alpha) n_-(u,\beta)\Phi^{s-(2\log 2)j})
\\
&+\sum_{j=J+1}^\infty 2^{-j/2}\Theta_{\Delta_-}(n_+(x,\alpha) n_-(u,\beta)\Phi^{s}n_-(0,1)\Phi^{-(2\log 2)j}) .
\end{split}
\ee
The proof of Theorem \ref{thm:2-rephrased} therefore follows from the following estimates, which we will derive below with the choice $J=\lceil \log_2 N\rceil$:
\begin{equation}
\left|\Theta_\chi(n_+(x,\alpha)\Phi^s)-\Theta_T(\Gamma n_+(x,\alpha)n_-(u,0)\Phi^s)\right| \ll \frac{|u|^{2/3}}{N^{1/2}}; \label{pf-thm2-9}
\end{equation}
\begin{equation}
\left|\Theta_T(n_+(x,\alpha)n_-(u,\beta)\Phi^s)-\Theta_T(n_+(x,\alpha)n_-(u,0)\Phi^s)\right|=O\!\left(\frac{|\beta|}{N^{1/2}}\right);\label{pf-thm2-10}
\end{equation}
\begin{equation}\label{three.sixty}
\sum_{j=J+1}^\infty 2^{-j/2}\left| \Theta_\Delta(n_+(x,\alpha) n_-(u,\beta)\Phi^{s-(2\log 2)j}) \right|\ll_\kappa 
\frac{F_A(u)}{N^{1/2}};
\end{equation}
\begin{equation}\label{three.sixty1}
\sum_{j=J+1}^\infty 2^{-j/2}\left| \Theta_{\Delta_-}(n_+(x,\alpha) n_-(u,\beta)\Phi^{s}n_-(0,1)\Phi^{-(2\log 2)j}) \right| \ll_\kappa \frac{F_A(u)}{N^{1/2}},
\end{equation}
with
\begin{equation}
F_A(u)= \log_2^+(1/|u|) \max(\tfrac{1}{2|u|} , 2)^{1/4} + A^{-\frac{1}{2\kappa}} u^{(1-\frac{1}{\kappa})/4}\,W(u)^{1/4} \max(|u|^{\frac{1}{2\kappa}},1) .
\end{equation}

\begin{proof}[Proof of \eqref{three.sixty} and \eqref{three.sixty1}.] In view of Lemma \ref{lem-expansion-at-infinity},
\begin{equation}\label{three.sixty22}
\sum_{j=J+1}^\infty 2^{-j/2}\left| \Theta_\Delta(n_+(x,\alpha) n_-(u,\beta)\Phi^{s-(2\log 2)j}) \right|\ll 
\sum_{j=J+1}^\infty 2^{-j/2} z_{s-(2\log 2)j}(x,u)^{1/4}. 
\end{equation}
and 
\begin{equation}\label{three.sixty2221}
\sum_{j=J+1}^\infty 2^{-j/2}\left| \Theta_{\Delta_-}(n_+(x,\alpha) n_-(u,\beta)\Phi^{s}n_-(0,1)\Phi^{-(2\log 2)j}) \right|\ll 
\sum_{j=J+1}^\infty 2^{-j/2} z_{s-(2\log 2)j}(x,u)^{1/4}. 
\end{equation}
We divide the sum on the right and side of the above into $j< J+ J_0$ and $j\geq J+J_0$ with $J_0=\lceil\log_2^+(1/|u|)\rceil$. The first is bounded by (apply Lemma \ref{lem:diophantine-visit-cusp22})
\be
\begin{split}
\sum_{J+1\leq j <J+J_0} 2^{-j/2} z_{s-(2\log 2)j}(x,u)^{1/4} 
& \ll  
2^{J/2} J_0 \max(\tfrac{1}{2|u|} , 2)^{1/4} \\
& \ll N^{-1/2} \log_2^+(1/|u|) \max(\tfrac{1}{2|u|} , 2)^{1/4}.
\end{split}
\ee
In the second range
\be
\begin{split}
\sum_{j \geq J+J_0} 2^{-j/2} z_{s-(2\log 2)j}(x,u)^{1/4} 
& \ll  A^{-\frac{1}{2\kappa}} u^{(1-\frac{1}{\kappa})/4}\,W(u)^{1/4}  \sum_{j \geq J+J_0} 2^{-j/2}  \e^{-(1-\frac{1}{\kappa})(s-(2\log 2)j)/4} \\
& \ll_\kappa  N^{-1/2} A^{-\frac{1}{2\kappa}} u^{(1-\frac{1}{\kappa})/4}\,W(u)^{1/4} \max(|u|^{\frac{1}{2\kappa}},1).
\end{split}
\ee
\end{proof}

\begin{proof}[Proof of \eqref{pf-thm2-9}.]
Recall that $2^J>\frac{1}{3}\e^{s/2}$ and we can therefore write \eqref{def-Theta-chi} as 
\be\label{pf-thm2-new-5}
\begin{split}
\Theta_\chi(n_+(x,\alpha)\Phi^s)=
&\sum_{j=0}^J 2^{-\frac{j}{2}}\Theta_{\Delta}(n_+(x,\alpha)\Phi^{s-(2\log 2)j})
\\
&+\sum_{j=0}^J 2^{-\frac{j}{2}}\Theta_{\Delta_-}(n_+(x,\alpha) n_-(0,e^{s/2})\Phi^{s-(2\log 2)j}).
\end{split}
\ee
Consider the sum in the first line of \eqref{pf-thm2-new-5} first.
We have (cf. \eqref{pf-thm1-1})
\begin{equation} 
n_+(x,\alpha)\Phi^{2\log N-(2\log2)j}=\left( x+i\frac{1}{N^2 2^{-2j}},0;\ve{\alpha}{0},0\right) .
\end{equation}
Furthermore
\begin{multline}
n_+(x,\alpha)n_-(u,0)\Phi^{2\log N-(2\log2)j}\\
=\left(x+\frac{u}{N^4 2^{-4j}+u^2}+i\frac{N^2 2^{-2j}}{N^4 2^{-4j}+u^2},\arctan\!\left(\frac{u}{N^2 2^{-2}}\right);\ve{\alpha}{0},0\right),\label{pf-thm2-6}
\end{multline}
so that, for $0\leq j\leq J$,
\be
\begin{split}
&\Theta_{\Delta} (n_+(x,\alpha)n_-(u,0)\Phi^{2\log N-(2\log 2)j})\\
& =
\left(\frac{N^2 2^{-2j}}{N^4 2^{-4j}+u^2}\right)^{1/4}
\sum_{n\in\Z}\Delta_{\arctan\left(\frac{u}{N^2 2^{-2j}}\right)}\!\left(n 
\left(\frac{N^2 2^{-2j}}{N^4 2^{-4j}+u^2}\right)^{1/2}\right)\\
&
\times e\!\left(\ha n^2 
\left(x+\frac{u}{N^4 2^{-4j}+u^2}\right)+n\alpha 
\right).\label{pf-thm2-0}
\end{split}
\ee
Let us now proceed as in the proof of Theorem \ref{thm:1}.
\begin{multline}
\Theta_{\Delta} (\Gamma n_+(x,\alpha)n_-(u,0)\Phi^{2\log N-(2\log 2)j})=
\left(\frac{1}{(N2^{-j})^{1/2}}+O\!\left(\frac{u^2}{(N2^{-j})^{9/2}}\right)\right)\\
 \times \sum_{n\in\Z}\left[\Delta\!\left(n
\left(\frac{(N2^{-j})^{2}}{(N2^{-j})^4+u^2}\right)^{1/2}\right)
+\mathcal E_\Delta\!\left(\arctan\!\left(\frac{u}{(N2^{-j})^2}\right), n
\left(\frac{(N2^{-j})^{2}}{(N2^{-j})^4+u^2}\right)^{1/2}\right)
\right]\\
\times e\!\left(\ha n^2 x+n\alpha\right)\left(1+O\left(\frac{|u|\, n^2}{(N2^{-j})^4}\wedge 1\right)\right).
\end{multline}
Using the Mean Value Theorem as in the proof of \eqref{pf-thm1-5} we obtain
\be
\begin{split}
&\sum_{n\in\Z}\Delta\!\left(n 
\left(\frac{(N2^{-j})^2}{(N2^{-j})^4+u^2}\right)^{1/2}\right)e\!\left(\ha n^2x+n\alpha\right)\\
&=\sum_{n\in\Z}\Delta\!\left(\frac{n}{N2^{-j}}\right)e\!\left(\ha n^2x+n\alpha\right)+
O\!\left(\frac{u^2}{(N2^{-j})^3}\right) .\label{pf-thm2-1}
\end{split}
\ee
Analogously to  \eqref{pf-thm1-6} we have
\begin{align}
&\sum_{n\in\Z} \Delta\!\left( n
\left(\frac{(N2^{-j})^2}{(N2^{-j})^4+u^2}\right)^\ha\right) O\!\left(\frac{u n^2}{(N2^{-j})^4}\wedge 1 \right)
=O\!\left(\frac{|u|}{N2^{-j}}\right).\label{pf-thm2-2}
\end{align}
Moreover, using \eqref{statement-lem-E(phi,t)} and the fact that $\sum_{|n|>2 A}\frac{1}{1+|n/A|^2}=O(A)$, we have
\be\label{pf-thm2-3}
\begin{split}
&\left(\frac{(N2^{-j})^2}{(N2^{-j})^4+u^2}\right)^{1/4}\sum_{n\in\Z} \mathcal E_\Delta\!\left(\arctan\!\left(\frac{u}{(N2^{-j})^2}\right),n 
\left(\frac{(N2^{-j})^{2}}{(N2^{-j})^4+u^2}\right)^{1/2}\right)\\
&\ll \frac{1}{(N2^{-j})^{1/2}}\left(\sum_{|n|\ll N2^{-j}}\left(\frac{|u|}{(N2^{-j})^2}\right)^{3/4}+\sum_{|n|\gg N2^{-j}}\frac{\left(\frac{|u|}{(N2^{-j})^{2}}\right)^{2/3}}{1+\left|\frac{n}{N2^{-j}}\right|^2}\right)\\
&=O\!\left(\frac{|u|^{3/4}}{N2^{-j}}\right)+O\!\left(\frac{|u|^{2/3}}{(N2^{-j})^{5/6}}\right)\\
&= O\!\left(\frac{|u|^{2/3}}{(N2^{-j})^{5/6}}\right).
\end{split}
\ee

Now, combining \eqref{pf-thm2-0}, \eqref{pf-thm2-1}, \eqref{pf-thm2-2}, \eqref{pf-thm2-3} we obtain that 
\be\label{pf-thm2-4}
\begin{split}
&\Theta_{\Delta} (\Gamma n_+(x,\alpha)n_-(u,0)\Phi^{2\log N-(2\log 2)j})\\
&=\Theta_{\Delta} (\Gamma n_+(x,\alpha)\Phi^{2\log N-(2\log 2)j})+O\!\left(\frac{|u|^{2/3}}{(N2^{-j})^{5/6}}\right)
.
\end{split}
\ee
We can use \eqref{pf-thm2-4} for $0\leq j\leq J$ to estimate
\be\label{pf-thm2-7}
\begin{split}
&\sum_{j=0}^J2^{-j/2}\left|\Theta_\Delta( n_+(x,\alpha)\Phi^{s-(2\log 2)j})-\Theta_\Delta( n_+(x,\alpha)n_-(u,0)\Phi^{s-(2\log2)j})\right|\\
&\ll \frac{|u|^{2/3}}{N^{5/6}}\sum_{j=0}^J 2^{j/3}
\ll \frac{|u|^{2/3}}{N^{1/2}}.
\end{split}
\ee

We leave to the reader to repeat the above argument for the sum in the second line of  \eqref{pf-thm2-new-5} and show that
\be\label{pf-thm2-8}
\begin{split}
&\sum_{j=0}^J2^{-j/2}\left|\Theta_\Delta( n_+(x,\alpha)n_-(0,\e^{s/2})\Phi^{s-(2\log 2)j})-\Theta_\Delta( n_+(x,\alpha)n_-(u,\e^{s/2})\Phi^{s-(2\log2)j})\right|\\
&\ll\frac{|u|^{2/3}}{N^{1/2}}
\end{split}
\ee 
\end{proof}

\begin{proof}[Proof of \eqref{pf-thm2-10}.]
This bound follows from the Mean Value Theorem as in the proof of \eqref{pf-thm1-5}.
\end{proof}

This concludes the proof of Theorem \ref{thm:2-rephrased} (and hence of Theorem \ref{thm:2}).
\end{proof}

\subsection{Hardy and Littlewood's approximate functional equation}\label{sec:HL}

To illustrate the strength of Theorem \ref{thm:2-rephrased}, let us show how it implies the approximate functional equation \eqref{ApproxFeq}. Recall the definition \eqref{invariance-by-h1} of $\gene_1\in\Gamma$.

\begin{lem}\label{lem-prods}
For $N>0$, $x>0$,
\begin{equation}\label{prods}
\gene_1 n_+(x,\alpha)n_-(u,\beta)\Phi^{2\log N} =
n_+(x',\alpha')n_-(u',\beta')\Phi^{2\log N'} \bigg( 1 ; \vecnull, \frac18 - \frac{\alpha^2}{2x} \bigg),
\end{equation}
where
\begin{equation}
x'=-\frac1x,\qquad N'= N x,\qquad \alpha'=\frac{\alpha}{x}, \qquad 
u'= x (1+u x) ,\qquad \beta'=\alpha+\beta x.
\end{equation}
\end{lem}

\begin{proof}
Multiplying \eqref{prods} from the right by the inverse of $n_-(u,\beta)\Phi^{2\log N}$ yields 
\begin{equation}\label{prods1}
\gene_1 n_+(x,\alpha) =
n_+(x',\alpha')n_-(\tilde u,\tilde \beta)\Phi^{2\log \tilde N} \bigg( 1 ; \vecnull, \frac18 - \frac{\alpha^2}{2x} \bigg) ,
\end{equation}
where
\begin{equation}
\tilde N = \frac{N'}{N},\qquad \tilde u= u'-u \tilde N^2, \qquad \tilde\beta=\beta'-\beta \tilde N .
\end{equation}
Multiplying the corresponding matrices in \eqref{prods1} yields
\begin{equation}
x'=-\frac1x,\qquad \tilde N  = \tilde u= x, \qquad \alpha'=\frac{\alpha}{x}, \qquad
\tilde\beta=\alpha .
\end{equation}
To conclude, we have to check that the $\phi$-coordinates in \eqref{prods} agree. In fact, the equality
\begin{align}
\arg\!\left(\frac{N^{-2} i}{u N^{-2}i+1}+x\right)+\arg\!\left(u N^{-2} i+1\right)=\arg\!\left( x(1+ux) (x N)^{-2}i+1\right)
\end{align}
is equivalent (since $x,u,N$ are positive) to
\begin{align}
\arctan\!\left(\frac{N^2}{u+(N^4+u^2)x}\right)+\arctan\!\left(\frac{u}{N^2}\right)=\arctan\!\left(\frac{1+ux}{N^2 x}\right),
\end{align}
which can be seen to hold true using the identity $\arctan(A)+\arctan(B)=\arctan\!\left(\frac{A+B}{1-AB}\right)$.
\end{proof}

\begin{remark}\label{remark-D-kappa}
Note that in Lemma \ref{lem-prods}
\begin{equation}
x'+\frac{1}{u'} = - \bigg(x+\frac{1}{u}\bigg)^{-1}.
\end{equation}
Therefore, by Lemma \ref{lem-D(kappa)-is-Gamma-invariant}, $x+\frac{1}{u}\in\mathcal D(\kappa)$  if and only if $x'+\frac{1}{u'}\in\mathcal D(\kappa)$.
\end{remark}

\begin{cor}
For every $0<x<2$ and every $0\leq\alpha\leq 1$ the approximate functional equation \eqref{ApproxFeq} holds.
\end{cor}
\begin{proof} Let us use the notation of Lemma  \ref{lem-prods}.
The invariance of $\Theta_\chi$ under the left multiplication by $\gene_1\in\Gamma$, see \eqref{Jacobi1}, and Lemma \ref{lem-prods} 
yield
\begin{align}
\Theta_\chi\!\left(\Gamma n_+(x,\alpha)n_-(u,\beta)\Phi^{2\log N}\right)=\sqrt{i}\;e\!\left(-\frac{\alpha^2}{2x}\right)\Theta_\chi\!\left(\Gamma n_+(x',\alpha')n_-(u',\beta')\Phi^{2\log N'}\right).
\end{align}
By applying \eqref{statement-thm2-rephrased} twice with $n_-(u,\beta)\in P^x$ and (by Remark \ref{remark-D-kappa}) $n_-(u',\beta')\in P^{x'}$ we obtain
\be
\begin{split}
&\left|S_N(x,\alpha)-\sqrt{\frac{i}{x}}\;e\!\left(-\frac{\alpha^2}{2x}\right)S_{\lfloor xN\rfloor}\!\left(-\frac{1}{x},\frac{\alpha}{x}\right) \right|\\
&\leq \left| S_N(x,\alpha)-\sqrt{N}\Theta_\chi\!\left(\Gamma n_+(x,\alpha)n_-(u,\beta)\Phi^{2\log N}\right)\right|\\
&+\left|\sqrt{iN}\;e\!\left(-\frac{\alpha^2}{2x}\right)\Theta_\chi\!\left(\Gamma n_+(x',\alpha')n_-(u',\beta')\Phi^{2\log N'}\right)-\sqrt{\frac{i}{x}}\;e\!\left(-\frac{\alpha^2}{2x}\right)S_{\lfloor xN\rfloor}\!\left(-\frac{1}{x},\frac{\alpha}{x}\right) \right|\\
&\leq E_\chi^{x}(n_-(u,\beta))+\frac{1}{\sqrt x}E_\chi^{x'}(n_-(u',\beta')) .
\end{split}
\ee
What remains to be shown is that $E_\chi^{x}(n_-(u,\beta))$ and $E_\chi^{x'}(n_-(u',\beta'))$ are uniformly bounded in $x,\alpha$ over the relevant ranges. To this end, recall that $u$ and $\beta$ are free parameters that, given $x$, we choose as
\begin{equation}
u=\frac{1}{\sqrt{5}-x},\qquad \beta=0 .
\end{equation}
Then $x+\frac1u=\sqrt5\in\mathcal D(1)$ and $u$ is bounded away from $0$ and $\infty$ for $0<x<2$, and hence, by \eqref{upperEchi},  $E_\chi^{x}(n_-(u,\beta))$ is uniformly bounded. Furthermore, with the above choice of $u$, we have
\begin{equation}
u'=\bigg(\frac1x - \frac{1}{\sqrt5}\bigg)^{-1}, \qquad \beta'=\alpha. 
\end{equation}
Thus $x'+\frac{1}{u'}=-\frac{1}{\sqrt5}\in\mathcal D(1)$ and $u'$ is bounded away from $0$ and $\infty$ for $0<x<2$. Therefore, again in view of \eqref{upperEchi}, $E_\chi^{x'}(n_-(u',\beta'))$ is uniformly bounded for $0<x<2$, $0\leq\alpha\leq 1$.
\end{proof}

\subsection{Tail asymptotic for $\Theta_\chi$}\label{subs:tail-asymptotic-Theta_chi}

For the theta function $\Theta_\chi$ we also have tail asymptotics with an explicit power saving.
\begin{theorem}\label{tail-asymptotics-for-Theta_chi}
For $R\geq 1$, 
\be\mu(\{g\in\GamG:\: |\Theta_\chi(g)|>R\})=2 R^{-6} 
\left(1+O\!\left(R^{-\frac{12}{31}}\right)\right) .
\label{statement-theorem-tail-asymptotics-for-Theta_chi}
\ee
\end{theorem}

Recall the ``trapezoidal'' functions $\chi^{(J)}_L$ and $\chi^{(J)}_R$ defined in (\ref{def-chi_L^J}, \ref{def-chi_R^J}) and $\chi^{(J)}=\chi^{(J)}_L+\chi^{(J)}_R$. The proof of Theorem \ref{tail-asymptotics-for-Theta_chi} requires the following three lemmata.
\begin{lem}\label{Theta_chi-series-truncated}
Let $J\geq 1$. 
Define
\begin{align}
\mathcal F_J(g)&:=\Theta_{\chi^{(J)}_L}(g)=\sum_{j=0}^{J-1}2^{-\frac{j}{2}}\Theta_{\Delta}(g \Phi^{-(2\log 2)j})\\
\mathcal G_J(g)&:=\Theta_{\chi^{(J)}_L(-\cdot)}(g)=\sum_{j=0}^{J-1}2^{-\frac{j}{2}}\Theta_{\Delta_-}(g \Phi^{-(2\log 2)j}).
\end{align}
Then 
\be
\Theta_{\chi}(g)=\Theta_{\chi^{(J)}}(g)+\sum_{k=1}^\infty 2^{-k \frac{J}{2}}\left(\mathcal F_J(g \Phi^{-(2\log 2)k J})+\mathcal G_J(g (1;\sve{0}{1},0)\Phi^{-(2\log 2)k J})\right)\label{statement-Theta_chi-series-truncated}
\ee
\end{lem}
\begin{proof}
Recall that, by Lemma \ref{lem-general-trapezoid},  $\chi^{(J)}_L, \chi^{(J)}_L(-\cdot)\in\mathcal S_2(\R)$,  and therefore the two theta functions above are well defined for every $g$.

The first sum in \eqref{def-Theta-chi} can be written as
\begin{align}
\sum_{j=0}^\infty 2^{-\frac{j}{2}}\Theta_{\Delta}(g \Phi^{-(2\log 2)j})&=\sum_{k=0}^{\infty}\sum_{j=k J}^{(k+1)J-1}2^{-\frac{j}{2}}\Theta_{\Delta}(g\Phi^{-(2\log2)j}).
\end{align}
and the $k$-th term in the above series is
\be
\begin{split}
\sum_{l=0}^{J-1}2^{-\frac{l+ kJ}{2}}\Theta_{\Delta}(g \Phi^{-(2\log 2)(l+k J)})&=2^{-k\frac{J}{2}}\sum_{l=0}^{J-1}2^{-\frac{l}{2}}\Theta_\Delta(g\Phi^{-(2\log 2)k J}\Phi^{-(2\log2)l})\\
&=2^{-k\frac{J}{2}}\mathcal F_J(g\Phi^{-(2\log 2)k J}).
\end{split}
\ee
Similarly, the second sum in \eqref{def-Theta-chi} reads as
\be
\begin{split}
&\sum_{j=0}^\infty 2^{-\frac{j}{2}}\Theta_{\Delta_-}(g (1;\sve{0}{1},0)\Phi^{-(2\log 2)j})\\
& =\sum_{k=0}^\infty \sum_{j=k J}^{(k+1)J-1}2^{-\frac{j}{2}} \Theta_{\Delta_-}(g (1;\sve{0}{1},0)\Phi^{-(2\log 2)j})\\
&=\mathcal G_J(g (1;\sve{0}{1},0))+\sum_{k=1}^\infty 2^{-k\frac{J}{2}} \mathcal G_J(g (1;\sve{0}{1},0)\Phi^{-(2\log 2)k J}).
\end{split}
\ee
The last thing to observe is that $\mathcal G_J (g (1;\sve{0}{1},0))=\Theta_{\chi^{(J)}_L(1-\cdot)}(g)=\Theta_{\chi^{(J)}_R}(g)$. Thus
\be
\mathcal F_J(g)+\mathcal G_J (g (1;\sve{0}{1},0)) = \Theta_{\chi^{(J)}}(g),
\ee
which concludes the proof of \eqref{statement-Theta_chi-series-truncated}.
\end{proof}

\begin{lem}\label{lemma-union-bounds}
There is a constant $K$ such that, for all $R\geq K \delta^{-1} 2^{J/2}$, $2^{-(J-1)/2}\leq\delta\leq \ha$,
\begin{align}
&\mu\{g\in\GamG:\:|\Theta_\chi(g)|>R\}- \mu\{g\in\GamG:\: |\Theta_{\chi^{(J)}}(g)|>R(1-\delta)\}
=O\!\left(\frac{1}{R^6 \delta^6 2^{3J}}\right)\label{statement-lemma-union-bounds-1},
\\
&\mu\{g\in\GamG:\: |\Theta_{\chi^{(J)}}(g)|>R(1+\delta)\}-\mu\{g\in\GamG:\:|\Theta_\chi(g)|>R\}
=O\!\left(\frac{1}{R^6 \delta^6 2^{3J}}\right). \label{statement-lemma-union-bounds-2}
\end{align}
\end{lem}
\begin{proof}
We use the  identity $1=(1-\delta)+2\frac{(1-\delta)}{2}\sum_{k=1}^\infty\delta^k$ and 
 Lemma \ref{Theta_chi-series-truncated} for the following union bound estimate:
\begin{align}
&\mu\{g\in\GamG:\:|\Theta_\chi(g)|>R\} \nonumber \\
&\leq \mu\{g\in\GamG:\:|\Theta_{\chi^{(J)}}(g)|>R(1-\delta)\} \nonumber \\
&+\sum_{k=1}^\infty\mu\{g\in\GamG:\:|2^{-k\frac{J}{2}}\mathcal F_J (g\Phi^{-(2\log 2)k J})
|>R\tfrac{(1-\delta)}{2}\delta^k\}\label{sum-mu-F}
\\
&+\sum_{k=1}^\infty\mu\{g\in\GamG:\:|2^{-k\frac{J}{2}}\mathcal G_J (g(1;\sve{0}{1},0)\Phi^{-(2\log 2)k J})|>R\tfrac{(1-\delta)}{2}\delta^k\}. \label{sum-mu-G}
\end{align}
Let us consider the sum in \eqref{sum-mu-F}. We know by Lemma \ref{lem-general-trapezoid} that $\chi^{(J)}_L\in\mathcal S_2(\R)$ with $\kappa_2(\chi^{(J)}_L)=O(2^J)$. We choose the constant $K$  sufficiently large so that
\begin{equation}
R\tfrac{(1-\delta)}{2}\delta 2^{\frac{J}{2}} \geq K \kappa_2(\chi^{(J)}_L) ,
\end{equation}
holds uniformy in all parameters over the assumed ranges. This implies
\begin{equation}
R\tfrac{(1-\delta)}{2}\delta^k 2^{k\frac{J}{2}} \geq K \kappa_2(\chi^{(J)}_L) .
\end{equation}
for all $k$.
Hence, by Lemma \ref{mu(tail>R)}, we can write \eqref{sum-mu-F} as
\begin{align}
&\sum_{k=1}^\infty \mu\{g\in\GamG:\: |\Theta_{\chi^{(J)}_L}(g\Phi^{-(2\log2)k J})|>R\frac{(1-\delta)}{2}\delta^k 2^{k\frac{J}{2}}\}\\
&=\sum_{k=1}^\infty\frac{D(\chi^{(J)}_L)}{R^6 (1-\delta)^6\delta^{6k}2^{3kJ}} O(1)
=O\!\left(\frac{1}{R^6\delta^6 2^{3J}}\right) .\label{O()+O()}
\end{align}
The sum in \eqref{sum-mu-G} is estimated in the same way and also yields \eqref{O()+O()}. This proves \eqref{statement-lemma-union-bounds-1}. In order to get \eqref{statement-lemma-union-bounds-2} we use again Lemma \ref{Theta_chi-series-truncated} and the following union bound, yielding a lower bound:
\be
\begin{split}
&\mu\{g\in\GamG:\:|\Theta_{\chi}(g)|>R\}\\
&\geq \mu\{g\in\GamG:\:|\Theta_{\chi^{(J)}}(g)|> R(1+\delta)\}\\
&-\sum_{k=1}^\infty\mu\{g\in\GamG:\:|2^{-k\frac{J}{2}}\mathcal F_J (g\Phi^{-(2\log 2)k J})
|>R\tfrac{(1-\delta)}{2}\delta^k\}\\
&-\sum_{k=1}^\infty\mu\{g\in\GamG:\:|2^{-k\frac{J}{2}}\mathcal G_J (g(1;\sve{0}{1},0)\Phi^{-(2\log 2)k J})
|>R\tfrac{(1-\delta)}{2}\delta^k\}.
\end{split}
\ee
The last two sums are $O\!\left(\frac{1}{R^6\delta^6 2^{3J}}\right)$ as before and we obtain \eqref{statement-lemma-union-bounds-2}.
\end{proof}

\begin{lem}\label{lemma-contributions-main-term-delta-J}
There is a constant $K$ such that, for all $R\geq K 2^{J}$, $0<\delta\leq \ha$,
\begin{multline}
\mu\{g\in\GamG:\:|\Theta_{\chi^{(J)}}(g)|>R(1\pm \delta)\} \\
=\frac{2}{3}\frac{D(\chi)}{R^6}(1+O(\delta))\left(1+O\!\left(2^{4J} R^{-4}\right)\right)(1+O(2^{-J})) . \label{from-main-term}
\end{multline}
\end{lem}
\begin{proof}
Lemma \ref{mu(tail>R)} implies that
\be
\begin{split}
\mu(\{g\in\GamG:\:|\Theta_{\chi^{(J)}}(g)|>R(1\pm\delta)\})&=\frac{2}{3}\frac{D(\chi^{(J)})}{R^6(1\pm\delta)^6}\left(1+O\!\left(2^{4J} R^{-4}\right)\right)\\
&=\frac{2}{3}\frac{D(\chi^{(J)})}{R^6}(1+O(\delta))\left(1+O\!\left(2^{4J} R^{-4}\right)\right).
\end{split}
\ee
To complete the proof we need to show that
\be\label{claim-D(chi^J)=D(chi)(1+error)}
D(\chi^{(J)})=D(\chi)+O(2^{-J}) .
\ee
The identity $x^6-y^6=(x^2-y^2)(x^4+x^2y^2+y^4)$, and the fact that $\chi_\phi^{(J)}$, $\chi_\phi$ are uniformly bounded (Lemma \ref{lemma-|f_phi(w)|<const}), imply
\be
\begin{split}
\left| D(\chi) - D(\chi^{(J)}) \right| & \ll 
\left| \int_{0}^\pi\int_{-\infty}^\infty |\chi_\phi(w)|^2\de w\,\de \phi
- \int_{0}^\pi\int_{-\infty}^\infty |\chi^{(J)}_\phi(w)|^2\de w\,\de \phi \right| \\
& = \left| \int_{0}^\pi \| \chi_\phi \|_{L^2}^2 \de\phi - \int_{0}^\pi \| \chi_\phi^{(J)} \|_{L^2}^2 \de\phi \right| \\
& = \pi \left| \| \chi \|_{L^2}^2 - \| \chi^{(J)} \|_{L^2}^2 \right| ,
\end{split}
\ee
by unitarity of the Shale-Weil representation. The triangle inequality yields
\begin{equation}
\left| \| \chi \|_{L^2}^2 - \| \chi^{(J)} \|_{L^2}^2 \right| \leq \left( \| \chi \|_{L^2} + \| \chi^{(J)} \|_{L^2} \right) \| \chi  -  \chi^{(J)} \|_{L^2} = O(2^{-J}),
\end{equation}
and \eqref{claim-D(chi^J)=D(chi)(1+error)} follows.
\end{proof}

\begin{proof}[Proof of Theorem \ref{tail-asymptotics-for-Theta_chi}]
Combining Lemma \ref{lemma-union-bounds} and Lemma \ref{lemma-contributions-main-term-delta-J} we have
\begin{multline}
\mu(\{g\in\GamG:\: |\Theta_\chi(g)|>R\}) \\
=\frac{2}{3}\frac{D(\chi)}{R^6}(1+O(\delta))\left(1+O\!\left(2^{4J} R^{-4}\right)\right)(1+O(2^{-J}))
+O\!\left(R^{-6}\delta^{-6} 2^{-3J}\right),
\end{multline}
provided
\be\label{condi}
R\geq K \delta^{-1} 2^{J/2}, \qquad 2^{-(J-1)/2}\leq\delta\leq \ha .
\ee
We set $J=\alpha \log_2 R$, $\delta=K R^{-\beta}$,
with positive constants $\alpha,\beta, K>0$ satisfying $\alpha\geq 2\beta$, $\alpha+2\beta\leq 2$, $K\geq \sqrt 2$.
Then \eqref{condi} holds for all $R\geq 1$, and
\begin{equation}
\mu(\{g\in\GamG:\: |\Theta_\chi(g)|>R\} )
=\frac{2}{3}\frac{D(\chi)}{R^6} + \frac{1}{R^6} \, O(R^{-\beta}+R^{4\alpha-4}+R^{-\alpha}+R^{6\beta-3\alpha}).
\end{equation}
We need to work out the minimum of $\beta,4-4\alpha,\alpha,3\alpha-6\beta$ under the given constraints. If $\beta\leq 3\alpha-6\beta$, the largest possible value for $\beta$ is $\frac37 \alpha$. Optimizing $\alpha$ yields $\alpha=\frac{28}{31}$ and thus the error term is $O(R^{-\frac{12}{31}})$. If on the other hand $\beta\geq 3\alpha-6\beta$, we maximise $3\alpha-6\beta$ by choosing the smallest permitted $\beta$, which is again $\frac37 \alpha$. This yields $3\alpha-6\beta=\frac37\alpha$ and we proceed as before to obtain the error $O(R^{-\frac{12}{31}})$.

Finally, we need to prove 
\begin{align}
D(\chi)=\int_{-\infty}^{\infty}\int_0^\pi|\chi_\phi(w)|^6\de \phi\de w=3,\label{D(chi)-to-show}
\end{align}
see Figure \ref{fig:integrand}.
\begin{figure}[hb!]
\begin{center}
\hspace{-1cm}\includegraphics[width=13.5cm]{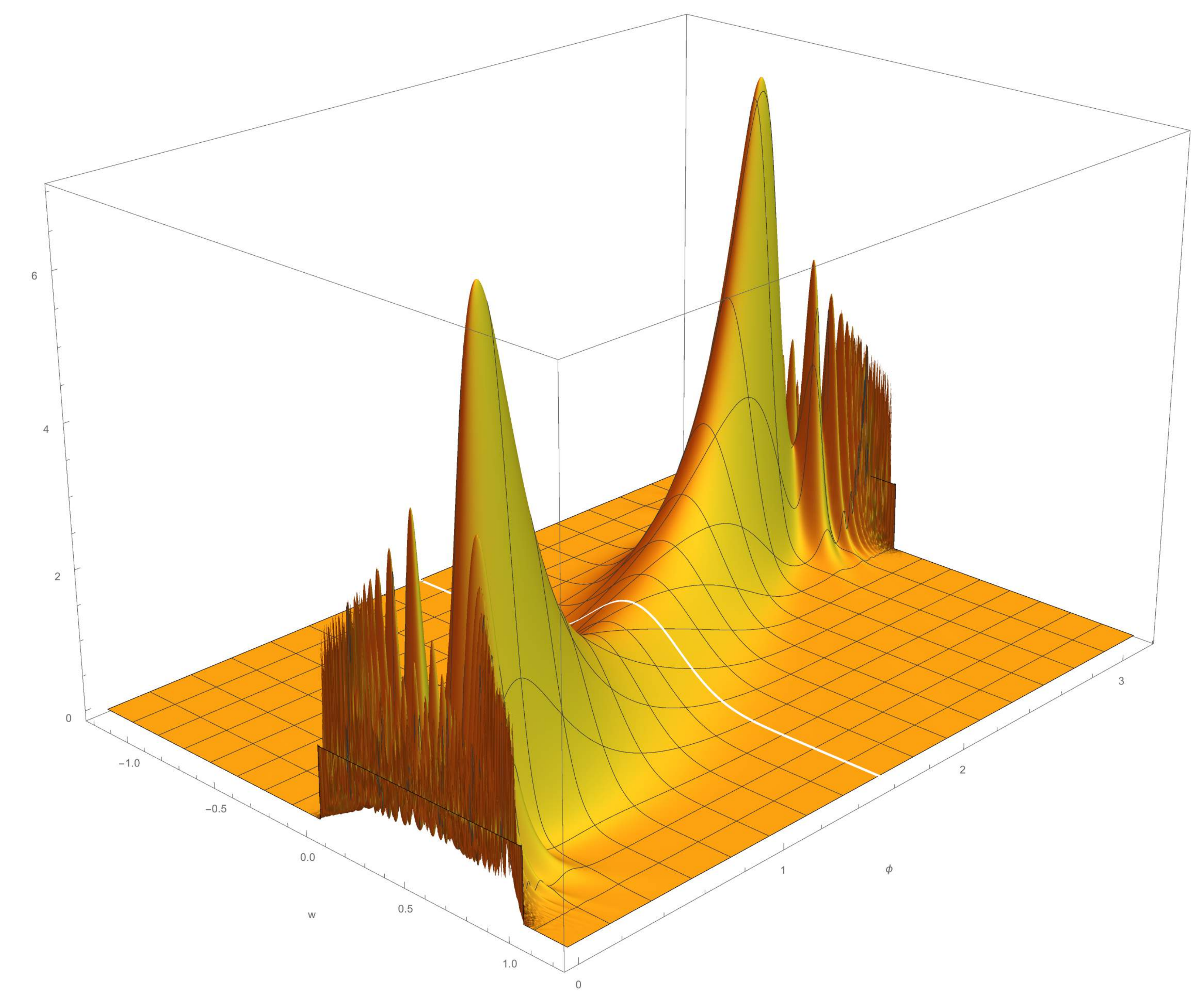}
\includegraphics[width=10.6cm]{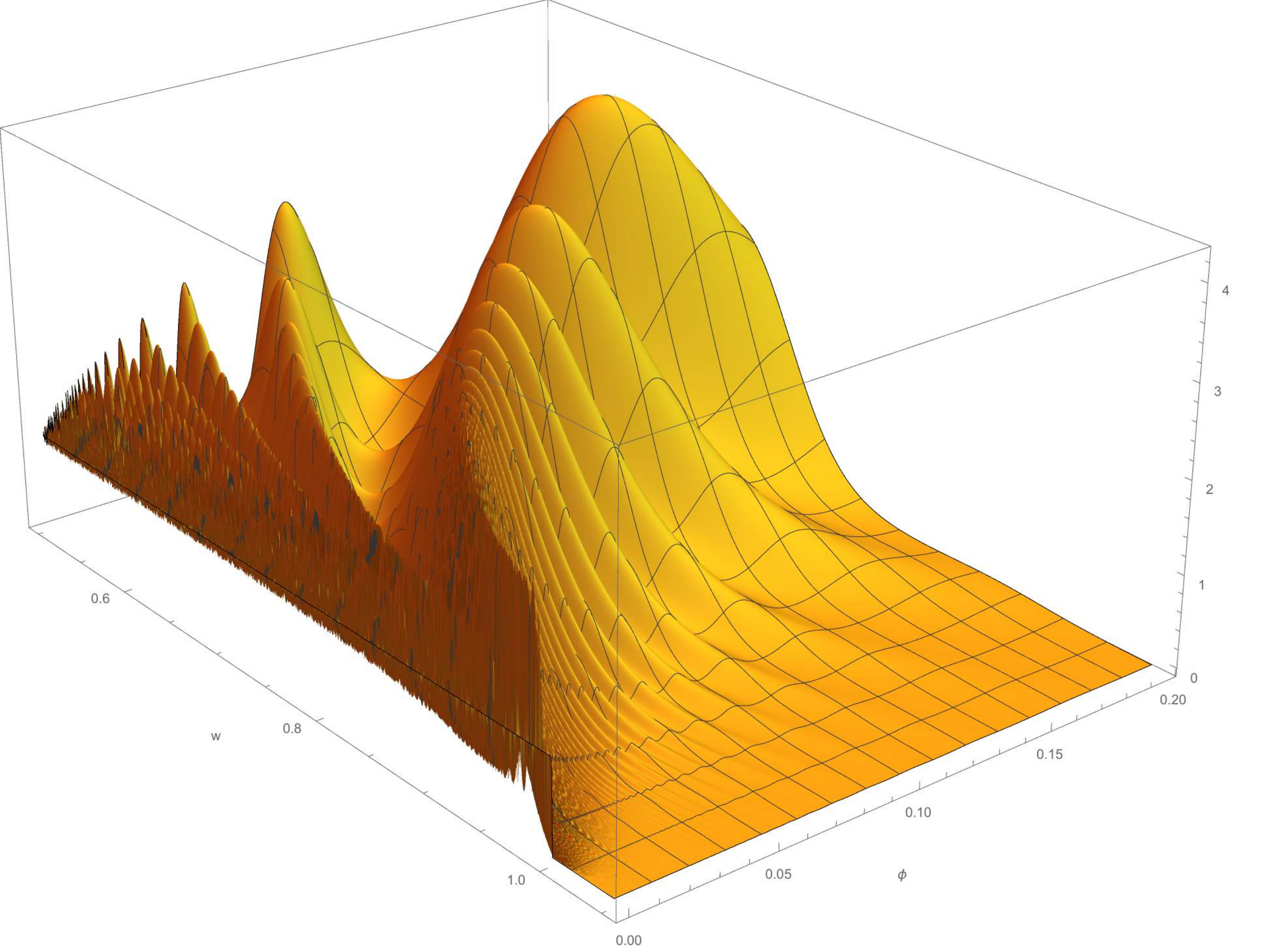}
\caption{The function $(w,\phi)\mapsto|\chi_\phi(w)|^6$. The first plot corresponds to $(w,\phi)\in[-\tfrac{3}{2},\tfrac{3}{2}]\times[0,\pi]$. The second plot corresponds to $(w,\phi)\in[0.5,1.05]\times[0,\tfrac{\pi}{16}]$ and illustrates the highly oscillatory nature of the integrand in  \eqref{D(chi)-to-show}.}\label{fig:integrand}
\end{center}
\end{figure}
Recall \eqref{def-R-tilde-k-f}. Using the change of variables $z=w \csc\phi$ we obtain
\be\begin{split}
D(\chi)&=\int_{0}^\pi\int_{-\infty}^\infty\frac{1}{|\sin \phi|^3}\left|\int_0^1 e\!\left(\ha w'^2\cot\phi-ww'\csc\phi\right)dw'\right|^6 dw d\phi\\
&=\int_{0}^\pi\int_{-\infty}^\infty\frac{1}{|\sin \phi|^2}\left|\int_0^1 e\!\left(\ha w'^2\cot\phi-w'z\right)dw'\right|^6 dz d\phi.
\end{split}
\ee
Now the change of variables $u=\ha\cot\phi$ yields
\begin{align}
D(\chi)&=2\int_{-\infty}^\infty\int_{-\infty}^\infty \left|\int_0^1 e(u w'^2-zw') dw'\right|^6 dzdu=3,\label{oscillatory-integral}
\end{align}
cf. \cite{Rogovskaya}. This concludes the proof of \eqref{statement-theorem-tail-asymptotics-for-Theta_chi}.
\end{proof}

\subsection{Uniform tail bound for $\Theta_\chi$}\label{subs:unif-tail-bound}
The goal of this section is to obtain a result, similar to Theorem \ref{tail-asymptotics-for-Theta_chi}, for the tail distribution of $|\Theta_\chi(x+i y,0;\vecxi,\zeta)|$ uniform in all variables. Namely the following 

\begin{prop}\label{lem-uniform-tail-estimate-chi}
Let $\lambda$ be a Borel probability measure on $\RR$ which is absolutely continuous with respect to Lebesgue measure. Then
\begin{align}
\lambda(\{x\in\R :\: |\Theta_\chi(x+i y,0;\vecxi,\zeta)|>R\})\ll \frac{1}{(1+R)^4},
\end{align}
uniformly in $y\leq 1$,  $R>0$, and $(\vecxi,\zeta)\in\Hei$.
\end{prop}
This proposition will be used in Section \ref{sec:tightness} to show that the finite dimensional limiting distributions for $X_N(t)$ are tight.
The proof of Proposition \ref{lem-uniform-tail-estimate-chi} requires three lemmata.
First, let us define $\hat\Gamma=\PSL(2,\ZZ)$ and $\Gamma_\infty=\{\sma{1}{m}{0}{1}:\: m\in\Z\}$. 
For $z\in\h$ and $\gamma\in\Gamma_\infty\backslash \hat\Gamma$ set $y_\gamma=\Im(\gamma z)$ and define
\begin{align}
H(z)&=\sum_{\gamma\in\Gamma_\infty\backslash \hat\Gamma} y_\gamma^{1/4}\chi_{[\frac{1}{2},\infty)} (y_\gamma^{1/4})\label{def-H(z)} .
\end{align}
We note that this sum has a bounded number of terms, and $H(\gamma z)=H(z)$. $H$ may thus be viewed as a function on $\hat\Gamma\backslash\h$. Since $\GamG$ is a fibre bundle over the modular surface $\hat\Gamma\backslash\h$ (recall Section \ref{sec:JTF}), we may furthermore identify $H$ with a $\Gamma$-invariant function on $G$ by setting $H(g) := H(z)$ with $g=(x+iy,\phi;\vecxi,\zeta)$.

\begin{lem}\label{lem-compare-|Theta_f|-and-H(z)}
Let $f\in\mathcal S_\eta(\R)$, $\eta>1$. Then, for every $g\in G$, we have
\be
|\Theta_f(g)|\leq C_1 H(g),
\ee
where $C_1=\kappa_0(f)+C_1' \kappa_\eta(f)$ and $C_1'>0$ is some absolute constant.
\end{lem}
\begin{proof}
Lemma \ref{lem-expansion-at-infinity} implies that $|\Theta_f(x+iy,\phi;\vecxi,\zeta)|\leq C_1 y^\frac{1}{4}$ uniformly in all variables for $y\geq \frac{1}{2}$, and thus uniformly for all points in the fundamental domain $\mathcal F_{\Gamma}$. By definition, $H(x+iy)$ is a sum of positive terms and $y^{1/4}$ is one of them if $y\geq\frac{1}{2}$. Hence 
 $y^{\frac{1}{4}}\leq H(x+iy)$.
\end{proof}
The following lemma estimates how much of the closed horocycle $\{x+iy:\:-\frac{1}{2}\leq x\leq \frac{1}{2}\}$ is above a certain height in the cusp. 
\begin{lem}\label{lem-sum-r}
Let $R\geq 1$ and $r(y)=\bm{1}_{\{y^{1/4}>R\}}$. Then, for every $y\leq 1$, 
\be
\int_0^1\sum_{\gamma\in\Gamma_\infty\backslash \hat\Gamma} r(y_\gamma)\de x\leq 2 R^{-4}
\ee
\end{lem}
\begin{proof}
For $z=x+i y$ we have $\Im\!\left(\sma{a}{b}{c}{d} z\right)=\frac{y}{|cz+d|^2}$.
Thus, writing $d=d'+mc$ with $1\leq d'\leq d-1$, we get
\be
\begin{split}
\sum_{\gamma\in\Gamma_\infty\backslash \hat\Gamma} r(y_\gamma)&= r(y)+\sum_{\footnotesize{\ba{c} (c,d)=1\\ c>0\\ d\in \Z \ea}}r\!\left(\frac{y}{|cz+d|^2}\right)\\
&=r(y)+\sum_{c=1}^\infty\sum_{\footnotesize{\ba{c}d'\bmod c\\ (c,d')=1\ea}}\sum_{m\in\Z}r\!\left(\frac{y}{c^2\left|z+\frac{d'}{c}+m\right|^2}\right).
\end{split}
\ee
So
\begin{align}
\int_0^1\sum_{\gamma\in\Gamma_\infty\backslash\hat\Gamma}r(y_\gamma)
\de x=r(y)+\sum_{c=1}^\infty\sum_{\footnotesize{\ba{c} d'\bmod c\\ (c,d')=1\ea}}\int_{-\infty}^\infty  r\!\left(\frac{y}{c^2|x+i y|^2}\right)\de x.\label{pf-lem-sum-r-1}
\end{align}
A change of variables allows us to write the last integral as
\begin{align}
\int_{-\infty}^\infty  r\!\left(\frac{y}{c^2|x+i y|^2}\right)\de x=y\int_{-\infty}^\infty r\!\left(\frac{1}{c^2 y (x^2+1)}\right)\de x=y\tilde r(c^2 y),
\end{align}
where 
\begin{align}
\tilde r(t)&=\int_{-\infty}^\infty  r\!\left(\frac{1}{t (x^2+1)}\right)\de x=\int_{\{x\in \R:\:t(x^2+1)<R^{-4}\}}\de x=\begin{cases}2\sqrt{\frac{1}{R^4 t}-1}&\mbox{if $R^{-4}>t$,}\\ 0&\mbox{otherwise.}\end{cases}
\end{align}
Now \eqref{pf-lem-sum-r-1} equals 
\begin{align}
&r(y)+2y\sum_{\footnotesize{\ba{c}c\geq 1\\ c^2 y<R^{-4}\ea}}  \sum_{\footnotesize{\ba{c} d'\bmod c\\ (c,d')=1\ea}}\sqrt{\frac{1}{R^{4}c^2 y}-1}.\label{pf-lem-sum-r-2}
\end{align}
Since  $y\leq 1$
we have $r(y)=0$ and
\eqref{pf-lem-sum-r-2} is
\begin{align}
&\leq 
2 y^{\ha}R^{-2}\sum_{\footnotesize{\ba{c}c\geq 1\\ c^2 y<R^{-4}\ea}}  \sum_{\footnotesize{\ba{c} d'\bmod c\\ (c,d')=1\ea}}\frac{1}{c}\\
&\leq 
2 y^{\ha}R^{-2}\sum_{\footnotesize{\ba{c}c\geq 1\\ c^2 y<R^{-4}\ea}} 1=
2 y^{\ha}R^{-2}\left\lfloor \sqrt{\frac{1}{y r^4}}\right\rfloor\leq y^{\ha}R^{-2}\sqrt{\frac{1}{yR^4}}=
2 R^{-4}.
\end{align}
\end{proof}

\begin{lem}\label{lem-uniform-tail-estimate}
Let $f\in \mathcal S_\eta(\R)$, $\eta>1$, and $\lambda$ as in Proposition \ref{lem-uniform-tail-estimate-chi}.
\begin{align}
\lambda(\{x\in\R :\:|\Theta_f(x+i y,\phi;\vecxi,\zeta)|>R\})\ll (1+R)^{-4}\label{statement-lemma-unif-tail-estimate},
\end{align}
uniformly in $y\leq 1$,  $R>0$, and all $\phi,\vecxi,\zeta$.
\end{lem}
\begin{proof}
Lemma \ref{lem-compare-|Theta_f|-and-H(z)} yields
\begin{equation}
\int_{\R}\bm 1_{\{|\Theta_f(x+i y,\phi;\vecxi,\zeta)|>R\}}\lambda(\de x) 
\leq\int_{\R}\bm 1_{\{H(x+ i y)>\frac{R}{C_1}\}} \lambda(\de x).
\end{equation}
Since $H(z+1)=H(z)$, we have
\begin{equation}
\int_{\R}\bm 1_{\{H(x+ i y)>\frac{R}{C_1}\}} \lambda(\de x) =
\int_{\R/\Z}\bm 1_{\{H(x+ i y)>\frac{R}{C_1}\}} \lambda_\Z(\de x)
\end{equation}
where $\lambda_\Z$ is the push forward under the map $\R\to\R/\Z$, $x\mapsto z+\Z$. Since $\lambda$ is absolutely continuous, so is $\lambda_\Z$ (with respect to Lebesgue measure on $\R/\Z$) and we have
\begin{equation}
\int_{\R/\Z}\bm 1_{\{H(x+ i y)>\frac{R}{C_1}\}} \lambda_\Z(\de x) \ll
\int_{\R/\Z}\bm 1_{\{H(x+ i y)>\frac{R}{C_1}\}} \de x .
\end{equation}
Now, for $R\geq C_1$ we have
\begin{equation}
\int_{\R/\Z}\bm 1_{\{H(x+ i y)>\frac{R}{C_1}\}} \de x =
\sum_{\gamma\in\Gamma_\infty\backslash\hat\Gamma} \bm 1_{\left\{y_\gamma^{1/4}>\frac{R}{C_1}\right\}}\de x \leq 2 C_1^4 R^{-4}
\end{equation}
in view of Lemma \ref{lem-sum-r}. As the left hand side of \eqref{statement-lemma-unif-tail-estimate} is bounded trivially by 1 for all $R<C_1$, the proof is complete.
\end{proof}

\begin{proof}[Proof of Proposition \ref{lem-uniform-tail-estimate-chi}]
Write 
\begin{align}
\Theta_\chi(x+iy,0;\vecxi,\zeta)=
&\sum_{j=0}^J 2^{-\frac{j}{2}}\Theta_{\Delta}((x+iy,0;\vecxi,\zeta)\Phi^{-(2\log 2)j})\label{union-bounds-for-tightness-3}\\
&+\sum_{j=0}^J 2^{-\frac{j}{2}}\Theta_{\Delta_-}((x+iy,0;\vecxi,\zeta)(1;\sve{0}{1},0)\Phi^{-(2\log 2)j})\label{union-bounds-for-tightness-4}
\end{align}
where  $J=\lceil \log_2 y^{-1/2}\rceil$.
Set
\begin{equation}
\delta_j :=\frac{3}{2\pi^2}\frac{1}{j^2} \qquad(1\leq j\leq J), \qquad
\delta_0 :=\frac{1}{2}-\sum_{j=1}^J \delta_j .
\end{equation}
Notice that $\tfrac{1}{4}<\delta_0\leq \frac{1}{2}-\frac{3}{2\pi^2}$ and 
$2\sum_{j=0}^J \delta_j
=1$. 
In order to handle the two sums \eqref{union-bounds-for-tightness-3} and \eqref{union-bounds-for-tightness-4} we use a union bound as in the proof of Lemma \ref{lemma-union-bounds} and apply Lemma \ref{lem-uniform-tail-estimate}. We obtain
\be\label{union-bounds-for-tightness-5}
\begin{split}
&\lambda(\{x\in\R:\:|\Theta_\chi(x+i y,0;\vecxi,\zeta)|>R\})\\
&\leq \sum_{j=0}^{\lceil \log y^{-1}\rceil}\lambda(\{x\in\R:\: |\Theta_{\Delta}((x+iy,0;\vecxi,\zeta)\Phi^{-(2\log 2)j})|>2^{\frac{j}{2}}\delta_j R \})\\
&+\sum_{j=0}^{\lceil \log y^{-1}\rceil}\lambda(\{x\in\R:\: |\Theta_{\Delta_-}((x+iy,0;\vecxi,\zeta)(1;\sve{0}{1},0)\Phi^{-(2\log 2)j})|>2^{\frac{j}{2}}\delta_j R \})\\
&\ll 
\sum_{j=0}^{\lceil \log y^{-1}\rceil} \frac{1 
}{(1+ 2^{\frac{j}{2}} \delta_j R)^4}\leq R^{-4} \sum_{j=0}^{\lceil \log y^{-1}\rceil}2^{-2j}\delta_j^{-4}\ll R^{-4} 
\end{split}
\ee
uniformly in $y$. This bound is useful for $R\geq 1$. For $R<1$ we use the trivial bound $\lambda(\{\ldots\})\leq 1$.
\end{proof}

\section{Limit theorems}\label{sec:limit-theorems}

We now apply the findings of the previous two main sections to prove the invariance principle for theta sums, Theorem \ref{thm-1}. Following the strategy of \cite{Marklof-1999}, we first establish that the random process $X_N(t)$ converges in finite-dimensional distribution to $X(t)$ (Section \ref{section:convergence-fdd}). The proof exploits equidistribution of translated horocycle segments on $\GamG$, which we derive in Section \ref{sec:equi} using theorems of Ratner \cite{Ratner} and Shah \cite{Shah}. Tightness of the sequence of processes $X_N(t)$ is obtained in Section \ref{sec:tightness}; it follows from the uniform tail bound for $\Theta_\chi$ (Section \ref{subs:unif-tail-bound}). Convergence in finite-dimensional distribution and tightness yield the invariance principle for $X_N(t)$. The limiting process $X(t)$ has a convenient geometric interpretation in terms of a random geodesic in $\GamG$ (Section \ref{sec:the-limiting-process}), 
from which the invariance and continuity properties stated in Theorem \ref{thm-2} can be extracted; cf.\ Sections \ref{section: invariance properties}--\ref{section:Holder-continuity}.

\subsection{Equidistribution theorems}\label{sec:equi}
Recall that $G=\widetilde\SL(2,\RR)\ltimes \HH(\RR)$, and define $\ASL(2,\ZZ)=\sltr\ltimes \R^2$ as in \cite{Marklof2003ann}. Throughout this section, we assume that  $\Gamma$ is a lattice in $G$ so that $\Gamma_0=\varphi(\Gamma)$ 
 is commensurable with $\ASL(2,\ZZ)$. 
An example of such $\Gamma$ is the one defined in \eqref{def-Gamma}. We have the following equidistribution theorems.
Let $\Phi^t$ and $\Psi^u$ be as in section \ref{section:geodesic+horocycle+flows}.

\begin{theorem}\label{thmJEQ}
Let $F:\GamG\to\RR$ be bounded continuous and $\lambda$ a Borel probability measure on $\RR$ which is absolutely contiuous with respect to Lebesgue measure. For any $M\in\widetilde\SL(2,\RR)$, $\vecxi\in\RR^2\setminus \QQ^2$ and $\zeta\in\RR$ we have
\begin{equation}
\lim_{t\to\infty} \int_{\RR} F( \Gamma (M; \vecxi, \zeta) \Psi^u \Phi^t ) \de\lambda(u) = \frac{1}{\mu(\GamG)} \int_{\GamG} F(g) \de\mu(g) .
\end{equation}
\end{theorem}

\begin{proof}
By a standard approximation argument it is sufficient to show that, for every $-\infty<a<b<\infty$, the curve
\begin{equation}
C_t = \{ \Gamma (M; \vecxi, \zeta) \Psi^u \Phi^t ) : u\in [a,b] \}
\end{equation}
becomes equdistributed in $\GamG$ with respect to $\mu$. In other words, the uniform probability measure $\nu_t$ on the orbit $C_t$ converges weakly to $\mu$ (appropriately normalised). We know from \cite{Elkies} (cf.\ also \cite{Marklof10}) that for $\vecxi\in\RR^2\setminus \QQ^2$ the projection $\varphi(C_t)$ becomes equidistributed in $\Gamma_0\backslash\ASL(2,\RR)$, where $\Gamma_0=\varphi(\Gamma)$ is commensurable with $\ASL(2,\ZZ)$. Since $\GamG$ is a compact extension of $\Gamma_0\backslash\ASL(2,\RR)$, this imples that (a) the sequence $(\nu_t)_t$ is tight and (b) the support of any possible weak limit $\nu$ projects to $\Gamma_0\backslash\ASL(2,\RR)$. Again a classic argument (cf.\ \cite{Elkies,Shah}) shows that $\nu$ is invariant under the right action of $\Psi^u$. Therefore, by Ratner's measure classification argument \cite{Ratner}, every ergodic component of $\nu$ is supported on the orbit $\Gamma\backslash\Gamma H$ for some closed connected subgroup $H\leq G$. By (b) we know that $\Gamma\backslash\Gamma H$ must project to $\Gamma_0\backslash\ASL(2,\RR)$ and hence $\varphi(H)=G/Z$. The only subgroup $H\leq G$ which satisfies $\varphi(H)=G/Z$ is, however, $H=G$ and hence every ergodic component of $\nu$ equals $\mu$ for any possible weak limit, i.e., $\nu=\mu$ (up to normalisation). Therefore the limit is unique, which in turn implies that every subsequence in $(\nu_t)_t$ converges.
\end{proof}

\begin{theorem}\label{thmJEQ2}
Let $F:\RR\times\GamG\to\RR$ be bounded continuous and $\lambda$ a Borel probability measure on $\RR$ which is absolutely continuous with respect to Lebesgue measure. Let $F_t:\RR\times\GamG\to\RR$ be a family of uniformly bounded, continuous functions so that $F_t\to F$, uniformly on compacta. For any $M\in\widetilde\SL(2,\RR)$, $\vecxi\in\RR^2\setminus \QQ^2$ and $\zeta\in\RR$ we have
\begin{equation}
\lim_{t\to\infty} \int_{\RR} F_t( u, \Gamma (M; \vecxi, \zeta) \Psi^u \Phi^t ) \de\lambda(u) = \frac{1}{\mu(\GamG)}  \int_{\RR\times\GamG} F(u,g) \de\lambda(u) \de\mu(g) .
\end{equation}
\end{theorem}

\begin{proof}
This follows from Theorem \ref{thmJEQ} by a standard argument, see \cite[Theorem 5.3]{Marklof10}.
\end{proof}

\begin{cor}\label{cor-thmJEQ}
Let $F$, $F_t$ and $\lambda$ be as in Theorem \ref{thmJEQ2}. For any $(\alpha,\beta)\in\RR^2\setminus \QQ^2$ and $\zeta,\gamma\in\RR$ we have
\begin{equation}
\lim_{t\to\infty} \int_{\RR} F_t\!\left( u, \Gamma\left(1 ; \sve{\alpha + u \beta}{0}, \zeta+u \gamma\right)\Psi^u \Phi^t\right) \de\lambda(u) = \frac{1}{\mu(\GamG)} \int_{\RR\times\GamG} F(u,g) \de\lambda(u) \de\mu(g) .
\end{equation}
\end{cor}

\begin{proof}
We have
\be
\begin{split}
\left(1 ; \sve{\alpha + u \beta}{0}, \zeta+u \gamma\right)\,\Psi^u \Phi^t
& =\left(1 ; \sve{\alpha}{-\beta}, \zeta \right)\,\Psi^u\; \left(1 ; \sve{0}{\beta}, u \gamma-\beta(\alpha+u\beta)\right)\,\Phi^t \\
&  =\left(1 ; \sve{\alpha}{-\beta}, \zeta\right)\, \Psi^u\Phi^t \, \left(1 ; \sve{0}{e^{-t/2} \beta}, u \gamma-\beta(\alpha+u\beta)\right).
\end{split}
\ee
Define 
\be
\tilde F_t(u,g)= F_t\!\left(u, g \left(1 ; \sve{0}{e^{-t/2} \beta}, u \gamma-\beta(\alpha+u\beta)\right)\right),
\ee
\be
\tilde F(u,g)= F\!\left(u, g\left (1 ; \vecnull , u \gamma-\beta(\alpha+u\beta)\right)\right).
\ee
Since right multiplication is continuous and commutes with left multiplication, we see that, under the assumptions on $F_t$, $\tilde F_t$ is a family of uniformly bounded, continuous functions $\RR\times\GamG\to\RR$ so that $\tilde F_t\to \tilde F$, uniformly on compacta. Theorem \ref{thmJEQ2} therefore yields
\begin{equation}
\lim_{t\to\infty} \int_{\RR} F_t\!\left( u, \Gamma \left(1 ; \sve{\alpha + u \beta}{ 0}, \zeta+u \gamma\right)\Psi^u \Phi^t\right) \de\lambda(u)  = \frac{1}{\mu(\GamG)} \int_{\RR\times\GamG} \tilde F(u,g) \de\lambda(u) \de\mu(g) .
\end{equation}
Finally, the invariance of $\mu$ under right multiplication by $(1;\vecnull,\zeta)$, for any $\zeta\in\RR$, shows that
\begin{equation}
 \int_{\RR\times\GamG} \tilde F(u,g) \de\lambda(u) \de\mu(g)= \frac{1}{\mu(\GamG)} \int_{\RR\times\GamG} F(u,g) \de\lambda(u) \de\mu(g).
\end{equation}
\end{proof}

\subsection{Convergence of finite-dimensional distributions}\label{section:convergence-fdd}
In this section we prove the convergence of the finite dimensional distributions for the random curves \eqref{definition-X_N(x;t)}. This means that, for every $k\geq1$, every $0\leq t_1<t_2<\cdots<t_k\leq 1$, and every bounded continuous function $B:\C^k\to\R$,
\begin{multline}
\lim_{N\ti}\int_{\R}B(X_N(x;t_1),\ldots ,X_N(x;t_k))h(x)\de x\\
=\frac{1}{\mu(\GamG)}\int_{\GamG}B(\sqrt{t_1}\Theta_\chi(g\Phi^{2\log t_1}),\ldots,\sqrt{t_k}\Theta_\chi(g\Phi^{2\log t_k}))\,\de \mu(g)\label{rephrasing-conv-fin-dim}
\end{multline}
Recall that, by (\ref{definition-X_N(x;t)}, \ref{rewriting-SNalpha}), 
\be
X_N(x;t)=\e^{s/4}\Theta_{f}(x+i y\e^{-s},0;\sve{\alpha+c_1x}{0},c_0x)+\frac{\{t N\}}{\sqrt{N}}\left(S_{\lfloor tN\rfloor+1}(x)-S_{\lfloor tN\rfloor}(x)\right),
\ee
where $s=2\log t$, $y=N^{-2}$ and $f=\bm{1}_{(0,1]}$. It will be more convenient to work with
\be
\tilde X_N(x;t)=\e^{s/4}\Theta_{\chi}(x+i y\e^{-s},0;\sve{\alpha+c_1x}{0},c_0x),\ee
where $\chi=\bm{1}_{(0,1)}$.
The difference
$X_N(t)-\tilde X_N(t)$ 
comprises finitely many terms and is thus of order $N^{-1/2}$ uniformly in $t$.
Therefore, it is enough to show that the  limit  \eqref{rephrasing-conv-fin-dim} 
holds for $\tilde X_N(x;t)$ 
 in place of $X_N(x;t)$. 
 We can write 
\begin{equation}
\tilde X_N(x;t)= \e^{s/4}\Theta_{\chi}\!\left(\left(1;\sve{\alpha+c_1 x}{0},c_0x\right)\Psi^{x}\Phi^{\tau}\Phi^{s}\right),
\end{equation}
where $\tau=-\log y$. 

To simplify notation, we will write in the following $\Theta_{f,s}(g):=\e^{s/4}\Theta_f(g\Phi^s)$, $g\in G$. Observe that $\Theta_{f,s}$ is also well defined on $\GamG$. 
Moreover, for every $k\in\N$ and every  $\underline{s}=(s_1,s_2,\ldots,s_k)\in\R^k$ with $s_1<s_2<\dots<s_k$, let us define
\be\Theta_{f,\underline{s}}:\GamG\to\C^k,\hspace{.5cm}\Theta_{f,\underline{s}}(g):=\left(\Theta_{f,s_1}(g),\Theta_{f,s_2}(g),\ldots,\Theta_{f,s_k}(g)\right).\label{def-Theta-array-s}\ee
With this we have, for $0<t_1<t_2<\ldots<t_k\leq 1$,
\begin{align} 
\left(\tilde X_N(x;t_1),\ldots, \tilde X_N(x;t_k)\right)=\Theta_{\chi,\underline{s}}(\left(1;\sve{\alpha+c_1 x}{0},c_0x\right)\Psi^x\Phi^\tau)
\end{align}
with $s_j=2\log t_j$ and $\tau=2\log N$.  
The weak convergence of finite dimensional distribution of the process $\tilde X_N(t)$  stated in 
\eqref{rephrasing-conv-fin-dim} is a consequence of the above discussion and the following

\begin{theorem}\label{convergence-fin-dim-distr-chi}
Let $\lambda$ be a Borel probability measure on $\RR$ which is absolutely continuous with respect to Lebesgue measure. Let $c_1,c_0,\alpha\in\R$ with $(\alpha,c_1)\notin\Q^2$. Then for every $k\geq1$, every $\underline{s}=(s_1,s_2,\ldots,s_k)\in\R^k$, and every bounded continuous function $B:\C^k\to\R$
\begin{align}
\lim_{\tau\to\infty}\int_{\R}B\!\left(\Theta_{\chi,\underline{s}}\!\left((1;\sve{\alpha+c_1 x}{0},c_0x)\Psi^x\Phi^\tau\right)\right)\de\lambda(x)=\frac{1}{\mu(\GamG)}\int_{\GamG}B(\Theta_{\chi,\underline{s}}(g))\de\mu(g).\label{statement-convergence-fin-dim-distr-chi}
\end{align}
\end{theorem}

We first prove a variant of this statement for smooth cut-off sums, with $\Theta_{\chi,\underline{s}}$ replaced $\Theta_{f,\underline{s}}$, where $f\in\mathcal{S}_\eta$, $\eta>1$. 

\begin{lem}\label{convergence-fin-dim-distr-eta>1}
Let $f\in\mathcal{S}_\eta(\R)$, with $\eta>1$. Under the assumptions of Theorem \ref{convergence-fin-dim-distr-chi}, for every $k\geq1$, every $\underline{s}=(s_1,s_2,\ldots,s_k)\in\R^k$, and every bounded continuous function $B:\C^k\to\R$
\begin{align}
\lim_{\tau\to\infty}\int_{\R}B\!\left(\Theta_{f,\underline{s}}\!\left((1;\sve{\alpha+c_1 x}{0},c_0x)
\Psi^x\Phi^\tau\right)\right) \de\lambda(x)=\frac{1}{\mu(\GamG)}\int_{\GamG}B(\Theta_{f,\underline{s}}(g))\de\mu(g).\label{statement-convergence-fin-dim-distr-eta>1}
\end{align}
\end{lem}
\begin{proof}
Apply Corollary \ref{cor-thmJEQ} with $F_t(u,g)=F(u,g)=B(\Theta_{f,\underline{s}}(g))$ (no dependence on $t,u$), $\beta=c_1$, $\gamma=c_0$, and $\zeta=0$.
\end{proof}

Lemma \ref{convergence-fin-dim-distr-eta>1} implies Theorem \ref{convergence-fin-dim-distr-chi} via a standard approximation argument that requires the following lemmta. We first consider the variance. For $\bm z=(z_1,\ldots,z_k)\in\C^k$, let $\|\bm z\|_{\C^k}=(|z_1|^2+\ldots+|z_k|^2)^{1/2}$.

\begin{lem}\label{lem:variance}
Let $f$, $h$ be compactly supported, Riemann-integrable functions on $\R$, and assume $h\geq 0$. Then, for all $\alpha,c_1,c_0\in\R$,
\begin{align}
\limsup_{\tau\to\infty}\int_{\R} \left\|\Theta_{f,\underline{s}}\!\left((1;\sve{\alpha+c_1 x}{0},c_0x)\Psi^x\Phi^\tau\right)\right\|^2_{\C^k} h(x) \de x\leq 2 k \| f \|_{L^2}^2  \| h \|_{L^1}.
\end{align}
\end{lem}

\begin{proof}
By a standard approximation argument, we may assume without loss of generality that $h\in C_c^2(\R)$. The Fourier transform $\widehat h$ of $h$ then satisfies the bound $|\widehat h(y)| \ll |y|^{-2}$. We have, by Parseval's identity,
\begin{equation}\label{pf-approximation-lemma-2}
\begin{split}
\int_{\R} & \left\|\Theta_{f,\underline{s}}\!\left((1;\sve{\alpha+c_1 x}{0},c_0x)\Psi^x\Phi^\tau\right)\right\|^2_{\C^k} h(x) \de x \\
& = y^{\ha} \sum_{j=1}^k \e^{s_j/2}\sum_{n,m\in\Z}f(n y^\ha\e^{-s_j/2})\, f(m y^\ha\e^{-s_j/2})\,e((n-m)\alpha) \,\widehat h\!\left(\tfrac12(m^2-n^2)+c_1(m-n)\right) \\
& \leq y^{\ha} \sum_{j=1}^k \e^{s_j/2}\sum_{n,m\in\Z} \left|f(n y^\ha\e^{-s_j/2})\, f(m y^\ha\e^{-s_j/2})\, \widehat h\!\left(\tfrac12(m^2-n^2)+c_1(m-n)\right) \right| ,
\end{split}
\end{equation}
where $y=\e^{-\tau}$.
Note that $(m^2-n^2)+2c_1(m-n)=(m-n)(m+n+2c_1)=0$ if and only if [$m=n$ or $m=-n-2c_1$].
The sum restricted to $m=n$ is a Riemann sum. In the limit $y\to 0$,
\begin{equation}
y^{\ha} \sum_{j=1}^k \e^{s_j/2}\sum_{n\in\Z} \left|f(n y^\ha\e^{-s_j/2})\, f(n y^\ha\e^{-s_j/2}) \right| 
\to k \| f \|_{L^2}^2  \| h \|_{L^1} .
\end{equation}
Likewise, the sum restricted to $m=-n-2c_1$ yields
\begin{equation}
\begin{split}
y^{\ha} \sum_{j=1}^k \e^{s_j/2}\sum_{n\in\Z} \left|f(n y^\ha\e^{-s_j/2})\, f(-(n+2c_1) y^\ha\e^{-s_j/2}) \right| 
& \to k \int_\RR |f(w) f(-w)| \de w\,  \| h \|_{L^1} \\
& \leq k \| f \|_{L^2}^2  \| h \|_{L^1} .
\end{split}
\end{equation}
The sum of the remaining terms with $(m^2-n^2)+2c_1(m-n)\neq 0$ is bounded above by (set $p=m-n$, $q=m+n$ and overcount by allowing all $p,q\in\Z$ with $p\neq 0$, $q\neq-2c_1$)
\begin{equation}
\begin{split}
& y^{\ha} \sum_{j=1}^k \e^{s_j/2}\sum_{\substack{p,q\in\Z\\ p\neq 0\\ q\neq-2c_1}} \left|f(\tfrac12(q-p) y^\ha\e^{-s_j/2})\, f(\tfrac12(q+p) y^\ha\e^{-s_j/2})\, \widehat h\!\left(\tfrac12 p(q+2c_1)\right) \right| \\
& \ll y^{\ha} \sum_{j=1}^k \e^{s_j/2}\sum_{\substack{p,q\in\Z\\ p\neq 0\\ q\neq-2c_1}}  \left|p(q+2c_1)\right|^{-2}
= O(y^{\ha}) .
\end{split}
\end{equation}
Hence all ``off-diagonal'' contributions vanish as $y\to 0$.
\end{proof}

For the rest of this section assume  that $\lambda$ is a Borel probability measure on $\RR$ which is absolutely continuous with respect to Lebesgue measure.

\begin{lem}\label{lem-Chebyshev}
Let $f$ be a compactly supported, Riemann-integrable function on $\R$. Then, for all $\alpha,c_1,c_0\in\R$, $K>0$,
\begin{align}\label{eq:Markov}
\limsup_{\tau\to\infty}
\lambda\left( \left\{ x\in\RR: \left\|\Theta_{f,\underline{s}}\!\left((1;\sve{\alpha+c_1 x}{0},c_0x)\Psi^x\Phi^\tau\right)\right\|_{\C^k} > K \right\}\right) < \frac{4 k^2 \| f \|_{L^2}^2}{K^2} .
\end{align}
\end{lem}

\begin{proof}
Let us denote by $\lambda'\in L^1(\R)$ the probability density of $\lambda$, and by $m_h$ the measure with density $h\in C_c(\R)$, $h\geq 0$. We have 
\begin{equation}
\begin{split}
& \lambda\left( \left\{ x\in\RR: \left\|\Theta_{f,\underline{s}}\!\left((1;\sve{\alpha+c_1 x}{0},c_0x)\Psi^x\Phi^\tau\right)\right\|_{\C^k} > K \right\}\right) \\
& \leq m_h\left( \left\{ x\in\RR: \left\|\Theta_{f,\underline{s}}\!\left((1;\sve{\alpha+c_1 x}{0},c_0x)\Psi^x\Phi^\tau\right)\right\|_{\C^k} > K \right\}\right) + \|\lambda'-h\|_{L^1} \\
& < 4 k^2 K^{-2} \| f \|_{L^2}^2 \| h \|_{L^1} + \|\lambda'-h\|_{L^1}
\end{split}
\end{equation}
by Chebyshev's inequality and Lemma  \ref{lem:variance}. Since $C_c(\R)$ is dense in $L^1(\R)$, rel.~\eqref{eq:Markov} follows.
\end{proof}
\begin{lem}\label{tight-Helly}
For all $\alpha,c_1,c_0\in\R$ and  $\varepsilon>0$ there exists  a constant $K_\varepsilon>0$ such that
\begin{align}
\limsup_{\tau\to\infty} \lambda\!\left(\left\{x\in\R:\: \left\|\Theta_{f,\underline{s}}\!\left((1;\sve{\alpha+c_1 x}{0},c_0x)\Psi^x\Phi^\tau\right)\right\|_{\C^k} > K_\varepsilon\right\}\right)\leq\varepsilon \|f\|_{L^2}^2
\end{align}
for every compactly supported, Riemann-integrable $f$. 
\end{lem}
\begin{proof}
It follows immediately from Lemma \ref{lem-Chebyshev}.
\end{proof}
\begin{lem}\label{eps-delta-f}
Let $f$ be a compactly supported, Riemann-integrable function on $\R$. Then, for every $\varepsilon>0$, $\delta>0$ there exists $\tilde f\in\scrS_2(\R)$ with compact support such that
\begin{align}
&\limsup_{\tau\to\infty}\lambda\!\left(\left\{x\in\R: \|\Theta_{f,\underline{s}}\!\left((1;\sve{\alpha+c_1 x}{0},c_0x)\Psi^x\Phi^\tau\right)-\Theta_{\tilde f,\underline{s}}\!\left((1;\sve{\alpha+c_1 x}{0},c_0x)\Psi^x\Phi^\tau\right)\|_{\C^k}>\delta\right\}\right)\nonumber\\
&< \varepsilon.
\end{align}
\end{lem}
\begin{proof}
Note that $\Theta_{f,\underline{s}}-\Theta_{\tilde f,\underline s}=\Theta_{f-\tilde f,\underline s}$. Since $f-\tilde f$ is compactly supported and Riemann-integrable, we can apply Lemma  \ref{lem-Chebyshev} 
with $f-\tilde f$ in place of $f$. Choose $K=\epsilon$ and $\tilde f$ so that
\begin{equation}
\frac{4 k^2 \| f-\tilde f \|_{L^2}^2}{K^2} \leq \delta,
\end{equation}
which is possible since $\scrS_2(\R)$ is dense in $L^2(\R)$.
The claim follows.
\end{proof}

\begin{proof}[Proof of Theorem \ref{convergence-fin-dim-distr-chi}]
Lemma \ref{tight-Helly} and Helly-Prokhorov's Theorem imply that every sequence $(\tau_j)_{j\geq 1}$ such that $\tau_j\to\infty$ as $j\to\infty$ has a subsequence $(\tau_{j_l})_{l\geq 1}$ with the following property: there is a probability measure $\nu$ on $\C$ such that for every bounded continuous function $B:\C^k\to\C$ we have
\begin{align}
\lim_{l\to\infty}\int_{\R}B\!\left(\Theta_{\chi,\underline{s}}\!\left((1;\sve{\alpha+c_1 x}{0},c_0x)\Psi^x\Phi^{\tau_{j_l}}\right)\right)\de\lambda(x)=\int_{\C}B(\bm z)\de\nu(\bm z).\label{pf-conv-fin-dim-distr-1}
\end{align}
The measure $\nu$ may of course depend on the choice of the subsequence. To identify that measure, we restrict to test functions $B\in C_c^\infty(\C^k)$.
We claim that for such $B$ the limit 
\begin{align}
I(\chi)=\lim_{j\to\infty}\int_{\R}B\!\left(\Theta_{\chi,\underline{s}}\!\left((1;\sve{\alpha+c_1 x}{0},c_0x)\Psi^x\Phi^{\tau_{j}}\right)\right)\de\lambda(x)\label{limit-to-show-it-exists}
\end{align}
exists. To prove this, let us first notice that, since $B\in C_c^\infty(\C^k)$, it is Lipschitz, i.e. $|B(\bm z')-B(\bm z'')|\leq C\|\bm z'-\bm z''\|_{\C^k}$ for some constant $C>0$. Therefore, by Lemma \ref{eps-delta-f}, for every $\varepsilon>0$, $\delta>0$, we can find a compactly supported $f\in\mathcal S_\eta(\R)$ with $\eta>1$ such that
\be\label{C-eps-delta}
\begin{split}
&\int_{\R}|B\!\left(\Theta_{\chi,\underline{s}}\!\left((1;\sve{\alpha+c_1 x}{0},c_0x)\Psi^x\Phi^{\tau_{j}}\right)\right)-B\!\left(\Theta_{f,\underline{s}}\!\left((1;\sve{\alpha+c_1 x}{0},c_0x)\Psi^x\Phi^{\tau_{j}}\right)\right)|\de\lambda(x)\\
&\leq C\int_{\R}\|\Theta_{\chi,\underline{s}}\!\left((1;\sve{\alpha+c_1 x}{0},c_0x)\Psi^x\Phi^{\tau_{j}}\right)-\Theta_{f,\underline{s}}\!\left((1;\sve{\alpha+c_1 x}{0},c_0x)\Psi^x\Phi^{\tau_{j}}\right)\|_{\C^k}\de\lambda(x)\\
&\leq C(\varepsilon+\delta).
\end{split}
\ee
Since the limit 
\begin{align}
I(f)=\lim_{j\to\infty}\int_{\R}B\!\left(\Theta_{f,\underline{s}}\!\left((1;\sve{\alpha+c_1 x}{0},c_0x)\Psi^x\Phi^{\tau_{j}}\right)\right)\de\lambda(x)
\end{align}
exists by Lemma \ref{convergence-fin-dim-distr-eta>1}, the sequence 
\begin{align}
\left(\int_{\R}B\!\left(\Theta_{f,\underline{s}}\!\left((1;\sve{\alpha+c_1 x}{0},c_0x)\Psi^x\Phi^{\tau_{j}}\right)\right)\de\lambda(x)\right)_{j\geq 1}
\end{align}
is Cauchy. Using this fact, the bound \eqref{C-eps-delta} and the triangle inequality, we see that
\begin{align}
\left(\int_{\R}B\!\left(\Theta_{\chi,\underline{s}}\!\left((1;\sve{\alpha+c_1 x}{0},c_0x)\Psi^x\Phi^{\tau_{j}}\right)\right)\de\lambda(x)\right)_{j\geq 1}
\end{align}
is a Cauchy sequence, too, and hence the limit $I(\chi)$ 
exists as claimed. The bound \eqref{C-eps-delta} implies that $I(f)\to I(\chi)$ as $f\to\chi$ in $L^2$ and therefore the right-hand side of \eqref{pf-conv-fin-dim-distr-1} must equal the right-hand side of \eqref{statement-convergence-fin-dim-distr-chi}. 

We have now established that, for any convergent subsequence, the weak limit $\nu$ in \eqref{pf-conv-fin-dim-distr-1} in fact unique, i.e.\ the same for every converging subsequence. This means that every subsequence converges---in particular the full sequence. This concludes the proof of the theorem.
\end{proof}

\subsection{Tightness}\label{sec:tightness}

The purpose of this section is to prove that the family of processes $\{X_N\}_{N\geq 1}$ is tight. Recall that each  $X_N$ is a random variable with values on the Polish space $(\mathcal{C}_0,d)$, cf. \eqref{definition-X_N(x;t)}.
If we denote by $\mathbb{P}_N$ the probability measure induced by $X_N$ on $\mathcal C_0$, then tightness means that for every $\varepsilon>0$ there exists a compact set $\mathcal K_\varepsilon\subset \mathcal{C}_0$ with $\mathbb{P}_N(\mathcal K_\varepsilon)>1-\varepsilon$ for every $N\geq 1$. We prove the following

\begin{prop}\label{thm-tightness}
The sequence $\{X_ N\}_{N\geq 1}$ is tight.
\end{prop}
\begin{proof}
For every $K>0$ and every positive integer $N$ set
\begin{align}
\mathcal {M}_{K,N}=\left\{x\in\R:\: \exists\, 
m\geq 1 \mbox{ and }0\leq k<2^m \mbox{ s.t. } 
\left|  X_N\!\left(\frac{k+1}{2^m}\right)- X_N\!\left(\frac{k}{2^m}\right)\right|> \frac{K}{m^2}  \right\},\label{def-set-M_C,N}
\end{align}
We will prove the tightness of the process $t\mapsto X_N(t)$ by establishing that
\begin{align}
\lim_{K\to \infty}\sup_{N\geq1}\lambda(\mathcal{M}_{K,N})=0.\label{this-implies-tightness}
\end{align}
To see that \eqref{this-implies-tightness} is equivalent to tightness, recall how the curve $t\mapsto X_N(t)$ depends on $x$ (cf. \eqref{theta-sum-intro}, \eqref{definition-X_N(x;t)}) and observe that functions in the set
\begin{align}
\mathcal C_0\smallsetminus \bigcup_{K>0}\{X_N:[0,1]\to\C\ \: | \:\: x\in \mathcal \mathcal{M}_{K,N}\}
\end{align}
are uniformly equicontinuous on a dense set (of dyadic rationals). Since in our case $X_N(0)=0$ by definition, uniform equicontinuity is equivalent to tightness (see \cite{Billingsley}, Theorem 7.3).

Let us now show \eqref{this-implies-tightness}.
By construction,
\begin{align}
\left|  X_N\!\left(\frac{k+1}{2^m}\right)- X_N\!\left(\frac{k}{2^m}\right)\right|\leq \frac{\sqrt{N}}{2^m}.
\end{align}
Therefore, if $m>0$ is such that 
$\frac{\sqrt{N}}{2^m}\leq \frac{1}{m^2}$ for all $N\geq 1$, then the inequality defining \eqref{def-set-M_C,N} has no solution for all $K>1$. For $N=1$ the inequality is $\frac{m^2}{2^m}\leq1$ is valid for every $m\geq 4$. For $N\geq 2$ a sufficient condition for  $\frac{m^2}{2^m}\leq \frac{1}{\sqrt{N}}$ is  $m> 5\log_{2}N$.
Thus, it is enough to restrict the range of $m$ in \eqref{def-set-M_C,N} to $1\leq m\leq 5 \log_2 N$.  For these values of $m$, let us estimate the measure of $\mathcal M_{K,N}$ from above by estimating the measure of 
\be
\left\{x\in\R:\: 
\left|  X_N\!\left(\frac{k+1}{2^m}\right)- X_N\!\left(\frac{k}{2^m}\right)\right|> \frac{K}{m^2}  \right\}
\ee
for fixed $m$ and $k$.
Define 
$N_1=\frac{k}{2^m}N$ and 
$N_2=\frac{k+1}{2^m}N$.  
Recall \eqref{rewriting-SNalpha}, and let us observe that
\be\label{S_N2-S_N1=Theta}
\begin{split}
\sqrt{N}\left(X_N\!\left(\frac{k+1}{2^m}\right)- X_N\!\left(\frac{k}{2^m}\right)\right)&=
\sum_{N_1<n\leq N_2} e\!\left((\tha n^2+c_1 n+c_0)x+\alpha x\right)+O(1)\\
&=y^{-\frac{1}{4}}\Theta_\chi\!\left(x+i y,0;\sve{\alpha+c_1 x}{N_1},c_0x+\zeta'\right)+O(1),
\end{split}
\ee
with $y=\frac{1}{(N_2-N_1)^2}$ and for some $\zeta'\in\R$, and where the $O$-term is bounded in absolute value by 2.
We have
\begin{multline}
\lambda\left\{x\in\R:\: 
\left|  X_N\!\left(\frac{k+1}{2^m}\right)- X_N\!\left(\frac{k}{2^m}\right)\right|> \frac{K}{m^2}  \right\}\\
\leq\lambda\left\{x\in\R:\: \sqrt{N_2-N_1}\left|\Theta_\chi(x+i\tfrac{1}{(N_2-N_1)^2},0;\sve{\alpha+c_1x}{N_1},c_0x+\zeta')\right|>\frac{K \sqrt N}{m^2}-2\right\}.
\end{multline}
Observe that $\frac{K \sqrt N}{m^2}-2>\frac{K \sqrt N}{2m^2}$ for $m<\frac{\sqrt{K}N^{1/4}}{2}$ and for sufficiently large $K$ (uniformly in $m,k$), the inequality $\frac{\sqrt{K}N^{1/4}}{2}>5\log_2 N$ holds true for all $N\geq 2$.
Now we apply Proposition \ref{lem-uniform-tail-estimate-chi}: 
\be
\begin{split}
&\lambda\left\{x\in\R:\: 
\left|  X_N\!\left(\frac{k+1}{2^m}\right)- X_N\!\left(\frac{k}{2^m}\right)\right|> \frac{K}{m^2}  \right\}\\
&\leq\lambda\left\{x\in\R:\: \sqrt{N_2-N_1}\left|\Theta_\chi(x+i\tfrac{1}{(N_2-N_1)^2},0;\sve{\alpha+c_1x}{N_1},c_0x+\zeta')\right|>\frac{K \sqrt N}{2 m^2}\right\}\\
&\ll \left(1+\frac{K}{2m^2}\sqrt{\frac{N}{N_2-N_1}}\right)^{-4}\ll K^{-4} m^8 \!\left(\frac{N_2-N_1}{N}\right)^2\ll K^{-4}2^{-2m}m^8.
\end{split}
\ee
Using the fact that 
$\sum_{m=1}^\infty 2^{-m} m^8$ is finite, we have
\begin{align}
\lambda(\mathcal{M}_{K,N})&\ll \sum_{m\leq 5\log_2 N}
\sum_{k=0}^{2^m-1} K^{-4}2^{-2m}m^8
\ll K^{-4} \sum_{m=1}^{\infty} 2^{-m}m^8\ll K^{-4},
\end{align}
uniformly in $N\geq 2$. Taking the limit as $K\to\infty$ concludes the proof of the Proposition. 
\end{proof}

\subsection{The limiting process}\label{sec:the-limiting-process}
Convergence of finite dimensional limiting distributions (from Theorem \ref{convergence-fin-dim-distr-chi}) and tightness (Proposition \ref{thm-tightness}) imply that there exists a random process $[0,1]\ni t\mapsto X(t)\in \C$ such that 
\be X_N
\Longrightarrow X\hspace{.3cm}\mbox{as $N\ti$}\ee 
where ``$\Rightarrow$'' denotes weak convergence in the Wiener space $\mathcal C_0$. 
This shows part (ii) of Theorem \ref{thm-1}. Part (i) follows from Corollary \ref{cor-L2L4}.
By \eqref{statement-convergence-fin-dim-distr-chi}, we can be more precise and write  the limiting process explicitly as a $\mathcal C_0$-valued measurable function on the probability space $(\GamG,\frac{3}{\pi^2}\mu)$, where $\mu$ is the Haar measure \eqref{Haar-measure-on-G} on the homogeneous space $\GamG$. We have 
\be
\left(\GamG,\tfrac{3}{\pi^2}\mu\right)\ni g\mapsto
X\in \mathcal C_0,\hspace{.5cm}X(t)=\begin{cases}0&\mbox{$t=0$}\\ \e^{s/4}\Theta_\chi(\Gamma g\Phi^s)&\mbox{$t>0$}\end{cases}\label{formula-limiting-process}
\ee
where 
$s=2\log t$. 
In other words, the curves of our random process are images (via the automorphic function $\Theta_\chi$ discussed in Section \ref{section:dyadic-representation-Theta-sharp-cutoff}) of geodesic paths in $\GamG$, rescaled by the function $\e^{s/4}=\sqrt t$, where the ``randomness'' comes from the choice of $g\in\GamG$ according to the normalized Haar measure $\frac{3}{\pi^2}\mu$.
Moreover, we can extend our process, a priori defined only for $0\leq t\leq 1$, to all  $t\geq 0$ by means of the formula \eqref{formula-limiting-process}.

Notice that the function $\Theta_\chi$, discussed in Section \ref{section:dyadic-representation-Theta-sharp-cutoff}, is not defined \emph{everywere}. 
However, we are only interested in the value of $\Theta_\chi$ \emph{along geodesics} starting at \emph{$\mu$-almost any point $\Gamma g\in\GamG$}. 
One can check that
\bey
\left(x+i y,\phi;\vecxi,\zeta\right)\Phi^s&=&\left(x_s+i y_s,\phi_s,\vecxi,\zeta\right),\label{gPhi^s=g_s}
\eey
where
\begin{align}
x_s&=x+\frac{y}{\cot 2\phi+\coth s\csc2\phi},\label{xt}\\
y_s&=\frac{y}{\cosh s+\cos 2\phi \sinh s },\label{yt}\\
\phi_s&=2k\pi+\epsilon_1 \arccos\!\left(\epsilon_2\frac{\sqrt2 \e^s\cos\phi}{\sqrt{1+\e^{2s}+(\e^{2s}-1)\cos2\phi}}\right)\hspace{.2cm}\mbox{if $(2k-1)\pi\leq \phi<(2k+1)\pi$}.\hspace{1cm}\label{phi-t}
\end{align}
In the above formula  $\arccos:[-1,1]\to[-\pi,\pi]$ and
\bey
(\epsilon_1,\epsilon_2)=\begin{cases} 
(-1,-1),&(2k-1)\pi\leq\phi<(2k-\tha)\pi;\\
(-1,+1),&(2k-\tha)\pi\leq \phi<2k\pi;\\
(+1,+1),&2k\pi\leq\phi<(2k+\tha)\pi;\\
(+1,-1),&(2k+\tha)\pi\leq \phi<(2k+1)\pi. \end{cases}
\eey
Moreover, the values at $\phi=2k\pi$ (resp. $\phi=(2k\pm\ha)\pi$) are understood as limits as $\phi\to2k\pi$ (resp. $\phi\to(2k\pm\ha)\pi$), at which we get $(x_s,y_s,\phi_s)=(x,\e^{-s}y,2k\pi)$ (resp. $(x_s,y_s,\phi_s)=(x,\e^sy,(2k\pm\ha)\pi)$).
For every $s\in\R$ the function $\R\to\R$, $\phi\mapsto\phi_s$ is a bijection, see Figure \ref{phit}. 
%
\begin{figure}[h!]
\hspace{3cm}\includegraphics[width=10cm]{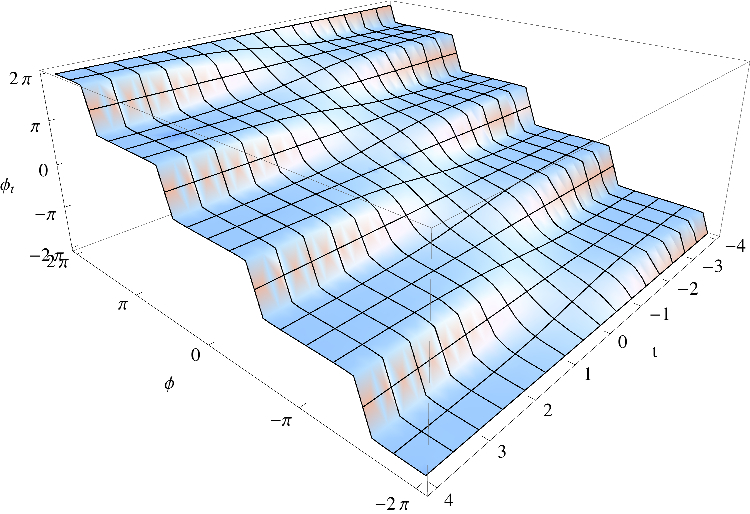}
\caption{The function $\phi_s$ for $-2\pi\leq\phi\leq 2\pi$ and $-4\leq s\leq 4$.}\label{phit}
\end{figure}
%
Moreover, for $\phi\notin\tfrac{\pi}{2}\Z$
\be
(x_s,y_s,\phi_s)\longrightarrow
\begin{cases}
(x+y\tan\phi,0,\lfloor\tfrac{\phi}{\pi}+\tha\rfloor\pi)&\mbox{as $s\to+\infty$},\\
(x-y\cot\phi,0,(\lfloor \tfrac{\phi}{\pi}\rfloor+\ha)\pi)&\mbox{as $s\to-\infty$}.
\end{cases}\ee
It follows from \eqref{gPhi^s=g_s} 
and Theorem \ref{thm:2} 
that for $\mu$-almost every $\Gamma g\in\GamG$, the function $\Theta_\chi(\Gamma g\Phi^s)$ is well defined for all $s\in\R$. Since $s=2\log t$, then the \emph{typical}  curve $t\mapsto X(t)$ process is well defined \emph{for every $t\geq0$}.
The explicit representation \eqref{formula-limiting-process} of the process $X(t)$ allows us to deduce several properties of its typical realizations. These properties reflect those of the geodesic flow $\Phi^s$ on $\GamG$.

Let us remark that part (i) and (ii) of Theorem \ref{thm-2} are simply a restatement of Theorem \ref{tail-asymptotics-for-Theta_chi} and Theorem \ref{convergence-fin-dim-distr-chi}, respectively.

\subsection{Invariance properties}\label{section: invariance properties}
By \emph{scaling invariance} of the theta process we refer to a family of time-changes that leave the distribution of the process $t\mapsto X(t)$ unchanged. In this section we show parts (iii)-(vii) of Theorem \ref{thm-2}.

\begin{lem}[Scaling invariance]\label{lem-scaling-invariance}
Let $X$ denote the theta process. Let $a>0$, then the process $\{Y(t)\colon t\geq0\}$ defined by $Y(t)=\frac{1}{a}X(a^2 t)$ is also a theta process.
\end{lem}
\begin{proof}
By \eqref{formula-limiting-process}
\begin{align}
Y(t)=\frac{1}{a}X(a^2 t)=\frac{1}{a}\e^{2 \log(a^2t)/4}\Theta_{\chi}(\Gamma g \Phi^{2\log a^2}\Phi^{2\log t})=\e^{s/4}\Theta_\chi(\Gamma g'\Phi^s),
\end{align}
where $s=2\log t$ and $g'=g\Phi^{\log a^4}$. By right-invariance of the Haar measure, if $g$ is distributed according to the normalized Haar measure $\frac{\mu}{\pi^2/3}$, then $g'$ is also distributed according to the same measure. Therefore the processes $X$ and $Y$ have the same distribution.

\end{proof}
Another time change that leaves the distribution of $\{X(t)\colon t\geq0\}$ is  unchanged after the rescaling  $t\mapsto 1/t$. This is called \emph{$t$-time-inversion} and is related to the \emph{$s$-time-reversal symmetry} for the geodesic flow $\Phi^s$ on $\GamG$.
\begin{prop}[Time inversion]\label{prop-time-inversion}
 The process $\{Y(t)\colon t\geq 0\}$ defined by
\be
Y(t)=\begin{cases}
      0&t=0;\\
      t X(1/t)&t>0,
     \end{cases}
\ee
is also a theta process.
\end{prop}
\begin{proof}
Observe that $g\Phi^s=gh\Phi^{-s}$, where
\be h=(i,\pi/2;\bm0,0)=\left(\sma{0}{-1}{1}{0}\!,\arg;\bm0,0\right)\ee
corresponds to the $s$-time-reversal symmetry for geodesics. The proposition then follows immediately by \eqref{formula-limiting-process} and the right-invariance of  the normalized Haar measure.
\end{proof}

Like in the case of Wiener process, time-inversion can be used to relate properties of sample paths in a neighborhood of time $t=0$ to properties at infinity. An example is the following
\begin{cor}[Law of large numbers]\label{cor-lln}
 Almost surely
\be \lim_{t\ti}\frac{X(t)}{t}=0.\ee
\end{cor}
\begin{proof}
 Let $Y$ be defined as in the proof of Proposition \ref{prop-time-inversion}. Then, applying this proposition, we have that $\lim_{t\ti}X(t)/t=\lim_{t\ti}Y(1/t)=Y(0)=0$ almost surely.
\end{proof}

We want to prove another basic property of the theta process, its \emph{stationarity}, i.e. the fact that any time-shift also leaves the distribution of the process unchanged.

\begin{theorem}[Stationarity]\label{thm-stationarity}
Fix $t_0\geq0$. Consider the process 
\be
Y(t)=X(t_0+t)-X(t_0).
\ee
Then $\{Y(t)\colon t\geq0 \}$ is also a theta process
\end{theorem}
\begin{proof}
A crucial observation is that we can write the term $\frac{X(t_0+t)-X(t_0)}{\sqrt t}$ as a theta function. In fact, for $t>0$ and $s=2\log t$ we have
\be
\begin{split}\label{chi_(0,t]-transformation}
[R(\Phi^s)\chi](w)& =[R(i\e^{-s},0;\bm 0,0)\chi](w)=\e^{-s/4}[R(i,0;\bm 0,0)\chi](\e^{-s/2}w)\\
& =\e^{-s/4}\chi(\e^{-s/2}w)=\e^{-s/4}\chi(\tfrac{w}{t})=\e^{-s/4}\bm1_{(0,t)}(w).
\end{split}
\ee
This implies
\be
\begin{split}
X(t)& =\e^{s/4}\Theta_\chi(\Gamma g \Phi^s)=\e^{s/4}\sum_{n\in\Z}[R(g\Phi^s)\chi](n)=\e^{s/4}\sum_{n\in\Z}[R(g)R(\Phi^s)\chi](n)\\
& =\sum_{n\in\Z}[R(g)\bm1_{(0,t)}](n)=\Theta_{\bm1_{(0,t)}}(\Gamma g)
\end{split}
\ee
and
\begin{align}
X(t_0+t)-X(t_0)=\sum_{n\in\Z}[R(g)\bm1_{(t_0,t_0+t)}](n)=\Theta_{\bm1_{(t_0,t_0+t)}}(\Gamma g).
\end{align}
Now, by \eqref{chi_(0,t]-transformation}, 
\be
\begin{split}
\chi_{(t_0,t_0+t)}(w)&=\bm1_{(0,t)}(w-t_0)=\e^{s/4}[R(\Phi^s)\chi](w-t_0)\\
&=\sqrt t\left[W\!\left(\sve{0}{t_0},0\right)R(\Phi^s)\chi\right](w)\\
&=\sqrt t\left[R(i,0,\sve{0}{t_0},0)R(\Phi^s_0)\chi\right](w)
\end{split}
\ee
and therefore
\be\label{pf-mod-of-cont-1}
\begin{split}
Y(t)&=X(t_0+t)-X(t_0)=\sqrt t \sum_{n\in\Z}\left[R(g)R(i,0;\sve{0}{t_0},0)R(\Phi^s_0)\chi \right](n)\\
&=\e^{s/4}\Theta_\chi(\Gamma g'\Phi^s)
\end{split}
\ee
where $s=2\log t$ and $g'=g(i,0;\sve{0}{t_0},0)$. Using the right-invariance of the normalized Haar measure as before, we get the desired statement.
\end{proof}

\begin{theorem}[Rotational Invariance]\label{thm-rotational-invariance}
Fix $\theta\in\R$ and consider the process $Y(t)=\e^{2\pi i \theta}X(t)$. Then $Y\sim X$.
\end{theorem}
\begin{proof}
Observe that
\begin{align}
Y(t)=e(\theta)\sqrt t\,\Theta_\chi(\Gamma g\Phi^{s})=\sqrt t\,\Theta_\chi(\Gamma g(i,0;\bm0,\theta)\Phi^{s})
\end{align}
and use the right-invariance of the normalized Haar measure as before.
\end{proof}

\subsection{Continuity properties}
\label{section:Holder-continuity}
In this section we prove parts (viii)-(x) of Theorem \ref{thm-2}.
By definition of the curves $t\mapsto X_N(t)$ and tightness, we already know that typical realizations of the theta process are continuous. In particular, for $0\leq t\leq 1$ (or any other compact interval) typical realizations are uniformly continuous, i.e. there exists some function $\varphi$ (depending on the realization) with $\lim_{ h\downarrow0}\varphi(t)=0$, called a \emph{modulus of continuity} of $X:[0,1]\to\C$, such that
\be\limsup_{h\downarrow0}\sup_{0\leq t\leq 1-h}\frac{|X(t+h)-X(t)|}{\varphi(h)}\leq 1\ee
If $X$ above is replaced by a Wiener process, then a classical theorem by L\'evy \cite{levy1937theorie} states that there exists a deterministic modulus of continuity, namely $\varphi(h)=\sqrt{2 h\log(1/h)}$, for almost every realization. For the theta process, a similar result is true, but with a smaller exponent on the logarithmic factor. This result follows from the representation of $\frac{X(t+h)-X(t)}{\sqrt h}$ as a theta function as before, and a logarithm law for the geodesic flow proved by Sullivan \cite{Sullivan-1982} and by Kleinbock and Margulis \cite{Kleinbock-Margulis-1999} in very general setting.

\begin{theorem}[Modulus of continuity]\label{thm-modulus-of-continuity}
For every $\varepsilon>0$ there exists a constant $C_\varepsilon>$ such that, almost surely, for every sufficiently small $h>0$ and all $0\leq t\leq 1-h$,
\be|X(t+h)-X(t)|\leq C\sqrt{h}(\log(1/h))^{1/4+\varepsilon}.\ee
\end{theorem}
\begin{proof}
Let us use the representation \eqref{pf-mod-of-cont-1} and write
\be
\frac{|X(t+h)-X(t)|}{\sqrt h}=\left|\Theta_{\chi}(\Gamma g\, (i,0;\sve{0}{t},0) \Phi^r)\right|\label{pf-mod-of-cont-1b}
\ee
where $r=2\log h$.
Theorem \ref{thm-modulus-of-continuity} thus reduces to a bound on the right hand side of \eqref{pf-mod-of-cont-1b} for almost every $\Gamma g\in\GamG$. We obtain this bound by using the dyadic decomposition \eqref{def-Theta-chi},
\bey \label{def-Theta-chi222}
\begin{split} 
\Theta_{\chi}(g\, (i,0;\sve{0}{t},0) \Phi^r) &=\sum_{j=0}^\infty 2^{-j/2}\Theta_\Delta(\Gamma g\, (i,0;\sve{0}{t},0) \Phi^{r-(2\log 2)j}) \\ 
& +
\sum_{j=0}^\infty2^{-j/2}\Theta_{\Delta_-}(\Gamma g\, (i,0;\sve{0}{t},0) \Phi^r (1;\sve{0}{1},0)\Phi^{-(2\log 2)j}), 
\end{split}
\eey
and estimating each summand. 
For $g$, $t$ fixed, let us set
\begin{equation}
(z_s,\phi_s;\vecxi_s,\zeta_s)=g\, (i,0;\sve{0}{t},0) \Phi^{s}
\end{equation}
and define 
\begin{equation}
\tilde y_s = \sup_{M\in\SL(2,\Z)} \Im(M z_s)
\end{equation}
as the height in the cusp of our trajectory at time $s$ (cf.~\eqref{3.39}). By Lemma \ref{lem-expansion-at-infinity}, we have
\begin{equation}
|\Theta_\Delta(\Gamma g\, (i,0;\sve{0}{t},0) \Phi^{r-(2\log 2)j})| \ll \tilde y_{r-(2\log 2)j}^{1/4},
\end{equation}
\begin{equation}
|\Theta_{\Delta_-}(\Gamma g\, (i,0;\sve{0}{t},0) \Phi^r (1;\sve{0}{1},0)\Phi^{-(2\log 2)j})| \ll \tilde y_{r-(2\log 2)j}^{1/4}.
\end{equation}
If $d_{\SL(2,\Z)}$ denotes the standard Riemannian metric on the modular surface $\SL(2,\Z)\backslash\frak H$, then
\begin{equation}
d_{\SL(2,\Z)}(i , x+i y) \sim \log y \qquad (y\to\infty).
\end{equation}
Kleinbock and Margulis \cite{Kleinbock-Margulis-1999} (cf.\ also Sullivan \cite{Sullivan-1982}) show that for almost every $g$ and every $\varepsilon>0$ ,
\begin{align}
d_{\SL(2,\Z)}(i,x_s+i y_s)\geq (1-\varepsilon)\log|s|&&\mbox{infinitely often}\\
d_{\SL(2,\Z)}(i,x_s+i y_s)\leq (1+\varepsilon)\log|s|&&\mbox{for all sufficienly large $|s|$}\label{log-law-suff-large} .
\end{align}
Thus
\begin{equation}
|\Theta_\Delta(\Gamma g\, (i,0;\sve{0}{t},0) \Phi^{r-(2\log 2)j})| \ll 
\max( |r-(2\log 2)j|^{\frac14(1+\epsilon)}, 1),
\end{equation}
\begin{equation}
|\Theta_{\Delta_-}(\Gamma g\, (i,0;\sve{0}{t},0) \Phi^r (1;\sve{0}{1},0)\Phi^{-(2\log 2)j})| \ll 
\max( |r-(2\log 2)j|^{\frac14(1+\epsilon)}, 1).
\end{equation}
In view of \eqref{def-Theta-chi222}, this yields
\be\left|\Theta_\chi(g(i,0;\sve{0}{t},0)\Phi^r)\right|\ll |r|^{\frac14(1+\varepsilon)}.\label{pf-mod-of-cont-6}\ee
By recalling that $r=2\log h$, the proof of Theorem \ref{thm-modulus-of-continuity} follows from \eqref{pf-mod-of-cont-1b}, and \eqref{pf-mod-of-cont-6}.
\end{proof}

Despite the  unusual modulus of continuity $\sqrt h(\log(1/h))^{1/4+\varepsilon}$, we can prove that typical realizations of the theta process are $\theta$-H\"older continuous for any $\theta<1/2$. This result is completely analogous to the one for Wiener process sample paths due to L\'evy \cite{levy1937theorie}.

\begin{cor}[H\"older continuity]\label{thm-Holder-continuity}
 If $\theta<1/2$, then, almost surely the theta process is everywhere locally $\theta$-H\"older continuous, i.e. for every $t\geq 0$ there exists $\delta>0$ and $C>0$ such that
\begin{align}
|X(t)-X(t')|\leq C|t-t'|^\theta\hspace{.5cm}\mbox{for all $t'\geq0$ with $|t-t'|<\delta$}.
\end{align}
\end{cor}
\begin{proof}
 Let $C>0$ be as in Theorem \ref{thm-modulus-of-continuity}. Applying this theorem to the theta process $\{X(t)-X(k)\colon t\in[k,k+1]\}$ (recall Theorem \ref{thm-stationarity})  
 where $k$ is a nonnegative integer, we get that, almost surely, for every $k$ there exists $h_+(k)>0$ such that for all $t\in[k,k+1)$ and $0<h<(k+1-t)\wedge h_+(k)$,
\be |X(t+h)-X(t)|\leq C\sqrt h(\log(1/h))^{1/4+\varepsilon}\leq C h^\theta.\ee
On the other hand, by applying the same argument to the theta process $\{X(k+1-t)-X(k+1)\colon t\in[k,k+1]\}$ we obtain that, almost surely, for every $k$ there exists $h_-(k)>0$ such that for all $k<t\leq k+1$ and $0<h<(t-k)\wedge h_-(k)$
\be |X(t)-X(t-h)|\leq C\sqrt h(\log(1/h))^{1/4+\varepsilon}\leq C h^\theta.\ee
The desired result now follows immediately.
\end{proof}

Now that we know that typical realization of the theta process are $\theta$-H\"older continuous for every $\theta<1/2$, it is natural to ask whether they enjoy stronger regularity properties. In particular, we study differentiability at any fixed time. We will show that for every fixed $t_0\geq0$, typical realizations of $\{X(t)\colon t\geq0\}$ are not differentiable at $t_0$. By stationarity it is enough to consider differentiability at $t_0=0$. Then, by time inversion, we relate differentiability at $0$ to a long-term property. This property is parallel to the law of large numbers: whereas Corollary \ref{cor-lln} states that typical trajectories of the theta process grow less then linearly, the following proposition states that the limsup growth of $|X(t)|$ is almost surely faster than $\sqrt t$.

\begin{prop}\label{limsup-growth}
Almost surely
\be\limsup_{t\ti}\frac{|X(t)|}{\sqrt t}=+\infty .
\label{statement-limsup-growth}\ee
\end{prop}

\begin{proof}
By \eqref{formula-limiting-process}, for $s=2\log t$
\be\frac{|X(t)|}{\sqrt t}=|\Theta_\chi(g\Phi^s)|\ee
with $g\in\GamG$.
Since $\Phi^\R$ is ergodic (cf.\ \cite[Prop.~2.2]{Kleinbock-1999}), we have by the Birkhoff ergodic theorem, for any $R>0$ and $\mu$-almost every $g\in\GamG$,
\begin{equation}
\lim_{T\to\infty} \frac1T\int_0^T {\mathbf 1}_{\{|\Theta_\chi(g\Phi^s)|>R\}} \,ds
= \frac{\mu(\{h\in\GamG:\: |\Theta_\chi(h)|>R\})}{\mu(\GamG)}.
\end{equation}
By Theorem \ref{tail-asymptotics-for-Theta_chi}, the right hand side is positive for any $R<\infty$. Hence the left hand side guarantees that there is an infinite sequence $s_1<s_2<\ldots\to \infty$ such that $|\Theta_\chi(g\Phi^{s_j})|>R$ for all $j\in\NN$. Since $R$ can be chosen arbitrarily large, $\limsup_{s\to\infty} |\Theta_\chi(g\Phi^s)|=+\infty$.
\end{proof}

Let us now prove that typical realizations of the theta process are not differentiable at $0$.
To this extent, let us introduce the upper derivative of a complex-valued $F:[0,\infty)\to\C$ at $t$ as 
\be D^*F(t)=\limsup_{h\downarrow0} \frac{|F(t+h)-F(t)|}{h}.\ee
We have the following 
\begin{theorem}
Fix $t_0\geq 0$. Then, almost surely, the theta process is not differentiable at $t_0$. Moreover, $D^* X(t_0)=+\infty$.
\end{theorem}
\begin{proof}
Let $Y$ be the theta process constructed by time inversion as in Proposition \ref{prop-time-inversion}. Then, by that proposition 
\begin{align}
D^*Y(0)=\limsup_{h\downarrow 0}\frac{|Y(h)-Y(0)|}{h}\geq\limsup_{t\uparrow\infty}\sqrt{t}\,|Y(\tfrac{1}{t})|=\limsup_{t\uparrow\infty}\frac{|X(t)|}{\sqrt t}
\end{align}
and the latter $\limsup$ is infinite by  Proposition \ref{limsup-growth}. Now let $t_0>0$ be arbitrary, Then $\tilde X(t)=X(t_0+t)-X(t_0)$ defines a  theta process by Theorem \ref{thm-stationarity} and differentiability of $X$ at $0$ is equivalent to differentiability of $\tilde X$ at $t_0$.
\end{proof}

\bibliographystyle{plain}
\bibliography{curlicue-bibliography}
\end{document}